\documentclass[11pt]{article}

\usepackage{geometry}
\usepackage{latexsym}
\usepackage{mathrsfs}
\usepackage{amsfonts}
\usepackage{amssymb}
\usepackage{subfigure}    
\usepackage{graphicx}
\usepackage{epstopdf}
\usepackage{float}
\usepackage{bm}
\usepackage{amsmath}
\usepackage{amsthm}
\usepackage{color}
\usepackage[subnum]{cases}
\usepackage{rotating}
\usepackage{enumitem}
\usepackage{accents}
\usepackage[ruled]{algorithm2e}
\usepackage[title]{appendix}
\usepackage{mathtools}
\usepackage{caption}
\usepackage{url}
\usepackage{longtable}
\allowdisplaybreaks

\usepackage{multirow}
\usepackage{multicol}
\usepackage[usestackEOL]{stackengine}
\usepackage{comment}

\makeatletter
\@addtoreset{equation}{section}
\makeatother

\makeatletter
\def \wideubar{\underaccent{{\cc@style\underline{\mskip15mu}}}}
\def \widebar{\accentset{{\cc@style\underline{\mskip10mu}}}}
\makeatother

\definecolor{blue}{rgb}{0,0,0.9}
\definecolor{red}{rgb}{0.9,0,0}
\definecolor{green}{rgb}{0,0.9,0}
\definecolor{brown}{rgb}{0.6,0.1,0.1}
\definecolor{lightgreen}{rgb}{0.1,0.5,0.1}
\newcommand{\blue}[1]{#1}

\newcommand{\red}[1]{\begin{color}{red}#1\end{color}}

\usepackage[colorlinks=true,
breaklinks=true,
bookmarks=true,
urlcolor=blue,
citecolor=lightgreen,
linkcolor=lightgreen,
bookmarksopen=false,
draft=false]{hyperref}

\graphicspath{{figures/}}

\newcommand{\abs}[1]{\left\lvert #1 \right\rvert}
\newcommand{\argmin}[1]{\underset{#1}{\arg\min}\,}

\newcommand{\Diag}{\mathtt{Diag}}

\newcommand{\inner}[1]{\left\langle #1 \right\rangle }

\newcommand{\norm}[1]{ \left\| #1 \right\|}

\newcommand{\prox}{\mathtt{prox}}

\renewcommand{\Vec}[1]{\mathtt{vec}\left(#1\right)}

\begin{document}

\newtheorem{property}{Property}[section]
\newtheorem{proposition}{Proposition}[section]
\newtheorem{append}{Appendix}[section]
\newtheorem{definition}{Definition}[section]
\newtheorem{lemma}{Lemma}[section]
\newtheorem{corollary}{Corollary}[section]
\newtheorem{theorem}{Theorem}[section]
\newtheorem{remark}{Remark}[section]
\newtheorem{problem}{Problem}[section]
\newtheorem{example}{Example}
\newtheorem{assumption}{Assumption}
\renewcommand*{\theassumption}{\Alph{assumption}}

\title{A Corrected Inexact Proximal Augmented Lagrangian Method with a Relative Error Criterion for a Class of Group-quadratic Regularized Optimal Transport Problems}
\author{
Lei Yang\thanks{School of Computer Science and Engineering, and Guangdong Province Key Laboratory of Computational Science, Sun Yat-Sen University ({\tt yanglei39@mail.sysu.edu.cn}). The research of this author is supported in part by the National Natural Science Foundation of China under grant 12301411, and the Natural Science Foundation of Guangdong under grant 2023A1515012026.},~~
Ling Liang\thanks{(Corresponding author) Department of Mathematics, University of Maryland at College Park ({\tt liang.ling@u.nus.edu}).},~~
Hong T.M. Chu\thanks{Department of Mathematics, National University of Singapore ({\tt hongtmchu@u.nus.edu}).},~~
Kim-Chuan Toh\thanks{Department of Mathematics, and Institute of Operations Research and Analytics, National University of Singapore, Singapore 119076 ({\tt mattohkc@nus.edu.sg}).}
}
\date{}
\maketitle

\begin{abstract}
The optimal transport (OT) problem and its related problems have attracted significant attention and have been extensively studied in various applications. In this paper, we focus on a class of group-quadratic regularized OT problems which aim to find solutions with specialized structures that are advantageous in practical scenarios. To solve this class of problems, we propose a corrected inexact proximal augmented Lagrangian method (ciPALM), with the subproblems being solved by the semi-smooth Newton ({\sc Ssn}) method. We establish that the proposed method exhibits appealing convergence properties under mild conditions. Moreover, our ciPALM distinguishes itself from the recently developed semismooth Newton-based inexact proximal augmented Lagrangian ({\sc Snipal}) method for linear programming. Specifically, {\sc Snipal} uses an absolute error criterion for the approximate minimization of the subproblem for which a summable sequence of tolerance parameters needs to be pre-specified for practical implementations. In contrast, our ciPALM adopts a relative error criterion with a \textit{single} tolerance parameter, which would be more friendly to tune from computational and implementation perspectives. These favorable properties position our ciPALM as a promising candidate for tackling large-scale problems. Various numerical studies validate the effectiveness of employing a relative error criterion for the inexact proximal augmented Lagrangian method, and also demonstrate that our ciPALM is competitive for solving large-scale group-quadratic regularized OT problems. 

\vspace{5mm}


\end{abstract}

\section{Introduction}

Optimal transport (OT), which provides an effective computational tool to compare two probability distributions, has gained increasing attention in a wide range of application areas such as computer vision \cite{rtg2000earth}, data analytics \cite{courty2014domain,courty2016optimal}, and machine learning \cite{arjovsky2017wasserstein,blondel2018smooth}. In contrast to other popular information divergences (e.g., Euclidean, Kullback-Leibler, Bregman) which typically perform a direct pointwise comparison of two distributions, OT aims to quantify the minimal effort of transferring one probability distribution to another by solving an optimization problem with a properly specified cost function. Mathematically, given two weight vectors $\bm{\alpha}:=(\alpha_1,\cdots,\alpha_m)\in\mathbb{R}^m_+$, $\bm{\beta}:=(\beta_1,\cdots,\beta_n)\in\mathbb{R}^n_+$, and two sets of support points $\big\{\bm{p}_1,\cdots,\bm{p}_m\big\}\subset\mathbb{R}^d$, $\big\{\bm{q}_1,\cdots,\bm{q}_n\big\}\subset\mathbb{R}^{{d}}$, we consider two discrete distributions $\bm{\mu}=\sum_{i=1}^{m}\alpha_i\delta_{\bm{p}_i}$ and $\bm{\nu}=\sum_{j=1}^{n}\beta_j\delta_{\bm{q}_j}$, where $\delta_{\bm{p}_i}$ (resp. $\delta_{\bm{q}_j}$) denotes the Dirac function at the point $\bm{p}_i$ (resp. $\bm{q}_j$). The discrete OT problem is then given as follows:
\begin{equation}\label{eq-DOT}
\min_{X\in \mathbb{R}^{m\times n}}\quad \inner{C,X} \quad \mathrm{s.t.} \quad X\bm{1}_n = \bm{\alpha}, ~X^{\top}\bm{1}_m = \bm{\beta}, ~X\geq 0,
\end{equation}
where $C\in\mathbb{R}^{m \times n}$ is a given cost matrix and $\bm{1}_n$ (resp. $\bm{1}_m$) denotes the vector of all ones in $\mathbb{R}^n$ (resp. $\mathbb{R}^m$). Problem \eqref{eq-DOT} was originally formulated by Kantorovich \cite{kantorovich1942translocation} via relaxing the Monge OT problem \cite{monge1781memoire} and is now well-known as the Monge-Kantorovich OT problem; we refer readers to \cite{v2008optimal} for a historical review. In the particular case when $C_{ij} = \norm{\bm{p}_i - \bm{q}_j}^p$ with $p \geq 1$ for $i = 1,\cdots, m$ and $j = 1, \cdots, n$, the value $(\mathcal{W}_p(\bm{\mu}, \bm{\nu}))^{1/p}$ defines the famous $p$-Wasserstein distance between $\bm{\mu}$ and $\bm{\nu}$, where $\mathcal{W}_p(\bm{\mu}, \bm{\nu})$ denotes the optimal objective function value of problem \eqref{eq-DOT}; see \cite[Chapter 6]{v2008optimal} for more details.
Since OT can capture the underlying geometry structures via constructing the cost matrix $C$ in \eqref{eq-DOT}, it usually provides a more robust comparison tool for the probability distributions. This underlies many recent practical successes of OT and its various generalizations such as the Wasserstein distributionally robust optimization problem \cite{kuhn2019wasserstein}.

Following the wave of research on OT, in this paper, we consider a class of group-quadratic regularized OT problems that can be formulated as follows:
\begin{equation}\label{eq-regOTpro}
\min_{X\in\mathbb{R}^{m\times n}} \inner{C,X} + \mathcal{R}(X) \quad \mathrm{s.t.} \quad X \in \mathscr{T}.
\end{equation}
Here, $\mathcal{R}:\mathbb{R}^{m\times n} \rightarrow \mathbb{R}$ is a proper closed convex regularizer taking the following form:
\begin{equation}\label{eq:R}
\mathcal{R}(X) \coloneqq \lambda_1 \sum_{G\in\mathcal{G}} \omega_G \norm{\bm{x}_G} + \frac{\lambda_2}{2}\norm{X}_F^2,
\end{equation}
where $\lambda_1, \lambda_2\geq0$ are regularization parameters, $\mathcal{G}$ is a partition of the index set $\{1,\dots, m\}\times \{1,\dots, n\}$ satisfying that $G\neq \emptyset$ for any $G\in \mathcal{G}$, $G\cap G'=\emptyset$ for any $G, \,G'\in\mathcal{G}$, and $\cup_{G\in\mathcal{G}}\, G=\{1,\dots, m\}\times \{1,\dots, n\}$, $\omega_G \geq 0$ is a weight scalar associated with the group $G$, $\bm{x}_{G}\in\mathbb{R}^{|G|}$ denotes the vector formed from a matrix $X$ by picking the entries indexed by $G$, and $\norm{\bm{x}_G}$ and $\norm{X}_F$ denote the Euclidean norm of $\bm{x}_G$ and the Frobenius norm of $X$, respectively. Moreover, $\mathscr{T}\subseteq\mathbb{R}^{m\times n}$ is a nonempty convex feasible set taking the following form:
\begin{equation}\label{eq-cT}
\mathscr{T}:= \left\{X\in\mathbb{R}^{m\times n}\,:\, AXB = S, ~\bm{\alpha} - X\bm{1}_n\in \mathcal{K}_r, ~\bm{\beta} - X^{\top}\bm{1}_m \in \mathcal{K}_c, ~ X\geq 0 \right\},
\end{equation}
where $A\in \mathbb{R}^{\widetilde m\times m}$, $B\in \mathbb{R}^{n\times \widetilde n}$ and $S\in \mathbb{R}^{\widetilde m\times \widetilde n}$ are given matrices, and $\mathcal{K}_r\subseteq \mathbb{R}^m$ and $ \mathcal{K}_c\subseteq \mathbb{R}^n$ are two convex cones which are typically chosen as the zero spaces or the nonnegative orthants. One can verify that the following constraint sets usually used in the literature readily fall into the form of \eqref{eq-cT} with proper choices of $A$, $B$, $S$, $\mathcal{K}_r$, and $\mathcal{K}_c$:
\begin{itemize}
\item[{\tt [T1]}] The classical OT constraint set $\mathscr{T}:=\big\{X\in\mathbb{R}^{m\times n}: X\bm{1}_n=\bm{\alpha}, ~X^{\top}\bm{1}_m=\bm{\beta}, ~X\geq 0\big\}$;

\item[{\tt [T2]}] The partial OT constraint set $\mathscr{T}:=\big\{X\in\mathbb{R}^{m\times n}: \bm{1}_m^\top X \bm{1}_n = s,  ~ X\bm{1}_n \leq \bm{\alpha}, ~X^{\top}\bm{1}_m\leq\bm{\beta}, ~X\geq0\big\}$ provided that $0 < s \leq \min\big\{\sum_{i=1}^m {\alpha}_i, \,\sum_{j=1}^n{\beta}_j\big\}$;

\item[{\tt [T3]}] The martingale OT constraint set $\mathscr{T}:=\big\{X\in\mathbb{R}^{m\times n}: XQ=\Diag(\bm{\alpha})P, ~ X\bm{1}_n=\bm{\alpha}, ~X^{\top}\bm{1}_m=\bm{\beta}, ~X\geq0\big\}$, where $P:=[\bm{p}_1,\cdots,\bm{p}_m]^{\top}\in\mathbb{R}^{m\times d}$, $Q:=[\bm{q}_1,\cdots,\bm{q}_n]^{\top}\in\mathbb{R}^{n\times d}$, and $\Diag(\bm{\alpha})$ denotes the diagonal matrix whose $i$th diagonal entry is $\alpha_i$.
\end{itemize}

Problem \eqref{eq-regOTpro} covers the Monge-Kantorovich OT problem \eqref{eq-DOT} and its several popular variants in the literature. First, when $\lambda_1=\lambda_2=0$ (namely, the unregularized case), problem \eqref{eq-regOTpro} has been studied in \cite{alfonsi2019sampling,beiglbock2016problem,benamou2015iterative,caffarelli2010free,figalli2010optimal,guo2019computational,de2018entropic} under different mass transport constraints. It is known that the classical OT constraint set \texttt{[T1]}
enforces that the amount of mass $\alpha_i$ at location $\bm{p}_i$ in the source distribution is \textit{fully} assigned and location $\bm{q}_j$ in the target distribution collects exactly the amount of mass $\beta_j$.
However, one significant limitation of this constraint set is that it imposes a mass conservation requirement, necessitating that the source distribution $\bm{\mu}$ and the target distribution $\bm{\nu}$ must have identical total mass, which may not be achievable in real-world scenarios. To relax such a requirement and to avoid the normalization which might amplify some artifacts,
the partial OT constraint set  \texttt{[T2]}
can be employed; see, for example, \cite{benamou2015iterative,caffarelli2010free,figalli2010optimal}. Compared with \texttt{[T1]}, \texttt{[T2]} allows that only a fraction of mass would be transported to the target distribution, and hence is more flexible to fit different practical circumstances to achieve better empirical performances. Moreover, one may also impose other constraints on the transportation plan to tailor the resulting model for specific applications. For instance, the martingale OT problem, as an important variant of the Monge-Kantorovich OT problem \eqref{eq-DOT}, has been studied recently as the dual problem of the robust superhedging of exotic options in mathematical finance; see, for example, \cite{alfonsi2019sampling,beiglbock2016problem,guo2019computational,de2018entropic}.  It additionally assumes that random variables $\mathscr{X}$ and $\mathscr{Y}$ associated with probability distributions $\bm{\mu}$ and $\bm{\nu}$ form a martingale sequence satisfying $\mathbb{E}[\mathscr{Y}| \mathscr{X} ] = \mathscr{X}$. In the discrete setting, this condition can be reformulated as $\sum_{j=1}^n X_{ij}\bm{q}_j = {\alpha}_i\bm{p}_i$ for all $i=1,\dots,m$, as in the constraint set \texttt{[T3]}; we refer readers to
\cite[Chaper 4]{durrett2019probability} for more details on martingales.

The rationale that underlines the relevance and usefulness of introducing a nontrivial regularizer $\mathcal{R}$ in \eqref{eq-regOTpro} stems from both the algorithmic aspect and the modeling aspect. Indeed, a proper choice of $\mathcal{R}$ may lead to a computationally more tractable regularized problem. A representative example is the entropy regularizer $\mathcal{R}(X):=\lambda\mathfrak{e}(X)$ with $\mathfrak{e}(X):=\lambda\sum_{i=1}^{m}\sum_{j=1}^{n}X_{ij}(\log X_{ij}-1)$ and $\lambda>0$. Here the resulting entropic regularized problem can be efficiently solved by, for example, Sinkhorn's algorithm or more generally the Bregman iterative projection algorithm for \texttt{[T1]} \cite{benamou2015iterative,cuturi2013sinkhorn} or for \texttt{[T2]} \cite{de2018entropic}, Newton's method for \texttt{[T1]} \cite{brauer2017sinkhorn}, and the Dykstra's algorithm with Kullback-Leibler projections for  \texttt{[T3]} \cite{benamou2015iterative}, in order to obtain an approximate solution within a favorable computational complexity (see also, e.g., \cite{awr2017near,d2018computational,lin2022efficiency}). Meanwhile, many other convex regularizers have also been shown to admit such computational advantages \cite{dessein2018regularized,d2018computational,essid2018quadratically,lorenz2021quadratically}.
The underlying idea is that a proper regularizer $\mathcal{R}$ can define a strongly convex problem \eqref{eq-regOTpro} so that the corresponding dual problem admits a smooth objective possibly with some simple and well-structured constraints. Hence, the regularized problem can be readily solved by many well-developed algorithms. In addition, a convex regularizer can help to induce a solution with desired structures to fit different applications, and hence improve the effectiveness and robustness of the model in practice.
For example, the entropy regularizer encourages a smooth solution with a spread support \cite{benamou2015iterative,cuturi2013sinkhorn,cp2016a}; the quadratic regularizer can maintain the sparsity of the solution \cite{blondel2018smooth,essid2018quadratically,lorenz2021quadratically}; a special variation regularizer helps to remove colorization artifacts \cite{fppa2014regularized}; the group regularizer enables one
to incorporate the label information \cite{courty2014domain,courty2016optimal}; the Laplacian regularizer can encode the neighborhood similarity between samples \cite{flamary2014optimal}. The aforementioned potential advantages of regularization motivate the study of various regularized OT problems.

\blue{In this paper, we are particularly interested in the group-quadratic regularizer $\mathcal{R}$ given as \eqref{eq:R} and consider problem \eqref{eq-regOTpro}. As outlined above, problem \eqref{eq-regOTpro} encompasses the classical OT problem along with several significant variants, including the partial/martingale OT problem, the quadratic regularized OT problem, the group regularized OT problem, and others. All these models have been studied in the literature and have shown considerable potential in a range of applications such as image retrieval \cite{rtg2000earth}, domain adaptation \cite{courty2014domain,courty2016optimal}, color transfer \cite{blondel2018smooth,bonneel2019spot}, human activity recognition \cite{lu2021cross}, object and face recognition \cite{montesuma2021wasserstein}, finance and economics \cite{bhp2013model,guo2019computational}, and so on.}
Moreover, compared with \cite{courty2014domain,courty2016optimal} where the entropy regularizer is used together with the group-sparsity regularizer (thereby leading to completely dense solutions), the regularizer in \eqref{eq:R} can take into account prior group structures while still promoting sparsity of $X$. On the other hand, compared with \cite{blondel2018smooth} which also considered \eqref{eq:R}, the quadratic term in our paper is optional (namely, $\lambda_2$ can be set to 0), and by using the notation $\bm{x}_G$ as in \eqref{eq:R}, elements in a group can also be arbitrarily selected from $X$. Moreover, existing solution methods used in \cite{blondel2018smooth,courty2014domain,courty2016optimal,essid2018quadratically,lorenz2021quadratically} \textit{fully} rely on the strong convexity of the objective and hence cannot be easily extended to the case of solely using the group-sparsity regularizer (namely, \eqref{eq:R} with $\lambda_2=0$).

When it comes to the solution methods for solving problem \eqref{eq-regOTpro}, to the best of our knowledge, most existing works only focused on the classical OT constraint set \texttt{[T1]} together with the quadratic regularizer
or group-quadratic regularizer, and proposed to use the accelerated gradient descent (APG) method \cite{d2018computational} or Newton-type methods \cite{blondel2018smooth,essid2018quadratically,li2020efficient,lorenz2021quadratically} for solving a certain dual problem.
However, APG would suffer from the slow convergence speed when the regularization parameter is small, and Newton-type methods should require a certain nondegeneracy condition to guarantee a fast convergence rate, which is uncheckable and may not be satisfied in practice. Note that, for the unregularized case, problem \eqref{eq-regOTpro} under the constraint set $\mathcal{T}$ is essentially a linear programming (LP) problem.
However, the problem size can be huge when the dimension of the distribution ($m$ or $n$) is large. Thus, classical LP methods such as the simplex and interior point methods are no longer efficient enough or consume excessive computational resources when solving such large-scale LP problems.
This could limit the potential applicability of OT and its various generalizations. Note also that in such an LP problem, the number of variables is typically much larger than the number of linear constraints. To efficiently solve this kind of LP problems, Li, Sun, and Toh \cite{li2020asymptotically} recently proposed to apply a semismooth Newton-based inexact proximal augmented Lagrangian ({\sc Snipal}) method.
The proposed {\sc Snipal} is shown to have a much better performance in comparison to current state-of-the-art LP solvers. But, to guarantee the global convergence and the asymptotic superlinear convergence rate of the proposed algorithm, the {\sc Snipal} subproblems have to be solved approximately under an \textit{absolute} error criterion for which a summable sequence of error tolerances must be pre-specified.
Consequently, one generally needs to perform hyperparameter tuning of the sequence to achieve superior convergence performances. This might be less friendly to users in practice. We refer readers to Section \ref{sec:PALM} for more detailed discussions. This also motivates us to seek a possibly simpler inexact error criterion for the augmented Lagrangian subproblems so that the appealing convergence properties can be preserved in both theoretical and numerical aspects, and meanwhile, the task of hyperparameter tunings can also be simplified.

In view of the above, in this paper, we attempt to develop a unified algorithmic framework for efficiently solving problem \eqref{eq-regOTpro} with $\mathcal{R}$ chosen as \eqref{eq:R} and $\mathscr{T}$ chosen as \eqref{eq-cT}, \blue{aiming to achieve a reasonable level of accuracy with less computational resources.} To this end, we first rewrite the problem in a unified manner and derive its dual problem in Section \ref{sec:PALM}. We then apply a corrected inexact proximal augmented Lagrangian method (ciPALM) in Algorithm \ref{algo:ciPALM} to solve the resulting dual problem and show that our ciPALM is in fact an application of a variable metric hybrid proximal extragradient (VHPE) method in Algorithm \ref{algo:VHPE}. Hence, the convergence properties of the ciPALM can be obtained as a direct application of the general theory for the VHPE as presented in Section \ref{sec:VHPE}. Further, in Section \ref{sec:SSN}, we apply a semismooth Newton method ({\sc Ssn}), which is
a second-order method that has a fast superlinear (or even quadratic) convergence rate, to solve the ciPALM subproblem \eqref{ciPALM-subpro}. We emphasize that the second-order sparsity structure of the problem is fully uncovered and exploited to significantly reduce the computational cost of solving the semismooth Newton systems. Various numerical experiments conducted in Section \ref{sec:num} demonstrate that the proposed ciPALM with {\sc Ssn} as a subsolver is efficient for solving problem \eqref{eq-regOTpro} with different choices of $\mathcal{R}$ and $\mathscr{T}$. Note that our ciPALM shares a similar algorithmic framework as the {\sc Snipal} in \cite{li2020asymptotically}.
However, we should point out that the {\sc Snipal} is specifically developed for solving the linear programming problems, while our ciPALM is tailored to problem \eqref{eq-regOTpro}, involving an additional group-quadratic regularizer \eqref{eq:R}.
Moreover, we have also made an essential change to the algorithm by introducing a more practical relative error criterion \eqref{inexcond-ciALM} for solving the subproblem \eqref{ciPALM-subpro} which requires an extra correction step in \eqref{extrastep-ciALM} to guarantee the convergence. It turns out that our ciPALM has shown comparable theoretical properties and numerical performances as {\sc Snipal} but only has a single tolerance parameter $\rho\in[0,1)$ in the error criterion \eqref{inexcond-ciALM}. Hence the corresponding parameter tuning is typically easier than that in the {\sc Snipal} from the computation and implementation perspectives, as shown in Section \ref{section-classic-ot} where we investigate the effects of different inexactness conditions.


The rest of this paper is organized as follows. We introduce the VHPE and present its convergence results in Section \ref{sec:VHPE}. We then develop the ciPALM for solving problem \eqref{eq-regOTpro} in Section \ref{sec:PALM}. Moreover, we derive its connection to the VHPE for obtaining the convergence properties for the ciPALM. Section \ref{sec:SSN} is devoted to applying the {\sc Ssn} for solving the ciPALM subproblem.
In Section \ref{sec:num}, we evaluate the numerical performance of our algorithm by solving various large-scale (un)regularized OT problems. Finally, we conclude the paper in Section \ref{sec:conc}.

\paragraph{Notation and preliminaries.}
We use $\mathbb{R}^n$, $\mathbb{R}^n_+$, $\mathbb{R}^{m\times n}$ and $\mathbb{R}^{m\times n}_+$ to denote the set of $n$-dimensional real vectors, $n$-dimensional nonnegative vectors, $m\times n$ real matrices, and $m\times n$ nonnegative matrices, respectively. We also denote $\overline{\mathbb{R}}:= \mathbb{R}\cup \{\pm \infty \}$ as the extended reals. For a vector $\bm{x}\in\mathbb{R}^n$, $x_i$ denotes its $i$-th entry, $\|\bm{x}\|$ denotes its Euclidean norm, and $\|\bm{x}\|_M:=\sqrt{\langle\bm{x},\, M \bm{x}\rangle}$ denotes its weighted norm associated with the symmetric positive semidefinite matrix $M$. For any $X\in\mathbb{R}^{m_1\times n_1}$ and $Y \in\mathbb{R}^{m_2\times n_2}$, $(X;Y)\in \mathbb{R}^{m_1\times n_1} \times \mathbb{R}^{m_2\times n_2}$ denotes the matrix obtained by vertically concatenating $X$ and $Y$. For a matrix $X\in\mathbb{R}^{m\times n}$, ${X_{ij}}$ denotes its $(i,j)$-th entry, and $\Vec{X}$ denotes the vectorization of $X$, where $[\Vec{X}]_{i+(j-1)m}=X_{ij}$ for any $1\leq i\leq m$ and $1\leq j \leq n$. For an index set $G\subseteq\{1,\cdots\!,m\} \times \{1,\cdots\!,n\}$ whose elements are arranged in the lexicographical order, let $|G|$ denote its cardinality and $G^c$ denote its complementarity set. We denote by $\bm{x}_{G}\in\mathbb{R}^{|G|}$ the vector formed from a matrix $X\in\mathbb{R}^{m\times n}$ by picking the entries indexed by $G$. The identity matrix of size $n \times n$ is denoted by $I_n$. We also use $\bm{x}\geq0$ and $X \geq 0$ to denote $x_i\geq0$ for all $i$ and $X_{ij}\geq 0$ for all $(i,j)$. Let $\mathcal{S}$ be a closed convex subset of $\mathbb{R}^n$. We write the weighted distance of $\bm{x}\in\mathbb{R}^n$ to $\mathcal{S}$ by $\mathrm{dist}_{M}(\bm{x},\mathcal{S}):=
\inf_{\bm{y}\in\mathcal{S}}\|\bm{x}-\bm{y}\|_M$. When $M$ is the identity matrix, we omit $M$ in the notation and simply use $\mathrm{dist}(\bm{x},\mathcal{S})$ to denote the Euclidean distance of $\bm{x}\in\mathbb{R}^n$ to $\mathcal{S}$. Moreover, we use $\Pi_{\mathcal{S}}(\bm{x})$ to denote the projection of $\bm{x}$ onto $\mathcal{S}$.

For an extended-real-valued function $f: \mathbb{R}^{n} \rightarrow [-\infty,\infty]$, we say that it is \textit{proper} if $f(\bm{x}) > -\infty$ for all $\bm{x}\in\mathbb{R}^{n}$ and its domain ${\rm dom}\,f:=\{\bm{x} \in \mathbb{R}^{n} : f(\bm{x})<\infty\}$ is nonempty. A proper function $f$ is said to be closed if it is lower semicontinuous. Assume that $f: \mathbb{R}^{n} \rightarrow (-\infty,\infty]$ is a proper and closed convex function, the subdifferential of $f$ at $\bm{x}\in{\rm dom}\,f$ is defined by $\partial f(\bm{x}):=\big\{\bm{d}\in\mathbb{R}^n: f(\bm{y}) \geq f(\bm{x}) + \langle \bm{d}, \,\bm{y}-\bm{x}\rangle, ~\forall\,\bm{y}\in\mathbb{R}^n\big\}$ and its conjugate function $f^*: \mathbb{R}^{n} \rightarrow (-\infty,\infty]$ is defined by $f^*(\bm{y}):=\sup\big\{\langle \bm{y},\,\bm{x}\rangle-f(\bm{x}) : \bm{x}\in\mathbb{R}^n\big\}$. For any $\nu>0$, the Moreau envelope of $\nu f$ at $\bm{x}$ is defined by $\mathtt{M}_{\nu f}(\bm{x}) := \min_{\bm{y}} \big\{f(\bm{y}) + \frac{1}{2\nu}\|\bm{y} - \bm{x}\|^2\big\}$,
and the proximal mapping of $\nu f$ at $\bm{x}$ is defined by $\mathtt{prox}_{\nu f}(\bm{x}) := \arg\min_{\bm{y}} \big\{f(\bm{y}) + \frac{1}{2\nu}\|\bm{y} - \bm{x}\|^2\big\}$.
For a given real symmetric matrix $M$, $\lambda_{\max}(M)$ and $\lambda_{\min}(M)$ denote its largest and smallest eigenvalues, respectively.

\blue{Let $\mathbb{X}$ and $\mathbb{Y}$ be two finite dimensional Euclidean spaces, we call a multivalued function $\mathcal{F}:\mathbb{X}\rightrightarrows \mathbb{Y}$ to be a multifunction. If for any $x\in \mathbb{X}$, the set $\mathcal{F}(x)\subset \mathbb{Y}$ is a polyhedral set, then we say that $\mathcal{F}$ is a polyhedral multifunction.}

\section{A variable metric hybrid proximal extragradient method}\label{sec:VHPE}

In this section, we present a variable metric hybrid proximal extragradient (VHPE) method and study its convergence properties, which will pave the way to establish the convergence of the method for solving problem \eqref{eq-regOTpro} developed in the next section. The VHPE is indeed a special case of a general hybrid inexact variable metric proximal point algorithm developed by Parente, Lotito, and Solodov \cite{pls2008class}, and can be viewed as an extension of the well-recognized hybrid proximal extragradient (HPE) method developed by Solodov and Svaiter \cite{ss1999hybridap,ss1999hybridpr}. Let $\mathcal{T}:\mathbb{R}^\ell\to\mathbb{R}^\ell$ be a maximal monotone operator. The VHPE for solving the monotone inclusion problem $0\in\mathcal{T}(\bm{x})$ is presented as Algorithm \ref{algo:VHPE}.

\begin{algorithm}[htb!]
\caption{A variable metric hybrid proximal extragradient (VHPE) method}\label{algo:VHPE}
	
\textbf{Initialization:} Choose $0\leq\rho<1$, $\bm{x}^0\in\mathbb{R}^{\ell}$, and choose two sequences $\{c_k\}\subseteq\mathbb{R}$ and $\{M_k\}\subseteq\mathbb{R}^{\ell\times\ell}$.
	Set $k=0$.

\While{a termination criterion is not met,}{

\textbf{Step 1.} Approximately solve
\begin{equation}\label{VHPE-subpro}
0\in c_kM_k\mathcal{T}(\bm{x}) + (\bm{x}-\bm{x}^k)
\end{equation}
to find a triplet $(\widetilde{\bm{x}}^{k+1}, \bm{d}^{k+1}, \varepsilon_{k+1}) \in\mathbb{R}^{\ell}\times\mathbb{R}^{\ell}\times\mathbb{R}_+$ such that
\begin{equation}\label{VHPE-inexcond}
\left\{\,\begin{aligned}
&\bm{d}^{k+1} \in \mathcal{T}^{\varepsilon_{k+1}}(\widetilde{\bm{x}}^{k+1}), \\[3pt]
&\|c_kM_k\bm{d}^{k+1} + \widetilde{\bm{x}}^{k+1} - \bm{x}^k\|^2_{M_k^{-1}} + 2c_k\varepsilon_{k+1} \leq \rho^2\|\widetilde{\bm{x}}^{k+1} - \bm{x}^k\|^2_{M_k^{-1}}.
\end{aligned}\right.
\end{equation}

\textbf{Step 2.} Update $\bm{x}^{k+1}=\bm{x}^k - c_kM_k\bm{d}^{k+1}$.
		
\textbf{Step 3.} Set $k=k+1$ and go to \textbf{Step 1}.

}	
\end{algorithm}

In the following, we study the convergence properties of the VHPE in Algorithm \ref{algo:VHPE}. To this end, we first make the following assumptions.

\begin{assumption}\label{assmp-cM}
The sequences $\{c_k\}\subseteq\mathbb{R}$ and $\{M_k\}\subseteq\mathbb{R}^{\ell\times\ell}$ satisfy the following conditions.
\begin{itemize}
\item[{\rm (i)}] $\{c_k\}\subseteq\mathbb{R}$ is a sequence of positive
numbers and is bounded away from zero, i.e., there exists a constant $c>0$ such that $c_k\geq c$ for all $k\geq0$.
\item[{\rm (ii)}] $\{M_k\}\subseteq\mathbb{R}^{\ell\times\ell}$ is a sequence of symmetric positive definite matrices satisfying $\frac{1}{1+\eta_k}M_k\preceq M_{k+1}$ and $\underline{\lambda}\leq\lambda_{\min}(M_k)\leq\lambda_{\max}(M_k)\leq\overline{\lambda}$ for all $k\geq0$, with some nonnegative summable sequence $\{\eta_k\}$ and constants $0<\underline{\lambda}<\overline{\lambda}$.
\end{itemize}
\end{assumption}

We then present the global convergence of the VHPE in the next theorem. Here, we should point out that the following results (i), (iii), and (iv) can be obtained by directly applying
\cite[Proposition 3.1, Proposition 4.1 and Theorem 4.2]{pls2008class} since the VHPE in Algorithm \ref{algo:VHPE} falls into the general algorithmic framework in \cite{pls2008class}. For the self-contained purpose, we provide a more succinct proof in the appendix.

\begin{theorem}\label{thm-VHPEcon}
Suppose that $\Omega:=\mathcal{T}^{-1}(0)\neq\emptyset$ and Assumption \ref{assmp-cM} holds. Let $\{\bm{x}^k\}$ be the sequence generated by the VHPE in Algorithm \ref{algo:VHPE}. Then, the following statements hold.
\begin{itemize}
\item[{\rm (i)}] The sequence $\{\bm{x}^k\}$ is bounded.
\item[{\rm (ii)}] For any $k\geq0$, we have
                  \begin{equation}\label{ineqzomega}
                  \mathrm{dist}_{M_{k+1}^{-1}}(\bm{x}^{k+1},\Omega) \leq (1+\eta_k)\,\mathrm{dist}_{M_{k}^{-1}}(\bm{x}^{k},\Omega).
                  \end{equation}
\item[{\rm (iii)}] $\lim\limits_{k\to\infty}\|\widetilde{\bm{x}}^{k+1}-\bm{x}^{k}\|
    = \lim\limits_{k\to\infty}\|\bm{d}^{k}\|
    = \lim\limits_{k\to\infty}\varepsilon_{k} = 0$.
\item[{\rm (iv)}] The sequence $\{\bm{x}^k\}$ converges to a point $\bm{x}^{\infty}$ such that $0\in\mathcal{T}(\bm{x}^{\infty})$.
\end{itemize}
\end{theorem}
\begin{proof}
See Appendix \ref{sec-appendix-vhpecon}.
\end{proof}

We next study the convergence rate of the VHPE under the following error-bound assumption. Note from \cite[Lemma 2.4]{li2020asymptotically} that this error bound condition is weaker than the local upper Lipschitz continuity of $\mathcal{T}^{-1}$ at the origin used in \cite{pls2008class} and has been employed in \cite{li2020asymptotically} for establishing the asymptotic Q-superlinear convergence rate of a preconditioned proximal point algorithm with absolute error criteria.

\begin{assumption}\label{assmp-errbdweak}
For any $r>0$, there exist a $\kappa>0$ such that
\begin{equation}\label{errbdineq}
\mathrm{dist}\big(\bm{x}, \mathcal{T}^{-1}(0)\big)
\leq\kappa\,\mathrm{dist}\big(0, \mathcal{T}(\bm{x})\big), \quad \forall\,\bm{x}\in\big\{\bm{x}\in\mathbb{R}^{\ell} \mid \mathrm{dist}\big(\bm{x}, \mathcal{T}^{-1}(0)\big)\leq r\big\}.
\end{equation}
\end{assumption}


\begin{theorem}\label{thm-VHPEconrate}
Under the same assumptions in Theorem \ref{thm-VHPEcon} and suppose additionally that Assumption \ref{assmp-errbdweak} holds
with $r:=\sqrt{\overline{\lambda}}\, \mathrm{dist}_{M_0^{-1}}(\bm{x}^0,\Omega) \prod^{\infty}_{i=0} (1+\eta_i)$. Let $\{\bm{x}^k\}$ be the sequence generated by the VHPE in Algorithm \ref{algo:VHPE}. Then, for all $k\geq0$, we have
\begin{equation*}
\mathrm{dist}_{M_{k+1}^{-1}}(\bm{x}^{k+1},\Omega) \leq \mu_k\,\mathrm{dist}_{M_{k}^{-1}}(\bm{x}^{k},\Omega),
\end{equation*}
where
\begin{equation}\label{defmuk}
\mu_k := \frac{1+\eta_k}{1-\rho(1-\rho)^{-1}}\left(\rho(1-\rho)^{-1}
+\frac{(1+\rho(1-\rho)^{-1})\kappa}{\sqrt{\kappa^2+\underline{\lambda}^2c_k^2}}\right)<1
\end{equation}
for sufficiently small $\rho$ and sufficiently large $c_k$.
\end{theorem}
\begin{proof}
See Appendix \ref{sec-appendix-vhpecon}.
\end{proof}

\begin{remark}[\textbf{Comments on the coefficient $\mu_k$}]
One can see from the definition of $\mu_k$ in \eqref{defmuk} that $\mu_k$ can be less than 1 whenever $\rho$ is sufficiently small and $c_k$ is sufficiently large. In practical implementations, one can choose a constant $\rho<\frac{1}{3}$ and an increasing sequence of $\{c_k\}$ with $c_k\uparrow\infty$. Recall that $\eta_k\to0$ (since $\{\eta_k\}$ is summable). Note also that $\{\eta_k\}$ is not involved in the error criterion \eqref{VHPE-inexcond}. Then, we have
\begin{equation*}
\lim\limits_{k\to\infty}\,\mu_k = \frac{\rho(1-\rho)^{-1}}{1-\rho(1-\rho)^{-1}}
= \frac{\rho}{1-2\rho} < 1.
\end{equation*}
This implies that the sequence $\big\{\mathrm{dist}_{M_{k}^{-1}}(\bm{x}^{k},\Omega)\big\}$ converges linearly to zero after finitely many iterations.
\end{remark}

\section{A corrected inexact proximal augmented Lagrangian method}\label{sec:PALM}

In this section, we aim to design a unified algorithmic framework to solve the regularized OT problem \eqref{eq-regOTpro} with $\mathcal{R}$ chosen as \eqref{eq:R}, and $\mathcal{T}$ chosen as \eqref{eq-cT}. To this end, we first rewrite the problem in the following unified manner:
\begin{equation}\label{eq:pmain}
\begin{aligned}
&\min_{X\in\mathbb{R}^{m\times n}, \,\bm{y}\in \mathbb{R}^{m}, \,\bm{z}\in \mathbb{R}^n} \inner{C, X} + p(X) + p_r(\bm{y}) + p_c(\bm{z}) \\
&\hspace{1.4cm} \mathrm{s.t.} \hspace{1.4cm}  AXB = S, ~~X\bm{1}_n + \bm{y} = \bm{\alpha}, ~~X^\top \bm{1}_m + \bm{z} = \bm{\beta},
\end{aligned}
\end{equation}
where $p:\mathbb{R}^{m\times n}\to\overline{\mathbb{R}}$, $p_r:\mathbb{R}^m\to \overline{\mathbb{R}}$ and $p_c:\mathbb{R}^n\to\overline{\mathbb{R}}$ are three proper closed convex functions, $\bm{\alpha}\in\mathbb{R}^m$, $ \bm{\beta}\in \mathbb{R}^{n} $, $C\in\mathbb{R}^{m\times n}$, $ A\in \mathbb{R}^{\widetilde m\times m} $, $B\in \mathbb{R}^{n\times \widetilde n}$ and $S\in \mathbb{R}^{\widetilde m\times \widetilde n}$ are given data. It is easy to see that problem \eqref{eq-regOTpro} falls into the form of \eqref{eq:pmain} with
\begin{equation*}
p(X) := \lambda_1 \sum_{G\in \mathcal{G}}\omega_G\norm{\bm{x}_G} + \frac{\lambda_2}{2}\norm{X}_F^2 + \delta_{\mathbb{R}_+^{m\times n}}(X),\quad p_r(\bm{y}):= \delta_{\mathcal{K}_r}(\bm{y}),\quad p_c(\bm{z}):= \delta_{\mathcal{K}_c}(\bm{z}).
\end{equation*}

Let $p^*:\mathbb{R}^{m\times n}\to \overline{\mathbb{R}}$, $p_r^*:\mathbb{R}^m\to \overline{\mathbb{R}}$ and $p_c^*:\mathbb{R}^n\to \overline{\mathbb{R}}$ be the conjugate functions of $p(\cdot)$, $p_r(\cdot)$ and $p_r(\cdot)$, respectively. Then, the dual problem of \eqref{eq:pmain} is equivalently given by
(modulo a minus sign)
\begin{equation}\label{eq:dmain}
\hspace{-2mm}
\min_{W\in \mathbb{R}^{\widetilde m\times \widetilde n}, \bm{u}\in\mathbb{R}^m, \bm{v}\in\mathbb{R}^n}
f(W, \bm{u}, \bm{v}) := \left\{\!\!\begin{array}{l}
-\langle S, \,W\rangle - \inner{\bm{\alpha},\bm{u}} - \inner{\bm{\beta},\bm{v}}
\\[5pt]
+\,  p^*\big(\bm{u}\bm{1}_n^\top + \bm{1}_m\bm{v}^\top + A^\top W B^\top - C\big)
+ p_r^*(\bm{u}) + p_c^*(\bm{v}).
\end{array}\right.\!\!
\end{equation}

Next, we present a corrected inexact proximal augmented Lagrangian method (ciPALM) with a relative error criterion to solve problem \eqref{eq:dmain}.
The algorithmic framework is developed based on the parametric convex duality framework (see, for example, \cite{r1970convex,r1974conjugate} and
\cite[Chapter 11]{rw1998variational}). We first identify problem \eqref{eq:dmain} with the following problem
\begin{equation}\label{para-dmain}
\min\limits_{W\in \mathbb{R}^{\widetilde m\times \widetilde n}, \,\bm{u}\in\mathbb{R}^m, \,\bm{v}\in\mathbb{R}^n} G\big(W, \bm{u}, \bm{v}, 0, 0, 0\big),
\end{equation}
where $G:\mathbb{R}^{\widetilde m\times \widetilde n} \times \mathbb{R}^m\times \mathbb{R}^n \times \mathbb{R}^{m \times n} \times \mathbb{R}^m\times \mathbb{R}^n \to \overline{\mathbb{R}} $ is defined by
\begin{equation*}
\begin{aligned}
\quad G\big(W, \bm{u}, \bm{v}, \Xi, \bm{\zeta}, \bm{\xi}\big)
:= & - \langle S, \,W\rangle -\inner{\bm{\alpha},\bm{u}} - \inner{\bm{\beta},\bm{v}}  \\
& + p^*\big(\bm{u}\bm{1}_n^\top + \bm{1}_m\bm{v}^\top + A^\top W B^\top - C + \Xi\big) + p_r^*(\bm{u} + \bm{\zeta}) + p_c^*\big(\bm{v}+\bm{\xi} \big).
\end{aligned}
\end{equation*}
Note that $G$ is proper closed convex since $p^*$, $p_r^*$ and $p_c^*$ are all proper closed convex. We also define $F:\mathbb{R}^{\widetilde m\times \widetilde n} \times\mathbb{R}^m\times \mathbb{R}^n \times  \mathbb{R}^{m\times n}\times \mathbb{R}^m \times \mathbb{R}^n \to \overline{\mathbb{R}}$ to be the concave conjugate of $G$, that is
\begin{equation*}
F(\widetilde{W}, \widetilde{\bm{u}}, \widetilde{\bm{v}},  X, \bm{y}, \bm{z})
:= \inf\limits_{W, \bm{u}, \bm{v}, \Xi,\bm{\zeta},\bm{\xi}}
\left\{
\begin{array}{c}
G(W, \bm{u}, \bm{v}, \Xi, \bm{\zeta}, \bm{\xi})  - \langle\widetilde{W},W\rangle
- \inner{\widetilde{\bm{u}}, \bm{u}}
- \inner{\widetilde{\bm{v}}, \bm{v}} \\[5pt]
- \inner{X,\Xi} - \langle\bm{y},\bm{\zeta}\rangle - \langle\bm{z},\bm{\xi}\rangle
\end{array}\right\},
\end{equation*}
which is a closed (upper semicontinuous) concave function. Then, the dual problem of problem \eqref{para-dmain} is given by
\begin{equation}\label{para-pmain}
\max\limits_{X\in \mathbb{R}^{m\times n}, \,\bm{y}\in\mathbb{R}^m, \,\bm{z}\in\mathbb{R}^n} F\big(0, 0, 0,X, \bm{y}, \bm{z}\big),
\end{equation}
which can be equivalently rewritten as problem \eqref{eq:pmain}.

The (ordinary) Lagrangian function of problem \eqref{eq:dmain} can be defined by taking the concave conjugate of $G$ with respect to its last three arguments (see \cite[Definition 11.45]{rw1998variational}), that is,
\begin{equation*}
\begin{aligned}
\ell\big(W, \bm{u}, \bm{v}, X, \bm{y}, \bm{z}\big)
:= &\; \inf\limits_{(\Xi, \bm{\zeta},\bm{\xi})\in \mathbb{R}^{m \times n} \times \mathbb{R}^m\times \mathbb{R}^n}\big\{G(W, \bm{u}, \bm{v},  \Xi, \bm{\zeta}, \bm{\xi})
- \inner{X, \Xi} - \langle\bm{y}, \bm{\zeta}\rangle - \langle\bm{z}, \bm{\xi}\rangle\big\} \\
= &\; - \inner{S, \,W} - \inner{\bm{\alpha},\bm{u}} - \inner{\bm{\beta},\bm{v}} - p(X) - p_r( \bm{y}) - p_c(\bm{z})  \\
&\;  + \inner{\bm{u}\bm{1}_n^\top + \bm{1}_m\bm{v}^\top + A^\top WB^\top - C, \,X} + \inner{\bm{u},\bm{y}} + \inner{\bm{v},\bm{z}}.
\end{aligned}
\end{equation*}
Clearly, $\ell$ is convex in its first three arguments and concave in the remaining arguments. Let $\partial\ell$ denote its subgradient map (see
\cite[Page 374]{r1970convex}). If $\big( W^*, \bm{u}^*, \bm{v}^*, X^*,\bm{y}^*, \bm{z}^*\big)$
is such that $0 \in\partial \ell\big(W^*, \bm{u}^*, \bm{v}^*, X^*, \bm{y}^*, \bm{z}^*\big)$, then $\big(W^*, \bm{u}^*, \bm{v}^*\big)$ solves problem \eqref{para-dmain} (i.e., problem \eqref{eq:dmain}) and $\big(X^*, \bm{y}^*, \bm{z}^*\big)$ solves problem \eqref{para-pmain} (i.e., problem \eqref{eq:pmain}). In this case, we say that $\big(W^*, \bm{u}^*, \bm{v}^*, X^*, \bm{y}^*, \bm{z}^*\big)$ is a \textit{saddle point} of the Lagrangian function $\ell\big(W, \bm{u}, \bm{v},  X, \bm{y}, \bm{z}\big)$. If such a saddle point exists, then strong duality holds, that is, $G\big(W^*, \bm{u}^*, \bm{v}^*, 0, 0, 0\big) = F\big(0, 0, 0, X^*, \bm{y}^*, \bm{z}^*\big)$ and thus the optimal values of the primal and dual problems \eqref{para-dmain} and \eqref{para-pmain} exist and coincide.

For a given parameter $\sigma>0$, the augmented Lagrangian function of problem \eqref{eq:dmain} is defined by (see \cite[Example 11.57]{rw1998variational})
\begin{align*}
&\; \mathcal{L}_{\sigma}\big(W, \bm{u}, \bm{v}, X, \bm{y}, \bm{z}\big) \\
:=&\; \sup\limits_{\Xi \in\mathbb{R}^{m\times n}, \bm{\zeta}\in\mathbb{R}^m, \bm{\xi}\in\mathbb{R}^n}\left\{\ell\big(W, \bm{u}, \bm{v}, \Xi, \bm{\zeta}, \bm{\xi}\big) - \frac{1}{2\sigma}\|\Xi - X\|_F^2
- \frac{1}{2\sigma}\|\bm{\zeta}-\bm{y}\|^2
- \frac{1}{2\sigma}\|\bm{\xi}-\bm{z}\|^2\right\} \\
=&\; - \inner{S, \,W} -\inner{\bm{\alpha},\bm{u}} - \inner{\bm{\beta},\bm{v}} - \frac{1}{2\sigma}\norm{X}_F^2 -  \frac{1}{2\sigma}\norm{\bm{y}}^2 -  \frac{1}{2\sigma}\norm{\bm{z}}^2\\
&\; - \mathtt{M}_{\sigma p}\big(X + \sigma(\bm{u}\bm{1}_n^\top + \bm{1}_m\bm{v}^\top + A^\top W B^\top - C)\big) + \frac{1}{2\sigma}\norm{X + \sigma(\bm{u}\bm{1}_n^\top + \bm{1}_m\bm{v}^\top + A^\top W B^\top - C)}_F^2   \\
&\; - \mathtt{M}_{\sigma p_r}\big(\bm{y} + \sigma \bm{u}\big)
+ \frac{1}{2\sigma}\norm{\bm{y} + \sigma \bm{u}}^2 - \mathtt{M}_{\sigma p_c}\big(\bm{z} + \sigma \bm{v}\big)
+ \frac{1}{2\sigma}\norm{\bm{z} + \sigma \bm{v}}^2.
\end{align*}
From the property of the Moreau envelope (see
\cite[Proposition 12.29]{bc2011convex}), we know that $\mathcal{L}_{\sigma}$ is continuously differentiable with respect to its first three arguments. In particular, given $(X,\,\bm{y},\,\bm{z}) \in \mathbb{R}^{m\times n}\times \mathbb{R}^m\times \mathbb{R}^n$, let
\begin{equation*}
\begin{aligned}
X_\sigma(W, \bm{u}, \bm{v}):= &\; \mathtt{prox}_{\sigma p} \big(X + \sigma(\bm{u}\bm{1}_n^\top + \bm{1}_m\bm{v}^\top + A^\top W B^\top - C)\big), \\
\bm{y}_\sigma(W, \bm{u}, \bm{v}):= &\; \mathtt{prox}_{\sigma p_r}\big(\bm{y} + \sigma \bm{u}\big), \quad
\bm{z}_\sigma(W, \bm{u}, \bm{v}):= \, \mathtt{prox}_{\sigma p_c}\big(\bm{z} + \sigma \bm{v}\big).
\end{aligned}
\end{equation*}
Then, it holds that
\begin{equation*}
\begin{aligned}
\nabla_{W}\mathcal{L}_{\sigma}\big(W, \bm{u}, \bm{v}, X, \bm{y}, \bm{z}\big) = &\;AX_\sigma(\bm{u}, \bm{v}, W)B - S, \\
\nabla_{\bm{u}}\mathcal{L}_{\sigma}\big(W, \bm{u}, \bm{v}, X, \bm{y}, \bm{z}\big) = &\; X_\sigma(\bm{u}, \bm{v}, W)\bm{1}_n + \bm{y}_\sigma(W, \bm{u}, \bm{v}) - \bm{\alpha}, \\
\nabla_{\bm{v}}\mathcal{L}_{\sigma}\big(W, \bm{u}, \bm{v}, X, \bm{y}, \bm{z}\big) = &\; X_\sigma(\bm{u}, \bm{v}, W)^\top \bm{1}_m + \bm{z}_\sigma(W, \bm{u}, \bm{v}) - \bm{\beta}.
\end{aligned}
\end{equation*}
With the above preparations, we are now ready to present the ciPALM for solving problem \eqref{eq:dmain} in Algorithm \ref{algo:ciPALM}.

\begin{algorithm}[htb!]
\caption{A corrected inexact proximal augmented Lagrangian method (ciPALM) for solving problem \eqref{eq:dmain}}\label{algo:ciPALM}
	
\textbf{Input:} Let $\rho\in[0,1)$, and let $\{\sigma_k\}_{k=0}^{\infty}$ and $\{\tau_k\}_{k=0}^{\infty}$ be two sequences of positive real numbers. Choose $\big(W^0, \bm{u}^0, \bm{v}^0, X^0, \bm{y}^0, \bm{z}^0\big) \in \mathbb{R}^{\widetilde m\times \widetilde n}\times\mathbb{R}^m\times \mathbb{R}^n \times \mathbb{R}^{m\times n}\times \mathbb{R}^m \times \mathbb{R}^n$ arbitrarily. Set $k=0$.

\While{a termination criterion is not met,}{
\vspace{1mm}
\textbf{Step 1.} Approximately solve the subproblem
\begin{equation}\label{ciPALM-subpro}
\min\limits_{\bm{u},\bm{v},W}~\mathcal{L}_{\sigma_k}\big(W,\bm{u},\bm{v},X^k,\bm{y}^k,\bm{z}^k\big) + \frac{\tau_k}{2\sigma_k}\left(\big\|W-W^k\big\|_F^2
+ \big\|\bm{u}-\bm{u}^k\big\|^2
+ \big\|\bm{v}-\bm{v}^k\big\|^2\right)
\end{equation}
to find $\big(\widetilde{W}^{k+1}, \widetilde{\bm{u}}^{k+1}, \widetilde{\bm{v}}^{k+1}, \widetilde{X}^{k+1}, \widetilde{\bm{y}}^{k+1}, \widetilde{\bm{z}}^{k+1}\big)$ such that
\begin{eqnarray}
\widetilde{X}^{k+1}
\!\!&:=&\!\! \mathtt{prox}_{\sigma_k p} \left(X^k + \sigma_k\big(\widetilde{\bm{u}}^{k+1}\bm{1}_n^\top + \bm{1}_m(\widetilde{\bm{v}}^{k+1})^\top + A^\top\widetilde{W}^{k+1}B^\top - C\big)\right), \nonumber \\[2pt]
\widetilde{\bm{y}}^{k+1}
\!\!&:=&\!\! \mathtt{prox}_{\sigma_k p_r}\big(\bm{y}^k + \sigma_k \widetilde{\bm{u}}^{k+1}\big),  \nonumber \\[2pt]
\quad
\widetilde{\bm{z}}^{k+1}
\!\!&:=&\!\! \mathtt{prox}_{\sigma_k p_c}\big(\bm{z}^k + \sigma_k \widetilde{\bm{v}}^{k+1} \big),  \nonumber \\[3pt]
\big\|\Delta^{k+1}\big\|
\!\!&\leq&\!\! \frac{\min(\sqrt{\tau_k},\,1)}{\sigma_k} \rho \sqrt{\tau_k\big\|\Delta^{k+1}_d\big\|^2 + \big\|\Delta^{k+1}_p\big\|^2}, \label{inexcond-ciALM}
\end{eqnarray}
where
\begin{equation*}
\begin{aligned}
\Delta^{k+1} \,&:=\, \big(\Delta^{k+1}_W, \,\Delta^{k+1}_u, \,\Delta^{k+1}_v\big), \\[2pt]
\Delta^{k+1}_p \,&:=\, \big(\widetilde{X}^{k+1} - X^k, \,\widetilde{\bm{y}}^{k+1} - \bm{y}^k,\,\widetilde{\bm{z}}^{k+1} - \bm{z}^k\big), \\[2pt]
\Delta^{k+1}_d \,&:=\, \big(\widetilde{W}^{k+1} - W^k, \, \widetilde{\bm{u}}^{k+1} - \bm{u}^k,	\,\widetilde{\bm{v}}^{k+1} - \bm{v}^k\big), \\[2pt]
\Delta^{k+1}_u	\,&:=\, \nabla_{\bm{u}}\mathcal{L}_{\sigma_k}\big(\widetilde{W}^{k+1}, \widetilde{\bm{u}}^{k+1}, \widetilde{\bm{v}}^{k+1}, X^k, \bm{y}^k, \bm{z}^k\big) + \tau_k\sigma_k^{-1}\big(\widetilde{\bm{u}}^{k+1} - \bm{u}^k\big), \\[2pt]
\Delta^{k+1}_v	\,&:=\,  \nabla_{\bm{v}}\mathcal{L}_{\sigma_k}\big(\widetilde{W}^{k+1}, \widetilde{\bm{u}}^{k+1}, \widetilde{\bm{v}}^{k+1}, X^k, \bm{y}^k, \bm{z}^k\big) + \tau_k\sigma_k^{-1}\big(\widetilde{\bm{v}}^{k+1} - \bm{v}^k\big), \\[2pt]
\Delta^{k+1}_W	\,&:=\, \nabla_{W}\mathcal{L}_{\sigma_k}\big(\widetilde{W}^{k+1}, \widetilde{\bm{u}}^{k+1}, \widetilde{\bm{v}}^{k+1}, X^k, \bm{y}^k, \bm{z}^k\big) + \tau_k\sigma_k^{-1}\big(\widetilde{W}^{k+1} - W^k\big).
\end{aligned}
\end{equation*}
		
\textbf{Step 2.} Compute
\begin{equation}\label{extrastep-ciALM}
\begin{aligned}
W^{k+1} &= W^k - \tau_k^{-1}\sigma_k\big(A\widetilde{X}^{k+1}B - S\big),  \\
\bm{u}^{k+1} &= \bm{u}^k - \tau_k^{-1}\sigma_k\big(\widetilde{X}^{k+1}\bm{1}_n + \widetilde{\bm{y}}^{k+1} - \bm{\alpha}\big),  \\			
\bm{v}^{k+1} &= \bm{v}^k - \tau_k^{-1}\sigma_k\big((\widetilde{X}^{k+1})^\top \bm{1}_m + \widetilde{\bm{z}}^{k+1} - \bm{\beta}\big),  \\
X^{k+1} &= \widetilde{X}^{k+1}, \quad
\bm{y}^{k+1} =  \widetilde{\bm{y}}^{k+1}, \quad
\bm{z}^{k+1} = \widetilde{\bm{z}}^{k+1}.
\end{aligned}
\end{equation}
		
\textbf{Step 3.} Set $k=k+1$ and go to \textbf{Step 1}. \vspace{1mm}
}
	
\textbf{Output:} $\big(W^k, \bm{u}^{k}, \bm{v}^{k},  X^k, \bm{y}^{k}, \bm{z}^{k}\big) \in \mathbb{R}^{\widetilde m\times \widetilde n} \times\mathbb{R}^m\times \mathbb{R}^n \times \mathbb{R}^{m\times n}\times \mathbb{R}^m \times \mathbb{R}^n$
\end{algorithm}

The reader may have observed that our ciPALM in Algorithm \ref{algo:ciPALM} is developed based on the augmented Lagrangian function $\mathcal{L}_{\sigma}$ with an adaptive proximal term $\frac{\tau_k}{2\sigma_k}\big(\|W-W^k\|_F^2 + \|\bm{u}-\bm{u}^k\|^2 + \|\bm{v}-\bm{v}^k\|^2 \big)$, and thus, looks similar to the
recent semismooth Newton based inexact proximal augmented Lagrangian ({\sc Snipal}) method in \cite[Section 3]{li2020asymptotically}.
{However, we would like to point out that the {\sc Snipal} is specifically developed for solving linear programming problems, while our ciPALM is tailored to problem \eqref{eq-regOTpro}, which involves an additional group-quadratic regularizer \eqref{eq:R}.
Moreover, compared with the {\sc Snipal}, our ciPALM has used a very different error criterion \eqref{inexcond-ciALM} for solving the subproblem \eqref{ciPALM-subpro} and performed an extra correction step to update $(W^{k+1}, \bm{u}^{k+1},\bm{v}^{k+1})$ in \eqref{extrastep-ciALM}. Specifically, in our context, the {\sc Snipal} requires the error term $\Delta^{k+1}$ to satisfy
\begin{equation}\label{inexcond-SNIPAL}
\begin{aligned}
&(A) \quad \|\Delta^{k+1}\| \leq \frac{\min(\sqrt{\tau_k},\,1)}{\sigma_k}\, \varepsilon_k, ~~\varepsilon_k\geq0, ~~\sum^{\infty}_{k=1}\varepsilon_k<\infty,  \\
&(B) \quad \|\Delta^{k+1}\| \leq \frac{\min(\sqrt{\tau_k},\,1)}{\sigma_k}\,\delta_k \sqrt{\tau_k\big\|\Delta^{k+1}_d\big\|^2 + \big\|\Delta^{k+1}_p\big\|^2}, ~~0\leq\delta_k<1, ~~\sum^{\infty}_{k=1}\delta_k<\infty,
\end{aligned}
\end{equation}
to guarantee the asymptotic superlinear convergence\footnote{Note that the global convergence of the {\sc Snipal} can be readily guaranteed by only employing the error criterion (A); see \cite[Section 3]{li2020asymptotically}.}
and directly set $( W^{k+1},\bm{u}^{k+1},\bm{v}^{k+1}) = (\widetilde{W}^{k+1}, \widetilde{\bm{u}}^{k+1}, \widetilde{\bm{v}}^{k+1})$. Note that the error criteria (A) and (B) are of the absolute type and involve two summable sequences of error tolerance parameters {$\{\varepsilon_k\}\subseteq[0,\infty)$ and $\{\delta_k\}\subseteq[0,1)$}, which require careful tuning for the algorithm to achieve good convergence efficiency.
This indeed makes the parameter tuning of the {\sc Snipal} less friendly
in practical implementations since the performance of the 
algorithm may depend sensitively on the choices of those error tolerance parameters.
In contrast, our ciPALM employs a relative error criterion \eqref{inexcond-ciALM}, which only has a \textit{single} tolerance parameter $\rho\in[0,1)$, and hence the corresponding parameter tuning is typically easier from the computation and implementation perspectives as we shall see in Section \ref{section-classic-ot}. The extra correction step \eqref{extrastep-ciALM} to update the variables $W^{k+1}$, $\bm{u}^{k+1}$, $\bm{v}^{k+1}$ is another difference of our ciPALM from the {\sc Snipal}. It would help to establish the connection between the ciPALM in Algorithm \ref{algo:ciPALM} and the VHPE in Algorithm \ref{algo:VHPE} so that we can readily study the convergence properties of the ciPALM, as we shall see later.

In addition, unlike a recent inexact augmented Lagrangian method with a different relative error criterion developed by Eckstein and Silva \cite{es2013practical}, we are more interested in incorporating a proximal term $\frac{\tau_k}{2\sigma_k}\big(\|W-W^k\|_F^2 + \|\bm{u}-\bm{u}^k\|^2 + \|\bm{v}-\bm{v}^k\|^2\big)$ in the subproblem \eqref{ciPALM-subpro}. Such a proximal term would help not only to guarantee the existence of the optimal solution of the subproblem \eqref{ciPALM-subpro}, but also to ensure the positive definiteness of the coefficient matrix of the underlying semi-smooth Newton linear system when solving the subproblem \eqref{ciPALM-subpro}, as shown in Section \ref{sec:SSN}.


In the following, we study the convergence properties of our ciPALM by establishing the connection between the ciPALM and the VHPE. Then, the convergence results can be readily obtained as a direct application of the general theory of the VHPE in Section \ref{sec:VHPE}. To this end, we define an operator $\mathcal{T}_{\ell}$ associated with the Lagrangian function $\ell\big(W, \bm{u},\bm{v}, X, \bm{y}, \bm{z}\big)$ by
\begin{align*}
&\mathcal{T}_{\ell}\big(W,\bm{u},\bm{v},X,\bm{y},\bm{z}\big) \\[3pt]
:= &
\left\{\big(W',\bm{u}',\bm{v}', X',\bm{y}',\bm{z}'\big) \,\mid\, \big(W', \bm{u}',\bm{v}', -X', -\bm{y}', -\bm{z}'\big)\in\partial\ell\big(W, \bm{u}, \bm{v}, X, \bm{y}, \bm{z}\big)\right\}  \\[3pt]
= & \left\{\big(W',\bm{u}',\bm{v}', X',\bm{y}',\bm{z}'\big) ~\left\lvert~
\begin{aligned}
W' = &~ -S + AXB, \;\;
\bm{u}'=  -\bm{\alpha}+X\bm{1}_n +\bm{y}, \;\;
\bm{v}' = -\bm{\beta} + X^\top \bm{1}_m + \bm{z},\\
X' \in &~ C - \bm{u}\bm{1}_n^\top - \bm{1}_m\bm{v}^\top - A^\top W B^\top + \partial p(X), \\
\bm{y}'\in&~ -\bm{u}+\partial p_r(\bm{y}), \quad
\bm{z}'\in -\bm{v}+\partial p_c(\bm{z}),
\end{aligned}\right.
\right\}.
\end{align*}
It is known from \cite[Corollary 37.5.2]{r1970convex} that $\mathcal{T}_{\ell}$ is maximal monotone. Let $\mathcal{I}_{m}$, $\mathcal{I}_{n}$, $\mathcal{I}_{m,n}$, and $\mathcal{I}_{\widetilde m, \widetilde n}$ be the identity mappings over $\mathbb{R}^{m}$, $\mathbb{R}^{n}$, $\mathbb{R}^{m\times n}$, and $\mathbb{R}^{\widetilde m\times \widetilde n}$, respectively. We define the following self-adjoint positive definite operator over $\mathbb{R}^{\widetilde m\times \widetilde n} \times \mathbb{R}^{m} \times \mathbb{R}^{n} \times  \mathbb{R}^{m\times n} \times \mathbb{R}^{m} \times \mathbb{R}^{n}$:
\begin{equation*}
\Lambda_k := \big(\tau_k\mathcal{I}_{\widetilde m, \widetilde n}, \,\tau_k\mathcal{I}_m, \,\tau_k\mathcal{I}_n, \,\mathcal{I}_{m, n}, \,\mathcal{I}_m, \,\mathcal{I}_n\big)
\end{equation*}
such that for any $\big(W, \bm{u}, \bm{v}, X, \bm{y}, \bm{z}\big)\in \mathbb{R}^{\widetilde m\times \widetilde n} \times \mathbb{R}^{m} \times \mathbb{R}^{n} \times  \mathbb{R}^{m\times n} \times \mathbb{R}^{m} \times \mathbb{R}^{n}$,
\begin{equation*}
\Lambda_k\big(W, \bm{u}, \bm{v}, X, \bm{y}, \bm{z}\big) = \big(\tau_k W, \,\tau_k\bm{u}, \,\tau_k\bm{v}, \,X, \,\bm{y},\, \bm{z}\big),\quad \forall k\geq 0.
\end{equation*}
Clearly, $\Lambda_k$ is nonsingular, and hence $M_k:= \Lambda_k^{-1}$ for $k\geq 0$ is well-defined.

Now, we consider the sequences $\big\{( \widetilde{W}^{k}, \widetilde{\bm{u}}^{k}, \widetilde{\bm{v}}^{k},\widetilde{X}^{k}, \widetilde{\bm{y}}^{k}, \widetilde{\bm{z}}^{k})\big\}$ and $\big\{(W^k, \bm{u}^{k}, \bm{v}^{k}, X^k, \bm{y}^{k}, \bm{z}^{k})\big\}$ generated by the ciPALM. Using \eqref{inexcond-ciALM} with some manipulations, we can obtain that
\begin{equation}\label{eqvicond1}
\bm{d}^{k+1} := \big(\Delta^{k+1}-\tau_k\sigma_k^{-1}\Delta_d^{k+1}, \,-\sigma_k^{-1}\Delta_p^{k+1} \big) \in \mathcal{T}_\ell\big(\widetilde{W}^{k+1}, \widetilde{\bm{u}}^{k+1}, \widetilde{\bm{v}}^{k+1}, \widetilde{X}^{k+1}, \widetilde{\bm{y}}^{k+1}, \widetilde{\bm{z}}^{k+1}\big)
\end{equation}
and
\begin{equation}\label{eqvicond2}
\begin{aligned}
&~\left\|\sigma_k M_k \bm{d}^{k+1} + \big(\widetilde{W}^{k}, \widetilde{\bm{u}}^{k}, \widetilde{\bm{v}}^{k},\widetilde{X}^{k}, \widetilde{\bm{y}}^{k}, \widetilde{\bm{z}}^{k}\big) - \big(W^k, \bm{u}^{k}, \bm{v}^{k}, X^k, \bm{y}^{k}, \bm{z}^{k}\big)\right\|^2_{\Lambda_k} \\[3pt]
= &~ \tau_k^{-1}\sigma_k^2\big\|\Delta^{k+1}\big\|^2
\leq \left(\frac{\sigma_k}{\min(\sqrt{\tau_k},\,1)}\big\|\Delta^{k+1}\big\|\right)^2 \leq \rho^2 \left(\tau_k\big\|\Delta_d^{k+1}\big\|^2
+ \big\|\Delta_p^{k+1}\big\|^2\right) \\[3pt]
=&~ \rho^2\left\|\big(\widetilde{W}^{k+1}, \widetilde{\bm{u}}^{k+1}, \widetilde{\bm{v}}^{k+1}, \widetilde{X}^{k+1}, \widetilde{\bm{y}}^{k+1}, \widetilde{\bm{z}}^{k+1}\big) - \big(W^k, \bm{u}^{k}, \bm{v}^{k}, X^k, \bm{y}^{k}, \bm{z}^{k}\big)\right\|^2_{\Lambda_k}.
\end{aligned}
\end{equation}
Moreover, by the updates of $\big(\bm{u}^{k+1}, \bm{v}^{k+1}, W^{k+1}, X^{k+1}, \bm{x}^{k+1}\big)$ in \textbf{Step 2}, we further have that
\begin{align*}
W^{k+1} &= W^k - \tau_k^{-1}\sigma_k\big(\Delta_W^{k+1} - \tau_k\sigma_k^{-1}(\widetilde{W}^{k+1}-W^k)\big), \\
\bm{u}^{k+1} &= \bm{u}^k - \tau_k^{-1}\sigma_k\big(\Delta_u^{k+1} - \tau_k\sigma_k^{-1}(\widetilde{\bm{u}}^{k+1}-\bm{u}^k)\big), \\
\bm{v}^{k+1} &= \bm{v}^k - \tau_k^{-1}\sigma_k\big(\Delta_v^{k+1} - \tau_k\sigma_k^{-1}(\widetilde{\bm{v}}^{k+1}-\bm{v}^k)\big), \\
X^{k+1} &= X^k - \sigma_k\big(\sigma_k^{-1}(X^k - \widetilde{X}^{k+1})\big),
\\
\bm{y}^{k+1} &= \bm{y}^k - \sigma_k\big(\sigma_k^{-1}(\bm{y}^k-\widetilde{\bm{y}}^{k+1})\big), \\
\bm{z}^{k+1} &= \bm{z}^k - \sigma_k\big(\sigma_k^{-1}(\bm{z}^k-\widetilde{\bm{z}}^{k+1})\big),
\end{align*}
and hence
\begin{equation}\label{eqvicond3}
\big(W^{k+1}, \bm{u}^{k+1}, \bm{v}^{k+1}, X^{k+1}, \bm{y}^{k+1}, \bm{z}^{k+1}\big) = \big(W^{k}, \bm{u}^{k}, \bm{v}^{k}, X^{k}, \bm{y}^{k}, \bm{z}^k\big)
- \sigma_kM_k\bm{d}^{k+1}.
\end{equation}
In view of \eqref{eqvicond1}, \eqref{eqvicond2} and \eqref{eqvicond3}, one can see that the ciPALM in Algorithm \ref{algo:ciPALM} is indeed equivalent to the VHPE in Algorithm \ref{algo:VHPE} for solving the monotone inclusion problem
\begin{equation*}
0\in\mathcal{T}_{\ell}\big(W, \bm{u},\bm{v}, X, \bm{y}, \bm{z}\big)
\end{equation*}
with $\bm{x}^k:=\big(W^k, \bm{u}^{k}, \bm{v}^{k}, X^k, \bm{y}^{k}, \bm{z}^k\big)$, $\widetilde{\bm{x}}^k=\big(\widetilde{W}^{k}, \widetilde{\bm{u}}^{k}, \widetilde{\bm{v}}^{k}, \widetilde{X}^{k}, \widetilde{\bm{y}}^{k}, \widetilde{\bm{z}}^{k}\big)$, $M_k=\Lambda_k^{-1}$, $c_k=\sigma_k$ and $\varepsilon_k\equiv0$, for $k\geq 0$. Then, we can obtain the following convergence results of the ciPALM by applying the convergence results of the VHPE.

\begin{theorem}[\textbf{Global convergence of the ciPALM}]\label{Thm:conv}
Suppose that $\mathcal{T}_{\ell}^{-1}(0)\neq\emptyset$ (namely, there exists a saddle point), $\inf_{k\geq0}\{\sigma_k\}>0$, and the positive sequence $\{\tau_k\}$ satisfies that
\begin{equation*}
\tau_k\geq\tau_{\min}>0, ~~\tau_{k+1} \leq (1+\eta_k)\tau_k ~~ \mbox{with}~ \eta_k>0 ~\mbox{and} ~{\textstyle\sum^{\infty}_{k=0}}\,\eta_k < \infty.
\end{equation*}
Let $\big\{\big(W^k, \bm{u}^{k}, \bm{v}^{k}, X^k, \bm{y}^{k}, \bm{z}^k\big)\big\}$ be the sequence generated by the ciPALM in Algorithm \ref{algo:ciPALM}. Then, $\big\{\big(W^k, \bm{u}^{k}, \bm{v}^{k}, X^k, \bm{y}^{k}, \bm{z}^k\big)\big\}$ is bounded. Moreover, $\big\{\big(W^k, \bm{u}^k, \bm{v}^k\big)\big\}$ converges to an optimal solution of problem \eqref{eq:dmain} and $\big\{\big(X^k, \bm{y}^k, \bm{z}^k\big)\big\}$ converges to an optimal solution of problem \eqref{eq:pmain}.
\end{theorem}
\begin{proof}
Using the conditions on $\{\tau_k\}$, we see that
$0<\tau_{\min}\leq\tau_k\leq\Pi^{\infty}_{i=0}(1+\eta_i)\tau_0<\infty$ for all $k\geq0$. This together with $\tau_{k+1} \leq (1+\eta_k)\tau_k$ implies that $(1+\eta_k)^{-1}\Lambda_k^{-1}\preceq\Lambda_{k+1}^{-1}$ and $0<\min\big\{\Pi^{\infty}_{i=0}(1+\eta_i)^{-1}\tau_0^{-1},\,1\big\}
\leq\lambda_{\min}(\Lambda_k^{-1})\leq\lambda_{\max}(\Lambda_k^{-1})
\leq\max\big\{\tau_{\min}^{-1},\,1\big\}$ for all $k\geq0$. Since the ciPALM in Algorithm \ref{algo:ciPALM} is equivalent to the VHPE in Algorithm \ref{algo:VHPE} for solving $0\in\mathcal{T}_{\ell}\big(W, \bm{u}, \bm{v}, X, \bm{y}, \bm{z}\big)$ (see from \eqref{eqvicond1}, \eqref{eqvicond2} and \eqref{eqvicond3}), it then follows from Theorem \ref{thm-VHPEcon} that the sequence $\big\{\big(W^k, \bm{u}^{k}, \bm{v}^{k}, X^k, \bm{y}^{k}, \bm{z}^k\big)\big\}$ is bounded and converges to a point $\big(W^\infty, \bm{u}^{\infty},\bm{v}^\infty, X^\infty, \bm{y}^{\infty}, \bm{z}^{\infty}\big)$ such that $0\in\mathcal{T}_{\ell}\big(W^\infty, \bm{u}^{\infty},\bm{v}^\infty, X^\infty, \bm{y}^{\infty}, \bm{z}^{\infty}\big)$. Thus, we obtain the desired results and the proof is completed.
\end{proof}

Moreover, under an additional error-bound condition, we can also study the convergence rate of the ciPALM as follows.

\begin{theorem}[\textbf{Linear convergence of the ciPALM}]\label{Thm:convrate}
Suppose that $\mathcal{T}_{\ell}^{-1}(0)\neq\emptyset$ (namely, there exists a saddle point), $\inf_{k\geq0}\{\sigma_k\}>0$, and the positive sequence $\{\tau_k\}$ satisfies that
\begin{equation*}
\tau_k\geq\tau_{\min}>0, ~~\tau_{k+1} \leq (1+\eta_k)\tau_k ~~ \mbox{with}~ \eta_k>0 ~\mbox{and} ~{\textstyle\sum^{\infty}_{k=0}}\,\eta_k < \infty.
\end{equation*}
Let $\big\{\bm{x}^k:=\big(W^k, \bm{u}^{k}, \bm{v}^{k}, X^k, \bm{y}^{k}, \bm{z}^k\big)\big\}$ be the sequence generated by the ciPALM in Algorithm \ref{algo:ciPALM}. Suppose further that $\mathcal{T}_{\ell}$ satisfies Assumption \ref{assmp-errbdweak} associated with $r:=\sqrt{\max\big\{\tau_{\min}^{-1},\,1\big\}}\prod^{\infty}_{i=0} (1+\eta_i)\,\mathrm{dist}_{\Lambda_0}\big(\bm{x}^0,\mathcal{T}_{\ell}^{-1}(0)\big)$.
Then, for sufficiently small $\rho$ and sufficiently large $\sigma_k$, the sequence $\{\bm{x}^k\}$ converges to an element of $\mathcal{T}_{\ell}^{-1}(0)$ at a linear rate.
\end{theorem}
\begin{proof}
The desired results can be readily obtained from Theorem \ref{thm-VHPEconrate}.
\end{proof}

Note that, when $\lambda_1=0$ and $\mathcal{K}_r, \,\mathcal{K}_c\subseteq \mathbb{R}^n$ are chosen as the zero spaces or the nonnegative orthants, $\partial p$, $\partial p_r$, and $\partial p_c$ are polyhedral multifunctions, and hence $\mathcal{T}_{\ell}$ is \blue{a polyhedral multifunction}. It then follows from
\cite[Lemma 2 and Remark 1]{li2020asymptotically} that $\mathcal{T}_{\ell}$ satisfies Assumption \ref{assmp-errbdweak} \blue{when $\mathcal{T}_{\ell}^{-1}(0)\neq\emptyset$}.


\section{A semi-smooth Newton method for solving the subproblem}\label{sec:SSN}

As one can see, for the ciPALM to be truly implementable, it is important to design an efficient algorithm for solving the subproblem \eqref{ciPALM-subpro} to find a point $\big(\widetilde{W}^{k+1}, \widetilde{\bm{u}}^{k+1}, \widetilde{\bm{v}}^{k+1}, \widetilde{X}^{k+1}, \widetilde{\bm{y}}^{k+1}, \widetilde{\bm{z}}^{k+1}\big)$ satisfying the inexact condition \eqref{inexcond-ciALM}. In this section, we shall describe how the subproblem \eqref{ciPALM-subpro} can be solved efficiently. For simplicity, we drop the index $k$ and consider the following generic subproblem in the ciPALM with given $\big(\widehat{W}, \widehat{\bm{u}}, \widehat{\bm{v}}, \widehat{X}, \widehat{\bm{y}}, \widehat{\bm{z}}\big)$ and $\tau,\,\sigma>0$:
\begin{equation}\label{subprogen}
\min\limits_{W,\bm{u},\bm{v}}\, \Psi(W,\bm{u},\bm{v}) := \mathcal{L}_{\sigma}\big(W,\bm{u},\bm{v}, \widehat{X}, \widehat{\bm{y}}, \widehat{\bm{z}}\big)
+ \frac{\tau}{2\sigma}\left(\|W-\widehat{W}\|_F^2
+ \|\bm{u}-\widehat{\bm{u}}\|^2 + \|\bm{v}-\widehat{\bm{v}}\|^2\right).
\end{equation}
Since $\Psi$ is strongly convex and continuously differentiable, problem \eqref{subprogen} admits a unique solution $(W^*,\bm{u}^*,\bm{u}^*)$, which can be computed by solving the nonsmooth equation
\begin{equation}\label{subequa}
\nabla\Psi(W,\bm{u},\bm{v})=0, \quad \big(W,\,\bm{u},\,\bm{v}\big)\in  \mathbb{R}^{\widetilde m\times \widetilde n}\times \mathbb{R}^m\times \mathbb{R}^n,
\end{equation}
where
\begin{equation}\label{defPsi}
\begin{aligned}
&~\nabla \Psi(W,\bm{u},\bm{v}) \\
= &~ \begin{pmatrix}
A\mathtt{prox}_{\sigma p}\left(\widehat{X} + \sigma(\bm{u}\bm{1}_n^\top + \bm{1}_m\bm{v}^\top + A^\top W B^\top - C)\right)B - S + \frac{\tau}{\sigma}(W-\widehat{W})\\
\mathtt{prox}_{\sigma p}\left(\widehat{X} + \sigma(\bm{u}\bm{1}_n^\top + \bm{1}_m\bm{v}^\top + A^\top W B^\top - C)\right)\bm{1}_n + \mathtt{prox}_{\sigma p_r}(\widehat{\bm{y}} + \sigma \bm{u}) - \bm{\alpha} + \frac{\tau}{\sigma}(\bm{u} - \widehat{\bm{u}}) \\
\mathtt{prox}_{\sigma p}\left(\widehat{X} + \sigma(\bm{u}\bm{1}_n^\top + \bm{1}_m\bm{v}^\top + A^\top W B^\top - C)\right)^\top \bm{1}_m + \mathtt{prox}_{\sigma p_c}(\widehat{\bm{z}} + \sigma \bm{v}) - \bm{\beta} + \frac{\tau}{\sigma}(\bm{v} - \widehat{\bm{v}})
\end{pmatrix}.
\end{aligned}
\end{equation}
Then, under a proper semi-smoothness property on $\nabla\Psi(\cdot)$, we can apply an efficient semi-smooth Newton method ({\sc Ssn}) for solving the equation \eqref{subequa}. To this end, we first introduce the definition of ``semi-smoothness with respect to a multifunction", which is adopted from \cite{k1988newton,m1977semismooth,qs1993nonsmooth,s2002semismooth}.

\begin{definition}
Let $\mathcal{O}\subset\mathbb{R}^n$ be an open set, $\mathcal{E}:\mathcal{O}\rightrightarrows\mathbb{R}^{m \times n}$ be a nonempty and compact valued, upper-semicontinuous multifunction and $\mathcal{F}:\mathcal{O}\to\mathbb{R}^{m}$ be a locally Lipschitz continuous function.
$\mathcal{F}$ is said to be strongly semi-smooth at $\bm{x}\in\mathcal{O}$ with respect to $\mathcal{E}$ if $\mathcal{F}$ is directionally differentiable at $\bm{x}$ and for any $\mathcal{J}\in\mathcal{E}(\bm{x}+\Delta\bm{x})$ with $\Delta\bm{x}\to0$,
\begin{equation*}
\mathcal{F}(\bm{x}+\Delta\bm{x}) - \mathcal{F}(\bm{x}) - \mathcal{J}\Delta\bm{x} = O(\|\Delta\bm{x}\|^2).
\end{equation*}
Then, $\mathcal{F}$ is said to be a strongly semi-smooth function on $\mathcal{O}$ with respect to $\mathcal{E}$ if it is strongly semi-smooth everywhere in $\mathcal{O}$ with respect to $\mathcal{E}$.
\end{definition}

We next give the following proposition to identify the strong semi-smoothness of $\nabla\Psi(\cdot)$. For notational simplicity, we denote $\mathbb{X}$ as the space of all linear operators from $\mathbb{R}^{m\times n}$ to $\mathbb{R}^{m\times n}$.

\begin{proposition}\label{prop-ss}
Let $\mathcal{X}:\mathbb{R}^{m\times n}\rightrightarrows \mathbb{X}$, $\mathcal{Y}:\mathbb{R}^{m}\rightrightarrows \mathbb{R}^{m\times m}$ and $\mathcal{Z}:\mathbb{R}^{n}\rightrightarrows \mathbb{R}^{n\times n}$ be nonempty, compact valued, and upper-semicontinuous multifunctions such that for any $X\in\mathbb{R}^{m\times n}$, $\bm{y}\in \mathbb{R}^m$ and $\bm{z}\in \mathbb{R}^n$, $\mathcal{X}(X)\subseteq\mathbb{X}$, $\mathcal{Y}(\bm{y})\subseteq \mathbb{R}^{m\times m}$ and $\mathcal{Z}(\bm{z})\subseteq \mathbb{R}^{n\times n}$ are three sets of self-adjoint positive semidefinite linear operators, respectively. Suppose that $\mathtt{prox}_{\sigma p}(\cdot)$, $\mathtt{prox}_{\sigma p_r}(\cdot)$ and $\mathtt{prox}_{\sigma p_c}(\cdot)$ are strongly semi-smooth with respect to $\mathcal{X}$, $\mathcal{Y}$ and $\mathcal{Z}$, respectively. Then, $\nabla\Psi(\cdot)$ is strongly semi-smooth with respect to $\widehat{\partial}(\nabla\Psi)(\cdot)$, where for given $(W, \bm{u}, \bm{v}) \in \mathbb{R}^{\widetilde m\times \widetilde n} \times \mathbb{R}^m\times \mathbb{R}^n$,
\begin{equation}\label{defpartPsi}
\widehat{\partial}(\nabla \Psi)(W, \bm{u}, \bm{v}) \\
:= \left\{ H_{W, \bm{u},\bm{v}} ~\left\lvert~
\begin{aligned}
\mathcal{J}_{\mathcal{X}} &\in  \mathcal{X}\big(\widehat{X} + \sigma(\bm{u}\bm{1}_n^\top +
\bm{1}_m\bm{v}^\top + A^\top W B^\top - C)\big), \\[2pt]
\mathcal{J}_{\mathcal{Y}} &\in \mathcal{Y}\big(\widehat{\bm{y}} + \sigma \bm{u}\big), \;
\mathcal{J}_{\mathcal{Z}} \in \mathcal{Z}\big(\widehat{\bm{z}} + \sigma \bm{v}\big),
\end{aligned}
\right. \right\},
\end{equation}
and $H_{W, \bm{u}, \bm{v}}$ is a linear operator from $\mathbb{R}^{\widetilde m\times \widetilde n} \times \mathbb{R}^m\times\mathbb{R}^n$ to $\mathbb{R}^{\widetilde m\times \widetilde n} \times \mathbb{R}^m\times\mathbb{R}^n$, defined as
\begin{equation*}
H_{W, \bm{u}, \bm{v}}
\begin{pmatrix}
\Delta W \\ \Delta \bm{u} \\ \Delta\bm{v}
\end{pmatrix} :=
\begin{pmatrix}
\sigma A\left[\mathcal{J}_{\mathcal{X}}\left(\Delta\bm{u}\bm{1}_n^\top + \bm{1}_m(\Delta\bm{v})^\top + A^\top \Delta  W B^\top \right)\right]B + \frac{\tau}{\sigma}\Delta W \\[3pt]
\sigma \left[\mathcal{J}_{\mathcal{X}}\left(\Delta\bm{u}\bm{1}_n^\top + \bm{1}_m(\Delta\bm{v})^\top + A^\top \Delta  W B^\top \right)\right]\bm{1}_n + \sigma \mathcal{J}_{\mathcal{Y}}(\Delta \bm{u}) + \frac{\tau}{\sigma}\Delta\bm{u}  \\[3pt]
\sigma \left[\mathcal{J}_{\mathcal{X}}\left(\Delta\bm{u}\bm{1}_n^\top + \bm{1}_m(\Delta\bm{v})^\top + A^\top \Delta WB^\top \right)\right]^\top\bm{1}_m + \sigma \mathcal{J}_{\mathcal{Z}}(\Delta\bm{v})  + \frac{\tau}{\sigma}\Delta\bm{v}
\end{pmatrix},
\end{equation*}
for all $\big(\Delta W, \Delta\bm{u}, \Delta\bm{v}\big)\in\mathbb{R}^{\widetilde m\times \widetilde n} \times \mathbb{R}^m\times\mathbb{R}^n$. Moreover, for any $(W, \bm{u}, \bm{v})\in \mathbb{R}^{\widetilde m\times \widetilde n} \times \mathbb{R}^m\times\mathbb{R}^n$, every linear mapping in the set $\widehat{\partial}(\nabla \Psi)\big(W, \bm{u}, \bm{v}\big)$ is self-adjoint positive definite.
\end{proposition}
\begin{proof}
First, by definitions of $\mathcal{X}$, $\mathcal{Y}$ and $\mathcal{Z}$, for any $\big(W, \bm{u}, \bm{v}\big)$, every linear operator in the set $\mathcal{X}\big(\widehat{X} + \sigma(\bm{u}\bm{1}_n^\top + \bm{1}_m\bm{v}^\top + A^\top W B^\top - C)\big)$, $\mathcal{Y}\big(\widehat{\bm{y}}+\sigma\bm{u}\big)$ or $\mathcal{Z}\big(\widehat{\bm{z}}+\sigma\bm{v}\big)$) is self-adjoint and positive semidefinite. Since $\tau,\;\sigma > 0$, it is clear that every matrix in the set $\widehat{\partial}(\nabla \Psi)(W, \bm{u}, \bm{v})$ is self-adjoint and positive definite. Moreover, since $\mathtt{prox}_{\sigma p}(\cdot)$ is strongly semi-smooth with respect to $\mathcal{X}$, we see that, for any $(W, \bm{u}, \bm{v})$ and $\mathcal{J}_{\mathcal{X}} \in \mathcal{X}\big(\widehat{X} + \sigma((\bm{u}+\Delta \bm{u}) \bm{1}_n^\top + \bm{1}_m(\bm{v}+\Delta\bm{v})^\top + A^\top (W + \Delta W)B^\top - C)\big)$ with $\Delta W\to 0$, $\Delta\bm{u}\to0$ and $\Delta\bm{v}\to 0$, it holds that
\begin{equation*}
\begin{aligned}
&~\mathtt{prox}_{\sigma p}\left(\widehat{X} + \sigma\big((\bm{u}+\Delta \bm{u}) \bm{1}_n^\top + \bm{1}_m(\bm{v}+\Delta\bm{v})^\top + A^\top (W + \Delta W)B^\top - C\big)\right) \\
&\quad - \mathtt{prox}_{\sigma p}\left(\widehat{X} + \sigma\big(\bm{u}\bm{1}_n^\top + \bm{1}_m\bm{v}^\top + A^\top WB^\top - C\big)\right) \\
&\quad - \mathcal{J}_{\mathcal{X}}\left(\sigma\big(\Delta\bm{u}\bm{1}_n^\top + \bm{1}_m(\Delta\bm{v})^\top + A^\top \Delta W B^\top\big)\right)  \\
=&~ O\left(\big\|\sigma\big(\Delta\bm{u}\bm{1}_n^\top + \bm{1}_m(\Delta\bm{v})^\top + A^\top\Delta  WB^\top\big)\big\|^2\right)
= O\left(\big\|(\Delta W, \Delta\bm{u}, \Delta\bm{v})\big\|^2\right).
\end{aligned}
\end{equation*}
Similarly, we can verify that, for any $\mathcal{J}_{\mathcal{Y}}\in \mathcal{Y}\big(\widehat{\bm{y}}+\sigma(\bm{u}+\Delta\bm{u})\big)$ and $\mathcal{J}_{\mathcal{Z}}\in\mathcal{Z}\big(\widehat{\bm{z}}+\sigma(\bm{v} + \Delta\bm{v})\big)$,
\begin{equation*}
\begin{aligned}
&\; \mathtt{prox}_{\sigma p_r}\big(\widehat{\bm{y}} + \sigma(\bm{u} + \Delta \bm{u})\big) - \mathtt{prox}_{\sigma p_r}\big(\widehat{\bm{y}} + \sigma\bm{u}\big) - \mathcal{J}_{\mathcal{Y}}\big(\sigma\Delta\bm{u}\big) = O\big(\norm{\Delta \bm{u}}^2\big), \\
&\; \mathtt{prox}_{\sigma p_c}\big(\widehat{\bm{z}} + \sigma(\bm{v} + \Delta \bm{v})\big) - \mathtt{prox}_{\sigma p_c}\big(\widehat{\bm{z}} + \sigma\bm{v}\big) - \mathcal{J}_{\mathcal{Z}}\big(\sigma\Delta\bm{v}\big) = O\big(\norm{\Delta \bm{v}}^2\big).
\end{aligned}
\end{equation*}
Using these facts, it is easy to verify that, for any $(W, \bm{u}, \bm{v})$ and $H\in\widehat{\partial} (\nabla \Psi)\big(W + \Delta W, \bm{u}+\Delta\bm{u}, \bm{v}+\Delta\bm{v}\big)$ with $\Delta W\to 0$, $\Delta\bm{u}\to0$ and $\Delta\bm{v}\to 0$, it holds that
\begin{equation*}
\nabla\Psi(W + \Delta W, \bm{u}+\Delta\bm{u}, \bm{v}+\Delta\bm{v}) - \nabla\Psi(W, \bm{u}, \bm{v}) - H(\Delta W,\Delta\bm{u}, \Delta \bm{v}) = O\big(\norm{(\Delta W,\Delta\bm{u}, \Delta\bm{v})}^2\big),
\end{equation*}
which implies that $\nabla\Psi(\cdot)$ is strongly semi-smooth with respect to $\widehat{\partial}(\nabla\Psi)(\cdot)$.
\end{proof}

From Proposition \ref{prop-ss}, we see that the strong semi-smoothness of $\nabla\Phi(\cdot)$ with respect to $\widehat{\partial}(\nabla\Psi)$ can be implied by the strong semi-smoothness of $\mathtt{prox}_{\sigma p}(\cdot)$, $\mathtt{prox}_{\sigma p_r}(\cdot)$ and $\mathtt{prox}_{\sigma p_c}(\cdot)$ with respect to $\mathcal{X}$, $\mathcal{Y}$ and $\mathcal{Z}$, respectively. For many popular regularizers with proper choices of $\mathcal{X}$, $\mathcal{Y}$ and/or $\mathcal{Z}$, it is well-known that the corresponding proximal mappings are strongly semi-smooth (see examples later). With these preparations, we are now ready to present a general framework of the semi-smooth Newton ({\sc Ssn}) method for solving the equation \eqref{subequa} in Algorithm~\ref{algo:SSN}, provided that $\nabla\Psi(\cdot)$ is strongly semi-smooth with respect to $\widehat{\partial}(\nabla\Psi)$. Note that the main computational task in {\sc Ssn} is to solve a sequence of linear systems as described in \textbf{Step 1}. \blue{In our numerical implementation, when the size of the coefficient matrix is moderate (no larger than $4000 \times 4000$ in our experiments), we directly perform the (sparse) Cholesky factorization (e.g., \texttt{chol} provided by {\sc Matlab}) with forward and back substitution to solve the linear system. However, when the problem size becomes larger, factorizing a coefficient matrix (even though it is sparse) is time-consuming. Thus, in this case, we apply the conjugate gradient method (e.g., \texttt{pcg} provided by {\sc Matlab}) instead to approximately solve the linear system.}


\begin{algorithm}[htb!]
\caption{A semi-smooth Newton ({\sc Ssn}) method for solving equation \eqref{subequa}}\label{algo:SSN}
 	
\textbf{Initialization:} Choose $\bar{\eta}\in(0,1)$, $\gamma\in(0,1]$, $\mu\in(0,1/2)$, $\delta\in(0,1)$, and an initial point $(W^{0}, \bm{u}^{0}, \bm{v}^{0})\in\mathbb{R}^{\widetilde m\times \widetilde n} \times \mathbb{R}^m\times \mathbb{R}^n$. Set $j=0$.
 	
\While{a termination criterion is not met,}{ 		
\textbf{Step 1.} Compute $\nabla\Psi\big(W^{j}, \bm{u}^{j}, \bm{v}^{j}\big)$ and select an element $\mathcal{H}_j\in\widehat{\partial}\big(\nabla\Psi)(W^{j} ,\bm{u}^{j}, \bm{v}^{j}\big)$. Solve the linear system
\begin{equation*}
\mathcal{H}_j\big(\Delta W; \Delta\bm{u}; \Delta\bm{v}\big)
= -\nabla\Psi\big(W^{j}, \bm{u}^{j}, \bm{v}^{j}\big),
\end{equation*}
nearly exactly by \blue{the (sparse) Cholesky factorization with forward and backward substitutions}, \textit{or} approximately by the preconditioned conjugate gradient method to find  $\big(\Delta W^j, \Delta\bm{u}^j, \Delta\bm{v}^j\big)$ such that
\begin{equation*}
\big\|\mathcal{H}_j\big(\Delta W^j, \Delta\bm{u}^j, \Delta\bm{v}^j\big)
+ \nabla\Psi\big(W^{j}, \bm{u}^{j}, \bm{v}^{j}\big)\big\|
\leq \min\big(\bar{\eta}, \,\|\nabla\Psi\big(W^{j}, \bm{u}^{j}, \bm{v}^{j}\big)\|^{1+\gamma}\big).
\end{equation*}		

\textbf{Step 2.} (\textbf{Line search}) Find a step size $\alpha_j=\delta^{i_j}$, where $i_j$ is the smallest nonnegative integer $i$ for which
\begin{equation*}
\begin{aligned}
&~\Psi\big( W^{j}+\delta^{i}\Delta W^j, \bm{u}^{j}+\delta^{i}\Delta\bm{u}^j, \bm{v}^{j}+\delta^{i}\Delta\bm{v}\big) \\
\leq &~ \Psi\big(W^{j}, \bm{u}^{j}, \bm{v}^{j}\big) + \mu \delta^{i}\big\langle\nabla\Psi(W^{j}, \bm{u}^{j}, \bm{v}^{j}), \,(\Delta W^j, \Delta\bm{u}^j, \Delta\bm{v}^j)\big\rangle.
\end{aligned}
\end{equation*}
 		
\textbf{Step 3.} Set
$
(W^{j+1}, \bm{u}^{j+1},\bm{v}^{j+1}) = \left(W^j + \alpha_j \Delta W^j, \,\bm{u}^{j} +\alpha_j \Delta \bm{u}^j, \,\bm{v}^{j} +\alpha_j \Delta \bm{v}^j\right).
$

\vspace{1mm}
\textbf{Step 4.} Set $j=j+1$, and go to \textbf{Step 1}. \vspace{1mm}
}
 	
\textbf{Output:} $(W^{j}, \bm{u}^{j}, \bm{v}^{j})$.
\end{algorithm}

In the following, to implement the {\sc Ssn} in Algorithm \ref{algo:SSN}, we characterize $\mathtt{prox}_{\sigma p}(\cdot)$, $\mathtt{prox}_{\sigma p_r}(\cdot)$ and $\mathtt{prox}_{\sigma p_c}(\cdot)$, and choose proper $\mathcal{X}$, $\mathcal{Y}$ and $\mathcal{Z}$ for $\mathcal{R}$ chosen as \eqref{eq:R}, and $\mathcal{T}$ chosen as \eqref{eq-cT}. First, recall that problem \eqref{eq-regOTpro} can be written in the form of  \eqref{eq:pmain} with
\begin{equation*}
p(X) := \lambda_1 \sum_{G\in \mathcal{G}}\omega_G\norm{\bm{x}_G} + \frac{\lambda_2}{2}\norm{X}_F^2 + \delta_{\mathbb{R}_+^{m\times n}}(X), \quad p_r(\bm{y}) := \delta_{\mathcal{K}_r}(\bm{y}),\quad p_c(\bm{z}) := \delta_{\mathcal{K}_c}(\bm{z}).
\end{equation*}
To avoid possible confusions, we repeat here that $\bm{x}_G$ is the vector in $\mathbb{R}^{|G|}$ extracted from the matrix $X\in \mathbb{R}^{m\times n}$ via the lexicographically ordered index set $G\in \mathcal{G}$.

We first consider the function $p(\cdot)$. As a consequence of the non-overlapping structure of $\mathcal{G}$, to evaluate $\mathtt{prox}_{\sigma p}(\cdot)$, it is sufficient to discuss the computation on each $G\in\mathcal{G}$. In particular, given any $G\in \mathcal{G}$, we define the function (without loss of generality, we assume that $\lambda_1 > 0$ and $\omega_G > 0$):
\begin{equation*}
p_G(\bm{x}_G):= \lambda_1 \omega_G \norm{\bm{x}_G} + \frac{\lambda_2}{2}\norm{\bm{x}_G}^2 + \delta_{\mathbb{R}^{|G|}_+}(\bm{x}_G), \quad \forall\; \bm{x}_G\in \mathbb{R}^{|G|}.
\end{equation*}
Then, we can verify that
\begin{align*}
\mathtt{prox}_{\sigma p_G}(\bm{x}_G)
=~& \argmin{\bm{z}_G\in \mathbb{R}^{|G|}}\left\{p_G(\bm{z}_G) + \frac{1}{2\sigma}\norm{\bm{z}_G - \bm{x}_G}^2  \right\} \\
=~& \argmin{\bm{z}_G\in \mathbb{R}^{|G|}}\left\{\lambda_1 \omega_G \norm{\bm{z}_G} + \frac{\lambda_2}{2}\norm{\bm{z}_G}^2 + \frac{1}{2\sigma}\norm{\bm{z}_G - \bm{x}_G}^2 ~:~ \bm{z}_G\geq 0 \right\} \\
=~& \argmin{\bm{z}_G\in \mathbb{R}^{|G|}}\left\{\norm{\bm{z}_G} + \frac{\sigma \lambda_2 + 1}{2\sigma\lambda_1\omega_G}\norm{\bm{z}_G - \frac{1}{\sigma\lambda_2+1}\bm{x}_G}^2 ~:~ \bm{z}_G\geq 0 \right\} \\
=~& \argmin{\bm{z}_G\in \mathbb{R}^{|G|}}\left\{\norm{\bm{z}_G} + \frac{\sigma \lambda_2 + 1}{2\sigma\lambda_1\omega_G}\norm{\bm{z}_G - \frac{1}{\sigma\lambda_2+1}\Pi_{\mathbb{R}^{|G|}_+}(\bm{x}_G)}^2 \right\} \\
=~& \mathtt{prox}_{\frac{\sigma\lambda_1\omega_G}{\sigma\lambda_2+1}\norm{\cdot}}\left(\frac{1}{\sigma\lambda_2+1}\Pi_{\mathbb{R}^{|G|}_+}(\bm{x}_G)\right),
\end{align*}
where the fourth equality follows from \cite[Proposition 1]{kim2012group}. Consequently, it holds that
\begin{equation*}
\left[\mathtt{prox}_{\sigma p}(X)\right]_G = \mathtt{prox}_{\sigma p_G}(\bm{x}_G) = \mathtt{prox}_{\frac{\sigma\lambda_1\omega_G}{\sigma\lambda_2+1}\norm{\cdot}}\left(\frac{1}{\sigma\lambda_2+1}\Pi_{\mathbb{R}^{|G|}_+}(\bm{x}_G)\right),\quad \forall \; G\in \mathcal{G}.
\end{equation*}

We next discuss how to derive a suitable multifunction $\mathcal{X}$ for $\mathtt{prox}_{\sigma p}(\cdot)$. To this end, we first recall some well-known results which are useful for our later exposition. Given any scalar $\zeta > 0$ and $\bm{x}_G\in \mathbb{R}^{|G|}$, one can show by direct computation that
\begin{equation}\label{prox-Fnorm}
\mathtt{prox}_{\zeta\norm{\cdot}}(\bm{x}_G) =
\left\{
\begin{array}{ll}
\max\left\{1 - \frac{\zeta}{\norm{\bm{x}_G}}, \,0\right\} \bm{x}_G, & \textrm{if}~ \bm{x}_G \neq 0, \\
0, & \textrm{otherwise}.
\end{array}\right.
\end{equation}
Moreover, we know from, e.g., \cite[Lemma 2.1]{zzst2020efficient}, that $\mathtt{prox}_{\zeta \norm{\cdot}}(\cdot)$ is strongly semi-smooth with respect to its Clarke generalized Jacobian $\partial \mathtt{prox}_{\zeta \norm{\cdot}}(\cdot)$ which takes the following form:
\begin{equation}\label{Jabl2nrom}
\partial \mathtt{prox}_{\zeta \norm{\cdot}}(\bm{x}_G) =
\left\{
\begin{array}{ll}
\left\{\left(1 - \frac{\zeta}{\norm{\bm{x}_G}}\right)I_{|G|} + \frac{\zeta}{\norm{\bm{x}_G}^3}\bm{x}_G\bm{x}_G^\top \right\}, & \textrm{if}~ \norm{\bm{x}_G} > \zeta, \vspace{3mm}\\
\left\{\frac{\chi}{\zeta^2}\bm{x}_G\bm{x}_G^\top \mid \chi \in [0,1] \right\}, & \textrm{if}~ \norm{\bm{x}_G} = \zeta, \vspace{3mm}\\
0, & \textrm{otherwise},
\end{array}\right.
\end{equation}
for any $\bm{x}_G \in \mathbb{R}^{|G|}$. Second, it is known from, e.g., \cite[Proposition 7.4.7]{fp2003finite}, that $\Pi_{\mathbb{R}^{|G|}_+}(\cdot)$ is strongly semi-smooth with respect to its Clarke generalized Jacobian $\partial\Pi_{\mathbb{R}^{|G|}_+}(\cdot)$, which is given as follows: for any given $\bm{x}_G\in \mathbb{R}^{|G|}$ for $G\in \mathcal{G}$,
\begin{equation}\label{JabPi}
\partial \Pi_{\mathbb{R}^{|G|}_+}(\bm{x}_G) =
\left\{
\Diag(\bm{\theta}_G)\,:\,\bm{\theta}_G\in \mathbb{R}^{|G|},\,[\bm{\theta}_G]_i \in
\left\{
\begin{array}{ll}
\left\{1\right\}, & \textrm{if}~ [\bm{x}_G]_i > 0, \vspace{1mm} \\
\left[0,1\right], & \textrm{if}~ [\bm{x}_G]_i = 0, \vspace{1mm} \\
\left\{0\right\}, & \textrm{otherwise},
\end{array}\right.\, 1\leq i\leq |G|
\right\}.
\end{equation}
With the above preparations, we can give the following results showing that, for each $G\in \mathcal{G}$, one can derive a surrogate generalized Jacobian $\mathcal{X}_G(\cdot)$ of a composite map of $\mathtt{prox}_{\zeta\norm{\cdot}}(\cdot)$ and $\Pi_{\mathbb{R}_+^{|G|}}(\cdot)$ so that this composite map is strongly semi-smooth with respect to $\mathcal{X}_G$.
\begin{proposition}\label{prop:J}
For each $G\in \mathcal{G}$ and any given $\bm{x}_G\in \mathbb{R}^{|G|}$, define
a multifunction $\mathcal{X}_G: \mathbb{R}^{|G|} \rightrightarrows \mathbb{R}^{|G|\times |G|}$ as follows:
\begin{equation*}
    \mathcal{X}_G(\bm{x}_G) := \left\{\frac{1}{\sigma\lambda_2+1}\mathcal{J}_1 \mathcal{J}_2 ~:~ \mathcal{J}_1 \in \partial \mathtt{prox}_{\frac{\sigma\lambda_1\omega_G}{\sigma\lambda_2+1}\norm{\cdot}}\left(\frac{1}{\sigma\lambda_2+1}\Pi_{\mathbb{R}^{|G|}_+}(\bm{x}_G)\right),\; \mathcal{J}_2\in \partial \Pi_{\mathbb{R}^{|G|}_+}(\bm{x}_G)\right\}.
\end{equation*}
Then, the following statements hold.
\begin{itemize}
\item[(i)] $\mathcal{X}_G$ is a nonempty, compact-valued, and upper-semicontinuous multifunction.
\item[(ii)] For  any $\mathcal{J}_G\in \mathcal{X}_G(\bm{x}_G)$, $\mathcal{J}_G$ is symmetric and positive semidefinite.
\item[(iii)] For any $\mathcal{J}_G\in \mathcal{X}_G(\bm{x}_G + \Delta \bm{x}_G)$ with $\Delta \bm{x}_G \to 0$, it holds that
    \begin{equation*}
    \mathtt{prox}_{\sigma p_G}(\bm{x}_G + \Delta \bm{x}_G) - \mathtt{prox}_{\sigma p_G}(\bm{x}_G) - \mathcal{J}_G(\Delta \bm{x}_G) = O\big(\norm{\Delta \bm{x}_G}^2\big).
    \end{equation*}
\end{itemize}
\end{proposition}
\begin{proof}
Since statements (i) and (iii) follow from \cite[Theorem 7.5.17]{fp2003finite} and statement (ii) can be verified straightforwardly, we omit the detail here.
\end{proof}

Using Proposition \ref{prop:J}, we now can define a multifunction $\mathcal{X}$ for $\mathtt{prox}_{\sigma p}(\cdot)$ so that $\mathtt{prox}_{\sigma p}(\cdot)$  is strongly semi-smooth with respect to $\mathcal{X}$.
\begin{proposition}\label{prop:U}
For any given $X\in \mathbb{R}^{m\times n}$, define a multifunction $\mathcal{X}: \mathbb{R}^{m\times n}\rightrightarrows \mathbb{X}$ as follows:
\begin{equation*}
\mathcal{X}(X) := \left\{\mathcal{J}_{\{\mathcal{J}_G\;:\; G\in \mathcal{G}\}}\;:\; \mathcal{J}_G\in \mathcal{X}_G(\bm{x}_G),\; G\in \mathcal{G} \right\},
\end{equation*}
where $\mathcal{J}_{\{J_G\;:\; G\in \mathcal{G}\}}\in \mathcal{X}(X)$ is defined as
\begin{equation*}
\left[\mathcal{J}_{\{\mathcal{J}_G\;:\; G\in \mathcal{G}\}}(Z)\right]_G := \mathcal{J}_G(\bm{z}_G),\quad \forall\,G\in \mathcal{G},\; Z\in \mathbb{R}^{m\times n}.
\end{equation*}
Then, the following statements hold for the multifunction $\mathcal{X}$.
\begin{itemize}
\item[(i)] $\mathcal{X}$ is nonempty, compact-valued, and upper-semicontinuous multifunction.
\item[(ii)] For any $\mathcal{J}\in \mathcal{X}(X)$, $\mathcal{J}$ is self-adjoint and positive semidefinite.
\item[(iii)] For any $\mathcal{J}\in \mathcal{X}(X+\Delta X)$ with $\Delta X \to 0$,
    \begin{equation*}
    \mathtt{prox}_{\sigma p}(X+\Delta X) - \mathtt{prox}_{\sigma p}(X) - \mathcal{J}(\Delta X) = O\left(\norm{\Delta X}_F^2\right).
    \end{equation*}
\end{itemize}
\end{proposition}

For the function $p_r(\cdot)$, it is clear that
\begin{equation*}
\mathtt{prox}_{\sigma p_r}(\bm{y}) = 
\left\{
\begin{array}{ll}
0, &  \mathcal{K}_r = \{0\}^m, \vspace{1mm} \\
\Pi_{\mathbb{R}_+^m}(\bm{y}), &  \mathcal{K}_r = \mathbb{R}_+^m,
\end{array}
\right.,\quad \forall \bm{y}\in \mathbb{R}^m.
\end{equation*}
One can also verify that
\begin{equation*}
\partial \mathtt{prox}_{\sigma p_r}(\bm{y}) = \left\{
\begin{array}{ll}
\{\bm{0}\}, &  \mathcal{K}_r = \{0\}^m, \vspace{1mm} \\
\partial \Pi_{\mathbb{R}_+^m}(\bm{y}), &  \mathcal{K}_r = \mathbb{R}_+^m,
\end{array}
\right.,\quad \forall \bm{y}\in \mathbb{R}^m,
\end{equation*}
where, for any $\bm{y}\in\mathbb{R}^m$, $\partial\Pi_{\mathbb{R}_+^m}(\bm{y})$ is given by
\begin{equation*}
\partial\Pi_{\mathbb{R}_+^m}(\bm{y}) = \left\{
\Diag(\bm{\theta}) \,\Big|\,  ~\bm{\theta}_i\in\left\{\begin{array}{ll}
\{1\}, & \mathrm{if} ~\bm{y}_i>0, \\[2pt]
\mbox{$[0,1]$}, & \mathrm{if} ~\bm{y}_i=0, \\[2pt]
\{0\}, & \mathrm{otherwise},
\end{array}\right.
1\leq i\leq m \right\} \subseteq \mathbb{S}^m_{+}.
\end{equation*}
Since $\Pi_{\mathbb{R}_+^m}(\cdot)$ is strongly semi-smooth with respect to its Clarke generalized Jacobian $\partial\Pi_{\mathbb{R}_+^m}(\cdot)$, we can directly choose the multifunction $\mathcal{Y}$ as $\partial \mathtt{prox}_{\sigma p_r}$.

The case for the function $p_c(\cdot)$ can be argued similarly as above. With the above discussions and our choices of $\mathcal{X}$, $\mathcal{Y}$ and $\mathcal{Z}$, we can see that $\widehat{\partial}(\nabla\Psi)(\cdot)$ in \eqref{defpartPsi} is well-defined. Hence, the {\sc Ssn} in Algorithm \ref{algo:SSN} is also well-defined since one can show that any element $\mathcal{H}_j \in \widehat{\partial}(\nabla\Psi)( W^j, \bm{u}^j, \bm{v}^j)$, for $j\geq 0$, is self-adjoint positive definite and the line-search scheme (see \textbf{Step 2}) is also well-defined (which is ensured by our inexact conditions when solving the linear system in \textbf{Step 1}). Indeed, we have the following theorem stating the convergence properties for the {\sc Ssn} in Algorithm \ref{algo:SSN}.

\begin{theorem}\label{thm-rate-ssn}
Suppose that $\mathcal{X}$ is chosen as in Proposition \ref{prop:U}, $\mathcal{Y} = \partial\mathtt{prox}_{\sigma p_r}$, and $\mathcal{Z}=\partial \mathtt{prox}_{\sigma p_c}$. Let $\big\{\big(W^{j}, \bm{u}^{j}, \bm{v}^{j}\big)\big\}$ be the sequence generated by the {\sc Ssn} in Algorithm \ref{algo:SSN}. Then, $\big\{\big(W^{j},\bm{u}^{j}, \bm{v}^{j}\big)\big\}$ is well-defined and converges to the unique solution $(W^*, \bm{u}^*, \bm{v}^*)$ of equation \eqref{subequa}. Moreover, for sufficiently large $j$, we have
\begin{equation*}
\big\|\big( W^{j+1}-W^*,\,\bm{u}^{j+1}-\bm{u}^*,\,\bm{v}^{j+1}-\bm{v}^*\big)\big\| = O\left(\big\|(W^{j}-W^*, \,\bm{u}^{j}-\bm{u}^*, \,\bm{v}^{j}-\bm{v}^*)\big\|^{1+\gamma}\right),
\end{equation*}
where $\gamma\in(0,1]$ is the parameter pre-specified in Algorithm \ref{algo:SSN}.
\end{theorem}
\begin{proof}
The proof follows the same way as in
\cite[Theorem 3.6]{lst2018efficiently} and thus is omitted here.
\end{proof}

\blue{
From Theorem \ref{Thm:convrate}, we see that under a proper error-bound condition, our ciPALM in Algorithm \ref{algo:ciPALM} exhibits a linear convergence rate and the linear rate can be arbitrarily fast by selecting suitable hyperparameters (i.e., $\sigma_k$ and $\rho$). Moreover, from Theorem \ref{thm-rate-ssn}, the quadratically convergent semismooth Newton method enables one to solve the subproblem efficiently at each iteration. Thus, the proposed double-looped algorithmic framework is shown to be highly efficient in both outer and inner loops. This may partially explain why the proposed algorithm has promising practical performances, as shown in the next numerical section.
}

\section{Numerical experiments}\label{sec:num}

In this section, we conduct numerical experiments to evaluate the performance of our ciPALM in Algorithm \ref{algo:ciPALM} for solving some classes of unregularized and regularized OT problems that can be covered by \eqref{eq-regOTpro} or \eqref{eq:pmain}. All experiments are run in {\sc Matlab} R2023a on a PC with Intel processor i7-12700K@3.60GHz (with 12 cores and 20 threads) and 64GB of RAM, equipped with a Windows OS. The implementation details are given as follows.

\textbf{Termination conditions.} We denote $\mathtt{tol}$ as the stopping tolerance, $\mathtt{maxiter}$ as the maximum number of iterations, and $\mathtt{maxtime}$ as the maximum running time. We shall terminate our ciPALM when it returns a point $(W^{k}, \bm{u}^{k}, \bm{v}^{k}, X^{k}, \bm{y}^{k}, \bm{z}^k)$ satisfying one of the following conditions:
\begin{itemize}
\item The relative optimality residual $\eta^k := \max \big\{\eta_{X}^k, \,\eta_{\bm{y}}^k, \,\eta_{\bm{z}}^k, \,\eta_{{\rm feas}}^k, \,\eta_{{\rm gap}}^k\big\} < \mathtt{tol}$, where
	\begin{equation*}
	\begin{aligned}
	\eta_{ X}^k := &~
	\frac{\norm{ X^k - \prox_{p}\big(X^k + \bm{u}^k\bm{1}_n^\top + \bm{1}_m(\bm{v}^k)^\top + A^\top W^k B^\top - C\big)}_F}{1+\norm{C}_F }, \\
    \eta_{\bm{y}}^k := &~ \frac{\norm{\bm{y}^k - \prox_{p_r}\big(\bm{y}^k + \bm{u}^k\big)}}{1 + \norm{\bm{y}^k} + \norm{\bm{u}^k}} ,\quad
    \eta_{\bm{z}}^k := \frac{\norm{\bm{z}^k - \prox_{p_c}\big(\bm{z}^k + \bm{v}^k\big)}}{1 + \norm{\bm{z}^k} + \norm{\bm{v}^k}}, \\
    \eta_{{\rm feas}}^k := &~ \frac{\sqrt{\norm{X^k\bm{1}_n + \bm{y}^k - \bm{\alpha}}^2 + \norm{(X^k)^\top \bm{1}_m + \bm{z}^k - \bm{\beta}}^2 + \norm{AX^kB - S}_F^2 }}{1+\norm{\bm{\alpha}} + \norm{\bm{\beta}} + \norm{S}_F },\\
    \eta_{{\rm gap}}^k := &~ \frac{\abs{\mathtt{pobj}-\mathtt{dobj}} }{1+\abs{\mathtt{pobj}}+\abs{\mathtt{dobj}}},
	\end{aligned}
	\end{equation*}
    where $\mathtt{pobj} := \inner{C,X^k} + \lambda_1 \sum_{G\in\mathcal{G}} \omega_G \norm{\bm{x}_G^k} + \frac{\lambda_2}{2}\norm{X^k}^2$ and $\mathtt{dobj} := \inner{S,W^k} + \inner{\bm{\alpha},\bm{u}^k} + \inner{\bm{\beta},\bm{v}^k}-p^*\big(\bm{u}^k\bm{1}_n^\top + \bm{1}_m\bm{v}^{k\top} + A^\top W^k B^\top - C\big)$. 

\item The number of iterations $k>\mathtt{maxiter}$;

\item The total running time exceeds $\mathtt{maxtime}$.
\end{itemize}
In our experiments, we set $\mathtt{tol}=10^{-6}$, $\mathtt{maxiter} = 10^3$, and  $\mathtt{maxtime}$ to be 2 hours.

\textbf{Baseline solvers.} We next introduce our baseline solvers under two different scenarios: $\lambda_1=0$ and $\lambda_1>0$. For $\lambda_1=0$, problem \eqref{eq-regOTpro} is essentially a linear programming (LP) problem or a convex quadratic programming (QP) problem that can be solved efficiently and accurately by the well-developed commercial solver Gurobi. Moreover, the LP formed from \eqref{eq-regOTpro} can also be solved efficiently by the semismooth Newton based inexact proximal augmented Lagrangian ({\sc Snipal}) method developed in \cite{li2020asymptotically}. Thus, in this case, we shall compare our ciPALM with {\sc Snipal} and Gurobi\footnote{We use Gurobi (version 10.0.1 with an academic license) by only choosing the barrier method and disabling the cross-over strategy so that Gurobi has the best overall performance based on our experiments.}. For $\lambda_1>0$, the presence of the group regularizer in the objective function makes problem \eqref{eq-regOTpro} neither an LP or a convex QP. \blue{Consequently, {\sc Snipal} is not longer applicable. On the other hand, we observe that by adding slack variables, problem \eqref{eq-regOTpro} can be reformulated as a second-order cone programming (SOCP) problem which can be efficiently solved by commercial solvers such as Mosek; see Appendix \ref{appendix-socp} for the explicit SOCP reformulation. Moreover, the SOCP reformulation can be further {converted to} a quadratically constrained quadratic programming (QCQP) problem which can {then} be solved by Gurobi. However, our numerical experiments show that solving the QCQP reformulation via Gurobi is significantly slower than solving the SOCP reformulation directly via Mosek. Hence, we only compare our ciPALM with Mosek}\footnote{We only use the barrier method implemented in Mosek (version 10.0.46 with an academic license). Note that for LPs, Gurobi and Mosek share comparable performance when they are able to solve the tested problems successfully. However, based on our numerical experience, Mosek turns out to be less stable for solving large-scale LPs. Hence, for simplicity and ease of comparison, we only present the numerical results of Gurobi for LPs; see also Section \ref{section-MOT}.}. For both Gurobi and Mosek, we set the corresponding termination tolerances as $10^{-6}$, which matches the termination tolerance for our ciPALM. Finally, for a particular test problem, Gurobi or Mosek can often provide a reasonably accurate solution. We then use the primal solution $(X_b, \bm{y}_b, \bm{z}_b)$ obtained by Gurobi or Mosek as a benchmark to evaluate the quality of the primal solution $(X^k, \bm{y}^k, \bm{z}^k)$ obtained by our ciPALM. Specifically, we compute the normalized objective function value with respect to $(X_b, \bm{y}_b, \bm{z}_b)$, which is defined as $\texttt{nobj}:= \frac{\abs{\inner{C,\,X^k} +p(X^k) - \inner{C,\,X_{b}}-p(X_{b}) } }{1+\abs{\inner{C,\,X_{b} } + p(X_{b})}}$. Moreover, in order to measure the primal constraint violation at a given point $(X, \bm{y}, \bm{z})$, we also compute
\begin{equation*}
    \texttt{feas}:= \max\left\{\begin{array}{l}
    \frac{\sqrt{\norm{X\bm{1}_n + \bm{y} - \bm{\alpha}}^2 + \norm{X^\top \bm{1}_m + \bm{z} - \bm{\beta}}^2 + \norm{AXB - S}_F^2 }}{1+\norm{\bm{\alpha}} + \norm{\bm{\beta}} + \norm{S}_F },
    \frac{\big\|\Pi_{\mathbb{R}^{m\times n}_-}(X)\big\|_F}{1 + \norm{X}_F}, \frac{\norm{\Pi_{\mathcal{K}_r^\circ}(\bm{y})}}{1+\norm{\bm{y}}}, \frac{\norm{ \Pi_{\mathcal{K}_c^\circ}(\bm{z})}}{1+\norm{\bm{z}}}
    \end{array}
    \right\},
\end{equation*}
where $\mathcal{K}_r^\circ$ and $\mathcal{K}_c^\circ$ denote the polar cones of  $\mathcal{K}_r$ and $\mathcal{K}_c$, respectively.

\textbf{Initial points.} Our numerical experience \blue{(see, e.g., \cite{li2020asymptotically,liang2022qppal,sun2020sdpnal+})} suggests that it is beneficial to start with a reasonably good initial point so that our ciPALM, as well as the {\sc Snipal}, can converge faster. To this end, we proposed to apply a certain alternative direction method of multipliers (ADMM) type method for solving the dual problem \eqref{eq:dmain-ADMM} to perform the warmstart strategy. It is worth noting that, depending on how we update the dual variables, we can apply the classic ADMM (denoted by dADMM, see, e.g. \cite{boyd2011distributed,gabay1976dual}) method or a symmetric Gauss-Seidel based ADMM (denoted by dSGSADMM, see, e.g. \cite{chen2021equivalence,chen2017efficient}). We refer readers to Appendix \ref{appendix-admm} for detailed descriptions of the dADMM and dSGSADMM. As observed from our numerical experiments, the dSGSADMM is often more efficient than the dADMM, and hence, it is used to warm start our ciPALM and the {\sc Snipal}. Specifically, we terminate the dSGSADMM as long as it produces a point with the relative KKT residual less than $\texttt{toladmm}:=10^{-3}$ or it reaches the maximal number of iterations $\texttt{maxiteradmm}:=500$. Here, we
should mention that as first-order methods, both dADMM and dSGSADMM are usually too slow to provide a solution with
the residual $\eta^k$ less than $\texttt{tol} := 10^{-6}$. In this paper, to save space, we will not include the numerical results of applying them alone for solving problem \eqref{eq-regOTpro}. We would also like to mention that the computational time for warmstarting is included in the total computational time for fair comparisons.

\textbf{Hyperparameters.} Our ciPALM and the {\sc Snipal} also require proper choices of $\{\tau_k\}$ and $\{\sigma_k\}$ to achieve good performances. 
In our experiments, for both algorithms, we simply set $\tau_0=5$, $\tau_{k+1}=(1+(k+1)^{-1.1})\tau_k$, and $\sigma_k=\min\big\{10^4, \,\max\big\{10^{-4}, \,1.5^{k}\big\}\big\}$ for all $k\geq0$. Note that such choices of $\{\tau_k\}$ and $\{\sigma_k\}$ satisfy the required conditions in Theorems \ref{Thm:conv} and \ref{Thm:convrate}. Moreover, we would like to emphasize that more delicate updating rules for $\tau_k$ and $\sigma_k$ are possible and may lead to better numerical performances. In this paper, since we aim to investigate the influence of different inexact conditions on the subproblems, we then use the above simple updating rules and focus on different choices of $\rho$ in \eqref{inexcond-ciALM} for our ciPALM, and two summable sequences $\{\varepsilon_k\}$ and $\{\delta_k\}$ in \eqref{inexcond-SNIPAL} for the {\sc Snipal}, for the ease of comparison. In addition, for the {\sc Ssn} in Algorithm \ref{algo:SSN}, we set $\mu=10^{-4}$, $\delta=0.5$, $\bar{\eta}=10^{-3}$ and $\gamma=0.2$.

\subsection{The classical optimal transport problem}\label{section-classic-ot}

In this part of experiments, we investigate how the choices of $\rho\in [0,1)$, and $\{\varepsilon_k\}$ and $\{\delta_k\}$ would affect the performance of the ciPLAM and {\sc Snipal}, respectively. For simplicity, we consider solving the classical optimal transport problem \eqref{eq-DOT} and follow \cite[Section 4.1]{clty2023efficient} to randomly generate OT instances. Specifically, we first generate two discrete probability distributions denoted by $D_1 := \left\{ (a_i, \,\bm{p}_i)\in \mathbb{R}_+\times \mathbb{R}^3\;:\; i = 1,\dots,m \right\}$ and $D_2 := \left\{ (b_j, \,\bm{q}_j)\in \mathbb{R}_+\times \mathbb{R}^3\;:\; j = 1,\dots,n\right\}$. Here, $\bm{a}:=(a_1, \dots\!, a_{m})^{\top}$ and $\bm{b}:=(b_1, \dots\!, b_{n})^{\top}$ are probabilities/weights generated from the standard uniform distribution on the open interval $(0,1)$,  and further normalized such that $\sum^{m}_{i=1}a_i=\sum^{n}_{j=1}b_j=1$. Moreover, $\{\bm{p}_i\}$ and $\{\bm{q}_j\}$ are the support points whose entries are drawn from a Gaussian mixture distribution. With these support points, the cost matrix $C$ is generated by $C_{ij} = \|\bm{p}_i-\bm{q}_j\|^2$ for $1\leq i\leq m$ and $1\leq j\leq n$.

In our experiments, we choose $m=n\in\{1000,\,2000\}$. For the ciPALM, we solve the OT problem with $\rho\in\big\{8\times 10^{-1}, \,4\times 10^{-1}, \,1\times 10^{-1}, \,8\times 10^{-2}, \,4\times 10^{-2}, \,1\times 10^{-2}, \,8\times 10^{-3}, \,4\times 10^{-3}, \,1\times 10^{-3}, \,8\times 10^{-4}\big\}$ (there are 10 choices). For the {\sc Snipal}, we consider $\varepsilon_k = \varepsilon_0/(k+1)^p,\; \delta_k = \delta_0/(k+1)^q$ with $\varepsilon_0=\delta_0\in\big\{1, \,10^{-3}\big\}$ and $p,q\in\big\{1.1, \,2.1, \,3.1\big\}$ (hence, there are 18 combinations in total). In order to evaluate the performance, we record the computational time (\texttt{time}), the number of outer iterations (\#), and total number of linear systems solved (\texttt{lin}\#) of both algorithms.

The computational results are presented in Tables \ref{tab:ot1000} and \ref{tab:ot2000}. From the results, one can see that the performance of the
both algorithms would depend on the choices of error tolerance parameters. With proper choices of tolerance parameters, our ciPALM and the {\sc Snipal} can be comparable to each other. This is indeed reasonable because both ciPALM and {\sc Snipal} essentially use the similar PALM+{\sc Ssn} framework but with different stopping criteria for solving the subproblems. Since our ciPALM only involves a single tolerance parameter $\rho\in[0,1)$, it could be more friendly to the parameter tunings. This supports the main motivation of this work to employ a relative-type stopping criterion.

\blue{We also conduct the numerical comparisons between our ciPALM (with $\rho=0.01$) and  Gurobi for solving the classical OT problem on images in the “ClassicImages” class from the DOTmark collection \cite{schrieber2016dotmark}, which serves as a benchmark dataset for discrete OT problems. We mention that the images in the “ClassicImages” class are sourced from real-world scenarios and consist of ten different images, each with different resolutions of $32\times 32$, $64\times64$, $128\times128$ and $512 \times 512$. Thus, for each resolution, we can pair any two different images and compute the OT problem, resulting in 45 OT problems. However, due to the limited available memory (64GB in our machine), Gurobi {runs out of memory for} instances with the resolution of $128\times128$ and beyond. As a consequence, we resize the images to $96\times96$ (using the {\sc Matlab} command \texttt{imresize}) and present average results (over 45 instances) for resolutions of $32\times 32$, $64\times64$, and $96\times96$ resolutions in Table \ref{resDOTmark}, denoted by {\tt ClassicImages32}, {\tt ClassicImages64}, and {\tt ClassicImages96}, respectively. From the results, we see that the ciPALM always returns similar objective function values (compared to Gurobi) with satisfactory feasibility accuracy in significantly less CPU time.}

\begin{table}[htb!]
    \caption{Computational results of the ciPALM (\textbf{left}) and the {\sc Snipal} \cite{li2020asymptotically} (\textbf{right}) on the classical OT problem with $m = n = 1000$ under different choices of tolerance parameters.}\label{tab:ot1000}
    \begin{minipage}{0.4\textwidth}
    \centering
        \begin{tabular}{|c|c|c|c|}
        \hline
        \multicolumn{4}{|c|}{ciPALM}  \\
        \hline
        $\rho$ & \# & \texttt{lin}\# & \texttt{time} \\ \hline
            8e-1  &  23 & 132 & 6.891  \\
            4e-1  &  23 & 129 & 6.698  \\
            1e-1  &  23 & 132 & 6.605  \\
            8e-2  &  23 & 133 & 6.653  \\
            4e-2  &  23 & 136 & 6.742  \\
            1e-2  &  23 & 140 & 6.739  \\
            8e-3  &  23 & 140 & 6.738  \\
            4e-3  &  23 & 142 & 6.807  \\
            1e-3  &  23 & 143 & 6.774  \\
            8e-4  &  23 & 144 & 6.772  \\
        \hline
        \end{tabular}
    \end{minipage}
    \begin{minipage}{0.55\textwidth}
    \centering
        \begin{tabular}{|c|c|c|c|c|c|c|c|}
        \hline
        \multicolumn{8}{|c|}{{\sc Snipal}}  \\
        \hline
        $p$ & $q$ & \multicolumn{3}{c|}{$\varepsilon_0=\delta_0=1$} &
        \multicolumn{3}{|c|}{$\varepsilon_0=\delta_0=10^{-3}$} \\
        \hline
              & & \# & \texttt{lin}\# & \texttt{time} & \# & \texttt{lin}\# & \texttt{time} \\ \hline
             1.1 & 1.1 &  23 & 132 & 6.648 & 23 & 144 & 7.011 \\
             1.1 & 2.1 &  23 & 137 & 6.748 & 23 & 145 & 7.072 \\
             1.1 & 3.1 &  23 & 138 & 6.696 & 23 & 145 & 7.120 \\
             2.1 & 1.1 &  23 & 132 & 6.554 & 23 & 144 & 7.020 \\
             2.1 & 2.1 &  23 & 137 & 6.749 & 23 & 145 & 7.020 \\
             2.1 & 3.1 &  23 & 138 & 6.615 & 23 & 145 & 7.015 \\
             3.1 & 1.1 &  23 & 132 & 6.370 & 23 & 144 & 7.002 \\
             3.1 & 2.1 &  23 & 137 & 6.507 & 23 & 145 & 7.001 \\
             3.1 & 3.1 &  23 & 138 & 6.727 & 23 & 145 & 7.024 \\
             \hline
        \end{tabular}
    \end{minipage}
\end{table}

\begin{table}[htb!]
    \caption{Computational results of the ciPALM (\textbf{left}) and the {\sc Snipal} \cite{li2020asymptotically} (\textbf{right}) on the classical OT problem with $m = n = 2000$ under different choices of tolerance parameters.}\label{tab:ot2000}
    \begin{minipage}{0.4\textwidth}
    \centering
        \begin{tabular}{|c|c|c|c|}
        \hline
        \multicolumn{4}{|c|}{ciPALM}  \\
        \hline
        $\rho$ & \# & \texttt{lin}\# & \texttt{time} \\ \hline
            8e-1  &  24 & 242 & 51.378  \\
            4e-1  &  24 & 232 & 46.798  \\
            1e-1  &  24 & 231 & 45.351  \\
            8e-2  &  24 & 231 & 45.378  \\
            4e-2  &  24 & 234 & 45.021  \\
            1e-2  &  24 & 237 & 45.240  \\
            8e-3  &  24 & 237 & 45.265  \\
            4e-3  &  24 & 242 & 45.878  \\
            1e-3  &  24 & 244 & 46.031  \\
            8e-4  &  24 & 244 & 46.036  \\
        \hline
        \end{tabular}
    \end{minipage}
    \begin{minipage}{0.55\textwidth}
    \centering
        \begin{tabular}{|c|c|c|c|c|c|c|c|}
        \hline
        \multicolumn{8}{|c|}{{\sc Snipal}}  \\
        \hline
        $p$ & $q$ & \multicolumn{3}{c|}{$\varepsilon_0=\delta_0=1$} &
        \multicolumn{3}{|c|}{$\varepsilon_0=\delta_0=10^{-3}$} \\
        \hline
              & & \# & \texttt{lin}\# & \texttt{time} & \# & \texttt{lin}\# & \texttt{time} \\ \hline
             1.1 & 1.1 &  24 & 229 & 45.368 & 24 & 245 & 46.091 \\
             1.1 & 2.1 &  24 & 234 & 45.577 & 24 & 246 & 46.118 \\
             1.1 & 3.1 &  24 & 236 & 45.543 & 24 & 246 & 46.193 \\
             2.1 & 1.1 &  24 & 229 & 45.065 & 24 & 245 & 46.021 \\
             2.1 & 2.1 &  24 & 234 & 45.506 & 24 & 246 & 46.062 \\
             2.1 & 3.1 &  24 & 236 & 45.621 & 24 & 246 & 46.209 \\
             3.1 & 1.1 &  24 & 229 & 45.098 & 24 & 245 & 46.007 \\
             3.1 & 2.1 &  24 & 234 & 45.486 & 24 & 246 & 46.145 \\
             3.1 & 3.1 &  24 & 236 & 45.649 & 24 & 246 & 45.927 \\
             \hline
        \end{tabular}
    \end{minipage}
\end{table}

\begin{table}[ht]
\caption{\blue{Comparisons between  ciPALM and  Gurobi for the classical optimal transport problem on images in the ``ClassicImages" class from the DOTmark collection. In the table, ``\texttt{nobj}" denotes the normalized objective function value, ``\texttt{feas}" denotes the primal feasibility accuracy, ``\texttt{iter}" denotes the number of iterations, where the total number of linear systems solved in ciPALM is given in the bracket, and ``\texttt{time}" denotes the computational time in seconds.}}\label{resDOTmark}
\centering \tabcolsep 12pt
\scalebox{1}
{\renewcommand\arraystretch{1}
\begin{tabular}[!]{llllll}
\hline
{\tt image} & {\tt method} & \texttt{nobj} & \texttt{feas} & \texttt{iter}
& \texttt{time} \\
\hline \vspace{-4mm}\\
{\tt ClassicImages32} & Gurobi & 0        & 3.08e-13 & 14       & 3.48    \\
                      & ciPALM & 2.77e-07 & 3.75e-07 & 19 (101) & 2.82   \\[5pt]
{\tt ClassicImages64} & Gurobi & 0        & 1.97e-13 & 15       & 80.91   \\
                      & ciPALM & 3.00e-07 & 1.40e-07 & 20 (112) & 47.81   \\[5pt]
{\tt ClassicImages96} & Gurobi & 0        & 1.42e-13 & 16       & 437.96    \\
                      & ciPALM & 3.08e-07 & 3.35e-08 & 21 (133) & 229.85   \\
\hline
\end{tabular}
}
\end{table}

\subsection{The martingale optimal transport problem}\label{section-MOT}

In this section, we evaluate the performance of our ciPALM in Algorithm \ref{algo:ciPALM} for solving the martingale optimal transport problem, i.e., problem \eqref{eq-regOTpro} with $\lambda_1 =\lambda_2=0$ under the constraint set \texttt{[T3]}. 
In our experiments, we follow \cite[Example 6.3]{alfonsi2019sampling} and \cite[Section 10]{hobson2012robust} in which two distributions $\bm{\mu}=\sum_{i=1}^{m}\frac{1}{m}\delta_{\bm{p}_i}$ and $\bm{\nu}'=\sum_{j=1}^{n'} \frac{1}{n'} \delta_{\bm{q}'_j}$  are sampled from 1-dimensional lognormal distribution $\mathtt{Lognormal}(0,0.1^2) $ and $\mathtt{Lognormal}(0,0.15^2)  $, respectively. Suggested by \cite{alfonsi2019sampling}, we consider $\bm{\nu}\coloneqq \bm{\mu} \vee \bm{\nu}'=\sum_{j=1}^{n} \beta_j \delta_{\bm{q}_j} $ calculated by \cite[Algorithm 1]{alfonsi2019sampling}, which satisfies $\bm{\mu} \leq_{cv} \bm{\nu}$\footnote{We say that $\bm{\mu} \leq_{cv} \bm{\nu}$ if for any convex function $\phi$, $\mathbb{E}_{x\sim \bm{\mu}}[\phi(x)] \leq \mathbb{E}_{y\sim \bm{\nu}}[\phi(y)]$, provided that both expectations exist.
Then, $\leq_{cv}$ defines a convex order, and the supremum $\bm{\mu}\vee\bm{\nu}$ of $\bm{\mu}$ and $\bm{\nu}$ can be defined so that $\bm{\mu}\vee\bm{\nu}$ is greater that $\bm{\mu}$ in this convex order. For more theoretical details and efficient scheme of computing $\bm{\mu}\vee\bm{\nu}$, we refer readers to \cite{alfonsi2019sampling}.
} so that the feasible set \texttt{[T3]} associated with $\bm{\mu}$ and $\bm{\nu}$ is nonempty. The cost matrix is obtained by setting $C_{ij}\coloneqq \abs{\bm{p}_i-\bm{q}_j}^{2.1}$ for any $1\leq i\leq m$ and $1\leq j \leq n$. We also set $\rho=0.01$ for our ciPALM to obtain overall competitive performances based on our numerical observations.

We then generate a set of synthetic problems with $m = n\in \{1000,\,2000,\dots,10000\}$. For each $m$, we generate 10 instances with different random seeds, and present the average numerical performances of our ciPALM and Gurobi in Figure~\ref{fig-mot}. Here, we mention that the termination tolerance for Gurobi is set to $10^{-6}$, which is same as the termination tolerance for our ciPALM. It can be observed that the primal feasibility accuracy and the normalized objective function value (using Gurobi as a bechmark) of our ciPALM are always at around or lower than the level of $10^{-6}$, suggesting that our ciPALM is able to solve the testing problems to a reasonable accuracy. Moreover, for large-scale problems, Gurobi can be rather time-consuming and memory-consuming. As an example, for the case where $m=n=10000$, a large-scale LP containing $10^8$ nonnegative variables and 30000 equality constraints was solved, and in this case, one can observe that Gurobi is around 5 times slower than our ciPALM. In addition, we have observed 
that Gurobi 
cannot solve the problems with $m = n \geq 11000$ in our PC due to the out-of-memory issue, while our ciPALM can handle much larger problems up to $m=n= 17000$.

\begin{figure}[htb!]
\begin{minipage}[!]{0.48\linewidth}
\centering
\hfill
\renewcommand{\arraystretch}{1.15}
\tabcolsep 5pt
\scalebox{0.9}{
\begin{tabular}[!]{|c|cc|cc|cc|}
\hline
{\tt problem} & \multicolumn{2}{c|}{\texttt{nobj}}
& \multicolumn{2}{c|}{\texttt{feas}} & \multicolumn{2}{c|}{\texttt{iter}} \\
\hline
$m=n$ & \texttt{g} & \texttt{c} & \texttt{g} & \texttt{c} & \texttt{g} & \texttt{c} \\
\hline
1000 & 0 & 3.8e-7 & 4.1e-12 & 8.6e-8 & 14 & 21 (83)  \\
2000 & 0 & 4.2e-7 & 4.5e-10 & 6.2e-8 & 15 & 20 (69)  \\
3000 & 0 & 3.7e-7 & 1.9e-11 & 7.9e-8 & 16 & 18 (68)  \\
4000 & 0 & 5.6e-7 & 1.5e-10 & 3.0e-8 & 16 & 20 (85)  \\
5000 & 0 & 7.9e-7 & 1.8e-11 & 5.9e-8 & 15 & 20 (93)  \\
6000 & 0 & 5.4e-7 & 5.7e-11 & 4.7e-8 & 15 & 19 (94)  \\
7000 & 0 & 6.2e-7 & 3.4e-12 & 4.4e-8 & 16 & 19 (93)  \\
8000 & 0 & 5.5e-7 & 1.7e-11 & 4.7e-8 & 18 & 19 (94)  \\
9000 & 0 & 7.7e-7 & 1.0e-10 & 4.5e-8 & 16 & 19 (96)  \\
10000& 0 & 1.0e-6 & 3.3e-12 & 4.2e-8 & 17 & 20 (97)  \\
\hline
\end{tabular} }
\end{minipage}
\quad
\begin{minipage}[!]{0.52\linewidth}
\centering
\includegraphics[width=1\textwidth]{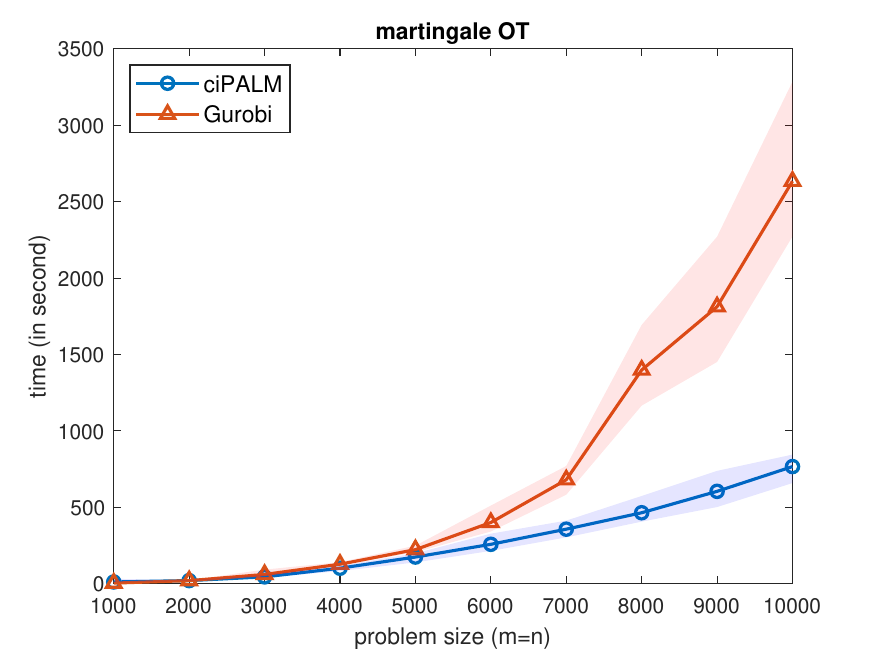}
\end{minipage}
\hfill\vspace{2mm}
\caption{Comparisons between the ciPALM (denoted by ``\texttt{c}") and the Gurobi (denoted by ``\texttt{g}") for the martingale optimal transport problem with $m = n\in \{1000,\,2000,\dots,10000\}$. \textbf{Left}: ``\texttt{nobj}" denotes the normalized objective function value, ``\texttt{feas}" denotes the primal feasibility accuracy, \blue{and ``\texttt{iter}" denotes the number of iterations, where the total number of linear systems solved in ciPALM is given in the bracket.}
\textbf{Right}: the average value and max-min range of the computational time. Note that Gurobi requires more memory than available resources in our experiments when $m=n\geq 11000$. }\label{fig-mot}
\end{figure}

\subsection{Group-quadratic regularized optimal transport problem}\label{sec:grp-quad-reg}

In this section, we evaluate the performance of our ciPALM in Algorithm \ref{algo:ciPALM} for solving the group-quadratic regularized optimal transport problem, i.e., problem \eqref{eq-regOTpro} with $\lambda_1>0$ and $\lambda_2>0$ subject to the constraint set \texttt{[T1]}. Here, we set $\rho=0.01$ for our ciPALM as Section \ref{section-MOT} to obtain overall competitive performances based on our numerical observations.



We follow \cite[Section 5.1]{courty2014domain} to generate two distributions $\bm{\mu}=\sum_{i=1}^{m} \frac{1}{m} \delta_{\bm{p}_i}$ and $\bm{\nu}=\sum_{j=1}^{n} \frac{1}{n} \delta_{\bm{q}_j}$ in $\mathbb{R}^2$ as follows. First, we choose $1\leq m_1 < m$. Then, $\bm{p}_i$ is sampled from the normal distribution $\mathtt{Normal}((-1;2), \,0.25I_2)$ if $1\leq i\leq m_1$, and is sampled from the normal distribution $\mathtt{Normal}((1;2), \,0.25I_2)$ otherwise. The associated binary label vector $\bm{\ell}^P \in \{0, 1\}^m$ is defined by $\bm{\ell}^P_i=0$ if  $1\leq i\leq m_1$, and $\bm{\ell}^P_i=1$ otherwise. In addition, $\bm{q}_j$ is sampled from the mixture Gaussian distribution defined as $\frac{1}{2}\mathtt{Normal}((-2;2), 0.5I_2) + \frac{1}{2} \mathtt{Normal}((2;3), 0.5I_2)$. Second, the group structure $\mathcal{G}$ on the variable $X\in\mathbb{R}^{m\times n}$ is defined as a partition of the indexes set $\{(1,1),(1,2),\dots,(m,n)\} $ so that $(i,j)$ and $(i',j')$ are assigned to the same group if $j=j'$ and $\bm{\ell}^{P}_i = \bm{\ell}^{P}_{i'}$. Last, the cost matrix $C\in \mathbb{R}^{m\times n}$ is obtained by setting $C_{ij}:= \norm{\bm{p}_i - \bm{q}_j}^2$ for any $1\leq i\leq m$ and $1\leq j \leq n$.

An illustration of the data set and the corresponding numerical solutions when $m=n=200,\;m_1=100$ are displayed in Figure~\ref{FigGQOTdata}, where $\{\bm{p}_i \}_{i=1}^{m_1}$, $\{\bm{p}_i \}_{i=m_1+1}^{200}$, and $\{\bm{q}_j \}_{j=1}^{n} $ are marked by red-dot, blue-dot and black-cross, respectively. In domain adaptation application, the goal is to obtain labels for the target domain (i.e.,$\{\bm{q}_j \}_{j=1}^{n}$) with the information from a labeled source (i.e., two clusters $\{\bm{p}_i \}_{i=1}^{m_1}$ and $\{\bm{p}_i \}_{i=m_1+1}^{m}$). Given a valid transport plan $X$, one may follow \cite[Section 4.3]{courty2014domain} to generate a set of labeled data points, denoted by $\left\{\tilde{\bm{p}}_i\right\}_{i = 1}^m$, on the target domain, where $\tilde{\bm{p}}_i:=\frac{\sum_{j=1}^{n} X_{i,j} \bm{q}_j }{\sum_{j=1}^{n} X_{i,j} }$ which is assigned with the same label as $\bm{p}_i$, for all $i = 1,\dots, m$. Then, one can train a machine learning model (such as a supervised learning model) by using the generated labeled dataset on the target domain to predict the labels for the dataset $\{\bm{q}_j \}_{j=1}^{n}$. Therefore, a transport plan $X$ that is able to leverage the label information of the source domain will be more appealing.

In Figure~\ref{FigGQOTdata}, we present $X$ and $\left\{\tilde{\bm{p}}_i\right\}_{i = 1}^m$ obtained from solving the classical unregularized OT problem (i.e., $\lambda_1=\lambda_2=0$), and the group-quadratic regularized problem with $\lambda_1=\lambda_2=1 $ in the middle and right sub-figures, respectively. In both figures, a red/blue arrow shows the transportation between $\bm{p}_i$ and $\tilde{\bm{p}}_i$, for all $i = 1,\dots, m$. We observe that when $\lambda_1=\lambda_2=0$, the set $\left\{\tilde{\bm{p}}_i\right\}_{i = 1}^m$ is in fact a permutation of $\left\{\bm{q}_j \right\}_{j=1}^{n}$.
However, it is clear that the solution $X$ in this case only depends on the cost matrix $C$ but does not depend on $\bm{\ell}^P$. Consequently, $\left\{\tilde{\bm{p}}_i\right\}_{i = 1}^m$ may not incorporate the label information from the source domain. Indeed, it can be seen from some red and blue dots located inside the highlighted box that the nearby points of $\bm{p}_i$ in the source domain are mapped to the nearby points of $\tilde{\bm{p}}_i$ in the target domain, regardless their labels. On the other hand, when $\lambda_1=\lambda_2=1$, one can see that these mismatching behaviors are alleviated, in the sense that points with different labels are now mapped along distinguished directions.
This phenomenon has also been observed in \cite[Figure 4]{courty2014domain} which employs a group-entropic regularizer. Note that while a group-entropic regularizer will lead to a fully dense transportation plan, a group-quadratic regularizer promotes appealing group sparsity, as indicated in Figure \ref{FigGQOTdata}.

We next generate a set of synthetic problems
with $m = n\in \{500,\,1000,\dots,\,3500\}$. For each $m$, we generate 10 instances with different random seeds, and present the average numerical performance of our ciPALM and Mosek in Figure~\ref{fig-gqot}, where the termination tolerance for the Mosek is set to $10^{-6}$ as our ciPALM. One can observe a similar behavior as in the previous subsection on the martingale OT problems. Specifically, our ciPALM always returns solutions with comparable quality as Mosek. Moreover, our ciPALM is able to solve all instances within 300 seconds, which is usually 5 to 8 times faster than Mosek. On the other hand, by using the SOCP reformulation, we observe that Mosek requires much more computational resource including the memory usage than that used by the ciPALM. In fact, Mosek is not able to solve problems with $m = m\geq 4000$ due to the out-of-memory issue while our ciPALM can handle much larger problems. This indicates the advantages of our ciPALM for solving large-scale problems that often appear in practical applications such as domain adaption \cite{courty2014domain,courty2016optimal,redko2019optimal} and activity recognition \cite{lu2021cross}.

\begin{figure}[htb!]
\centering
\includegraphics[width=0.32\textwidth]{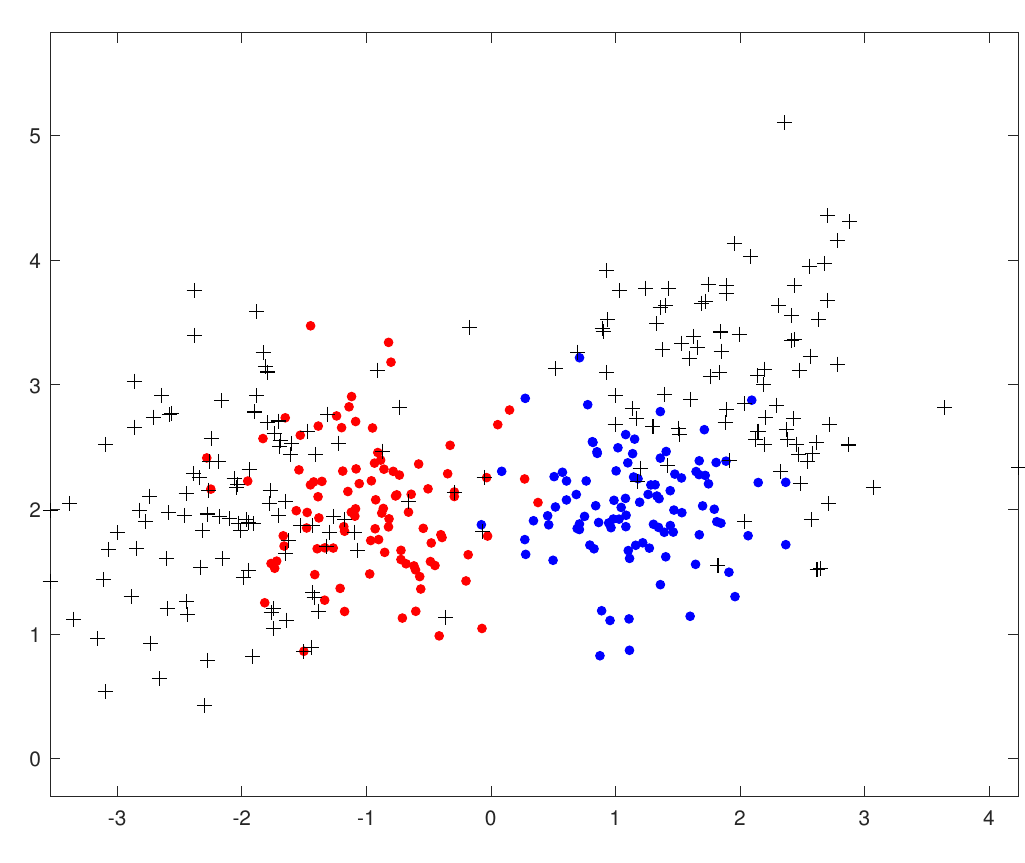}~
\includegraphics[width=0.32\textwidth]{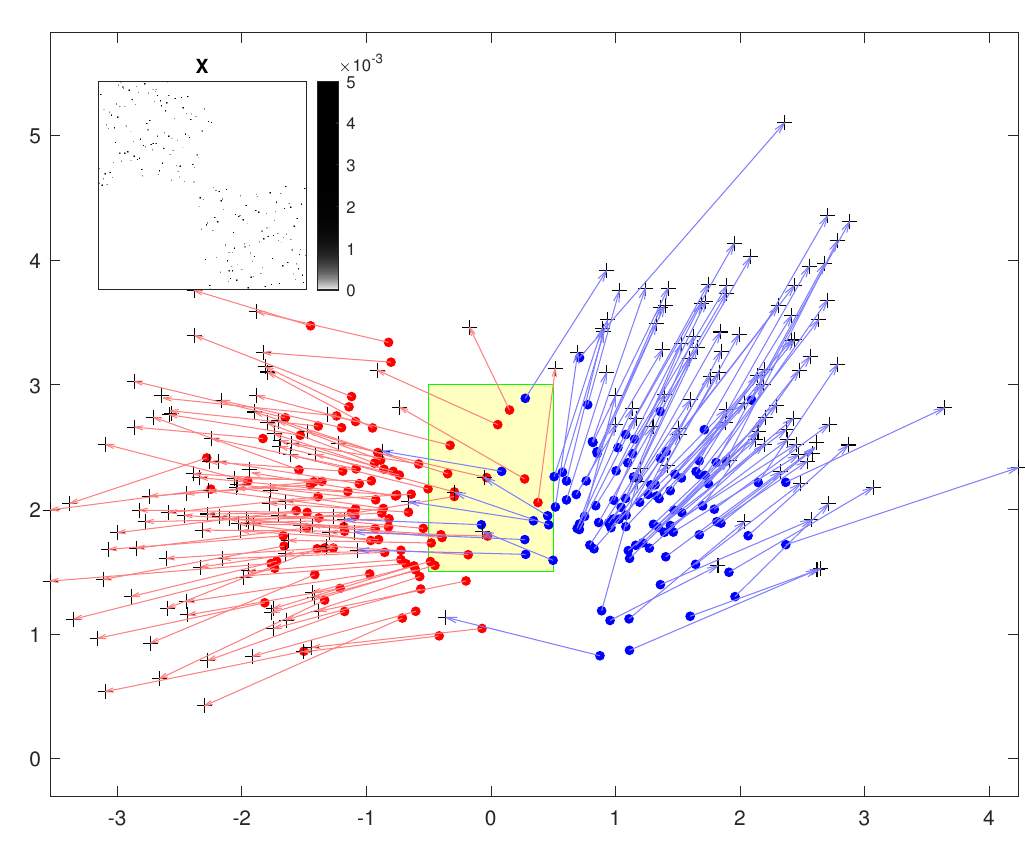}~
\includegraphics[width=0.32\textwidth]{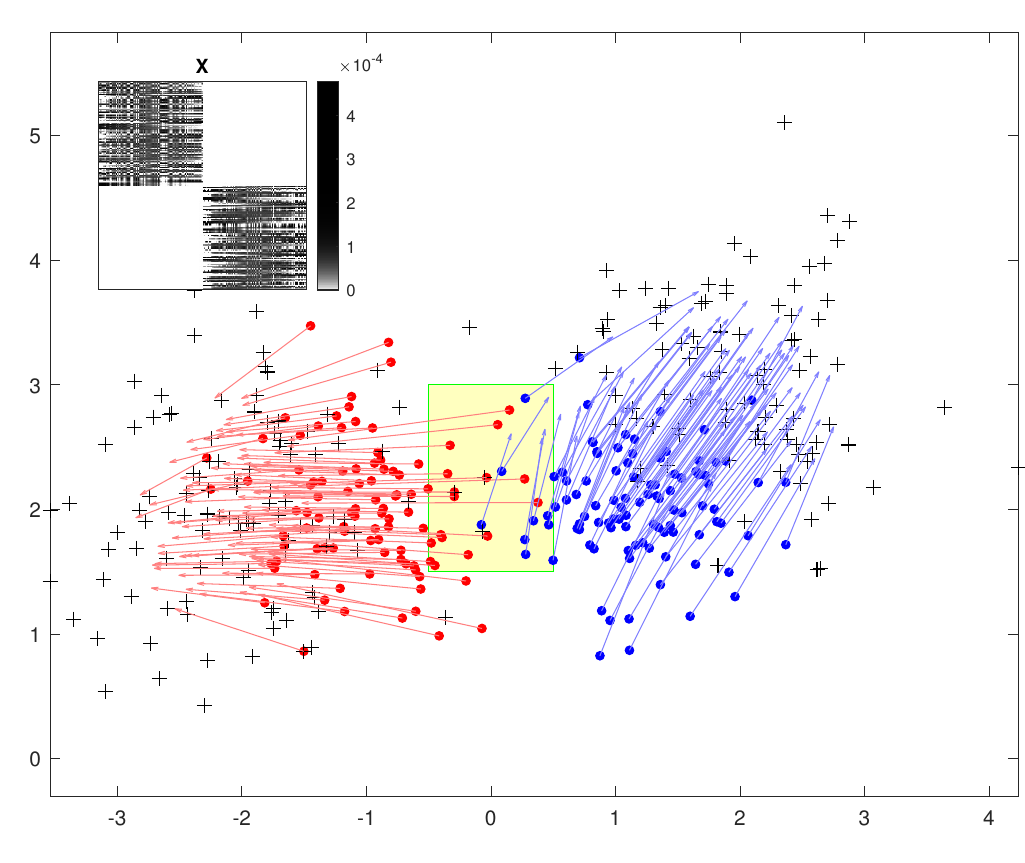}
\caption{An illustration of the data set (\textbf{left}) and numerical solutions (\textbf{middle} and \textbf{right}) when \(\lambda_1=\lambda_2=0\) and  \(\lambda_1=\lambda_2=1\). }\label{FigGQOTdata}
\end{figure}

\begin{figure}[htb!]
\begin{minipage}[!]{0.48\linewidth}
\centering
\hfill
\renewcommand{\arraystretch}{1.3}
\tabcolsep 5pt
\scalebox{0.95}{ 
\begin{tabular}[!]{|c|cc|cc|cc|}
\hline
{\tt problem} & \multicolumn{2}{c|}{\texttt{nobj}} 
& \multicolumn{2}{c|}{\texttt{feas}} & \multicolumn{2}{c|}{\texttt{iter}} \\
\hline
$m=n$ & \texttt{m} & \texttt{c} & \texttt{m} & \texttt{c} 
& \texttt{m} & \texttt{c}  \\ \hline
\multicolumn{7}{|c|}{$\lambda_1=\lambda_2=1$} \\
\hline
500  & 0 & 1.4e-5 & 3.0e-8 & 4.2e-7 & 14 & 10 (36) \\
1000 & 0 & 1.3e-5 & 1.4e-8 & 3.8e-7 & 16 & 14 (48) \\
1500 & 0 & 1.9e-5 & 1.3e-8 & 2.7e-7 & 16 & 12 (40) \\
2000 & 0 & 1.9e-5 & 1.0e-8 & 2.9e-7 & 17 & 12 (39) \\
2500 & 0 & 3.0e-5 & 1.3e-8 & 2.4e-7 & 17 & 12 (39) \\
3000 & 0 & 3.9e-5 & 1.4e-8 & 3.4e-7 & 18 & 14 (50) \\
3500 & 0 & 3.6e-5 & 1.1e-8 & 2.3e-7 & 17 & 14 (52) \\
\hline
\multicolumn{7}{|c|}{$\lambda_1=\lambda_2=0.1$} \\
\hline
500  & 0 & 4.4e-5 & 1.4e-7 & 4.4e-7 & 14 & 14 (42) \\
1000 & 0 & 7.6e-5 & 1.2e-7 & 3.5e-7 & 18 & 14 (42) \\
1500 & 0 & 6.5e-5 & 6.5e-8 & 2.6e-7 & 19 & 13 (41) \\
2000 & 0 & 1.0e-4 & 8.0e-8 & 3.1e-7 & 19 & 13 (38) \\
2500 & 0 & 1.1e-4 & 6.7e-8 & 2.9e-7 & 22 & 13 (40) \\
3000 & 0 & 1.0e-4 & 5.1e-8 & 3.1e-7 & 23 & 14 (42) \\
3500 & 0 & 1.1e-4 & 4.6e-8 & 2.0e-7 & 23 & 14 (44) \\
\hline
\end{tabular} }
\end{minipage}
\quad
\begin{minipage}[!]{0.52\linewidth}
\centering
\includegraphics[width=0.9\textwidth]{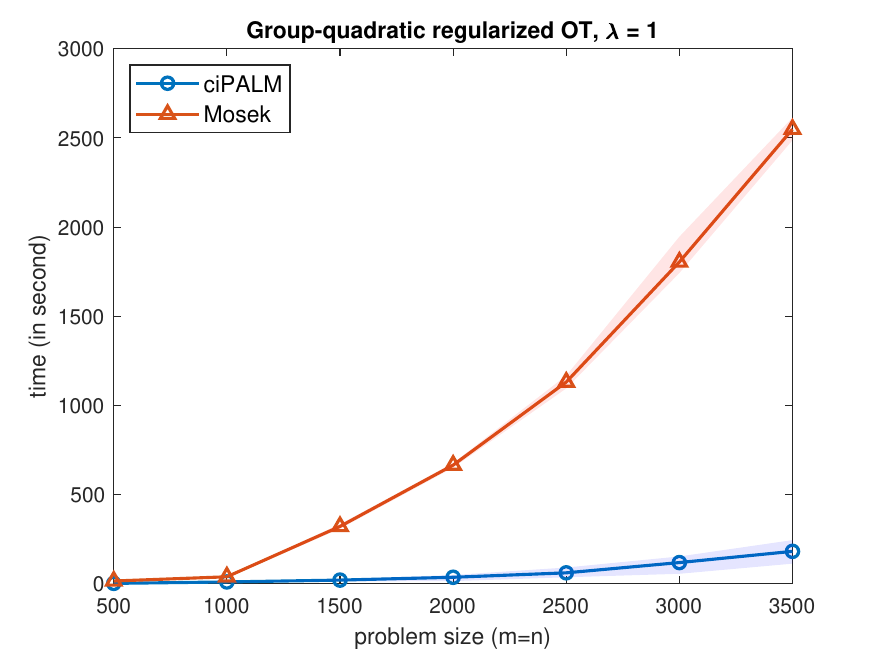}
\includegraphics[width=0.9\textwidth]{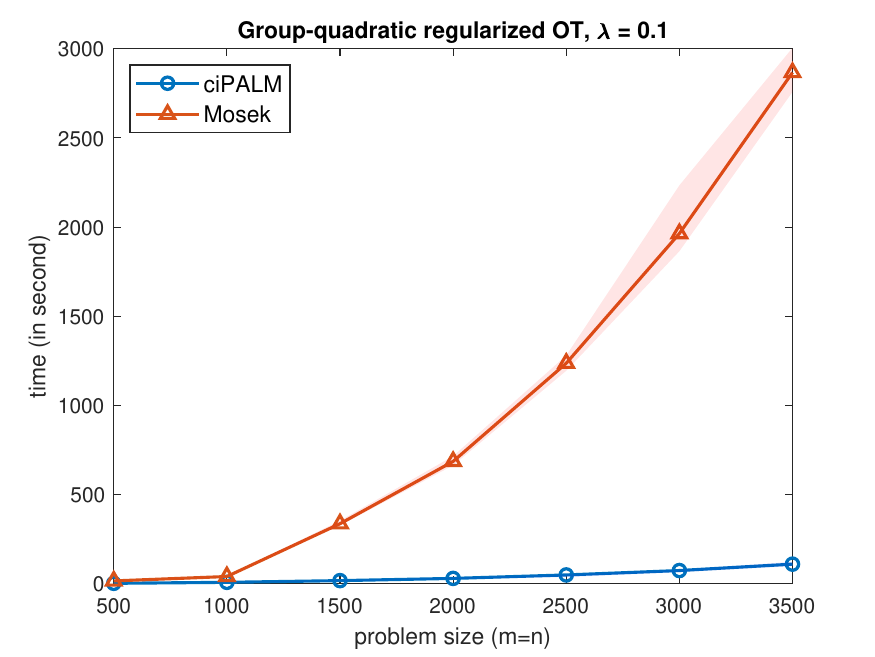}
\end{minipage}
\hfill\vspace{2mm}
\caption{Comparisons between the ciPALM (denoted by ``\texttt{c}") and the Mosek (denoted by ``\texttt{m}") for the group quadratic regularized optimal transport problem with $m = n\in \{500,\,1000,\dots,\,3500\}$. \textbf{Left}: ``\texttt{nobj}" denotes the normalized objective function value (use Mosek as a benchmark), ``\texttt{feas}" denotes the primal feasibility accuracy, \blue{and ``\texttt{iter}" denotes the number of iterations, where the total number of linear systems solved in ciPALM is given in the bracket.} \textbf{Top-right \& bottom-right}: the average value and max-min range of computational time for $\lambda=1$ and $\lambda=0.1$. Note that Mosek requires more memory than the available resources in our experiments when $m=n\geq 4000$.}\label{fig-gqot}
\end{figure}

\blue{
\begin{remark}
Based on the numerical experiments above, we have observed that interior-point-based methods, such as Gurobi and Mosek, exhibit slower performance and consume much higher memory compared to our ciPALM. This is possibly due to the difference in the efficiency in constructing and solving the involved linear systems between an interior-point-based method and our ciPALM. For instance, for the classical OT problem, suppose that the linear constraint is written as $A\mathrm{vec}(X) = \bm{b}$ where $A\in \mathbb{R}^{(m+n)\times (mn)}$ and $\bm{b}\in\mathbb{R}^{m+n}$. Then, in each iteration of an interior-point-based method, one needs to construct a coefficient matrix of the form  $A\mathrm{Diag}(\bm{d})A^{\top}$ with $\bm{d}\in\mathbb{R}_{++}^{mn}$, where all entries of $\bm{d}$ are positive. Such a coefficient matrix is typically dense, and more significantly, could be highly ill-conditioned. Thus, when $m$ and $n$ are large, the commonly employed approaches such as the Cholesky factorization and the conjugate gradient method would become inefficient or require substantial computational resources for solving the linear system. In contrast, the coefficient matrix of the linear system arising from our ciPALM is in the form of $A\mathrm{Diag}(\bm{\widehat{d}})A^{\top} + \tau I$ with $\bm{\widehat{d}}\in\mathbb{R}^{mn}_+$ and $\tau > 0$ (this can be seen from the construction of $\widehat{\partial}(\nabla\Psi)(\cdot)$ in Proposition \ref{prop-ss}). Here, $\bm{\widehat{d}}$ can have zero entries, and in fact, could be quite sparse in practical computation. Thus, by fully leveraging this sparsity structure (referred to as the ``second-order sparsity'' of the underlying problem), the cost of constructing the coefficient matrix or performing the matrix-vector multiplication can be significantly reduced. Moreover, the presence of $\tau I$ with proper choices of $\tau$ makes the coefficient matrix more well-conditioned. This further facilitates the computation of solving the linear system. More discussions on how to efficiently solve such kind of linear systems arising from the semismooth Newton method can be found in \cite[Section 4]{li2020asymptotically}. In addition, we would also like to mention that, although our ciPALM takes advantage of many efficient built-in functions (e.g., matrix multiplication and addition) in {\sc Matlab} that can be executed on multiple computational threads, we believe that there is still ample room for improving our ciPALM with a dedicated parallel implementation on a suitable computing platform other than {\sc Matlab}. But we will leave this topic as future research.

\end{remark}
}

\section{Conclusions}\label{sec:conc}

In this paper, we considered a class of group-quadratic regularized OT problems whose solutions are promoted to have special structures. To solve this class of problems, we proposed a corrected inexact proximal augmented Lagrangian method (ciPALM) whose subproblems are solved by the semi-smooth Newton method. The proposed method can be shown to admit appealing convergence properties under mild conditions. Moreover, different from the recent semismooth Newton based inexact proximal augmented Lagrangian ({\sc Snipal}) method, wherein a summable tolerance parameter sequence should be specified for practical implementations, our ciPALM employed a relative error criterion for the approximate minimization of the subproblem, wherein only a single tolerance parameter 
is needed and thus can be more friendly to tune from the computational and implementation perspectives. Numerical results illustrated the efficiency of the proposed method for solving large-scale problems.

There remain some problems that
open our future investigations. First, when $\lambda_1 > 0$, whether or not the operator $\mathcal{T}_\ell$ satisfies the error condition in Assumption \ref{assmp-errbdweak} needs more advanced tools and further studies. Second, we observed from our numerical experiments that, if the relative error condition in \eqref{inexcond-ciALM} is used for terminating the ALM subproblem but the corrected step in \eqref{extrastep-ciALM} is dropped in the proximal ALM framework, the algorithm can still converge empirically and perform very well. However, for the time being, the corrected step is still needed for the convergence analysis. This brings a gap between the theoretical analysis and the practical performance. Hence, more advanced tools are needed to close this gap and to get a better understanding of the inexact proximal ALM framework and its variants. Last but not least, the values of the regularization parameters $\lambda_1$ and $\lambda_2$ would affect the sparsity of the optimal solution for the group-quadratic regularized OT problem. To further improve the efficiency of the proposed framework, the ideas of dimension reduction and adaptive sieving studied in \cite{yuan2022dimension,yuan2023adaptive} may be employed as a future research topic. \blue{In addition, it is also interesting to extend our algorithm to some other important variants of the OT problem such as the multi-marginal OT problem; see, for example, \cite{clty2023efficient,gs1998optimal,lhcj2022on,pass2015multi}. But it would require additional effort to identify and leverage the underlying structures to achieve higher efficiency. We will leave it as another possible future research {project.}  }

\section*{Acknowledgments}

\blue{We thank the editor and referees for their valuable suggestions and comments, which have helped to improve the quality of this paper.}

\appendix

\section{Missing proofs in Section \ref{sec:VHPE}}\label{sec-appendix-vhpecon}

\begin{proof}[Proof of Theorem \ref{thm-VHPEcon}]
\textbf{Statement (i)}. For any $\bm{x}^*\in\Omega$, one can see that
\begin{equation*}
\begin{aligned}
&\quad \|\bm{x}^{k+1}-\bm{x}^*\|^2_{M_k^{-1}} - \|\bm{x}^{k}-\bm{x}^*\|^2_{M_k^{-1}} \\
&= \|\bm{x}^{k+1}-\widetilde{\bm{x}}^{k+1}+\widetilde{\bm{x}}^{k+1}-\bm{x}^*\|^2_{M_k^{-1}}
- \|\bm{x}^{k}-\widetilde{\bm{x}}^{k+1}+\widetilde{\bm{x}}^{k+1}-\bm{x}^*\|^2_{M_k^{-1}} \\
&= \|\widetilde{\bm{x}}^{k+1}-\bm{x}^{k+1}\|^2_{M_k^{-1}}
- 2\langle M_k^{-1}(\bm{x}^{k+1}-\bm{x}^{k}),\,\bm{x}^*-\widetilde{\bm{x}}^{k+1}\rangle
- \|\widetilde{\bm{x}}^{k+1}-\bm{x}^{k}\|^2_{M_k^{-1}} \\
&= \|c_kM_k\bm{d}^{k+1} + \widetilde{\bm{x}}^{k+1} - \bm{x}^k\|^2_{M_k^{-1}}
- 2 c_k\langle -\bm{d}^{k+1},\,\bm{x}^*-\widetilde{\bm{x}}^{k+1}\rangle
- \|\bm{x}^{k}-\widetilde{\bm{x}}^{k+1}\|^2_{M_k^{-1}}  \\
&\leq \|c_kM_k\bm{d}^{k+1} + \widetilde{\bm{x}}^{k+1} - \bm{x}^k\|^2_{M_k^{-1}}
+ 2 c_k\varepsilon_{k+1}
- \|\widetilde{\bm{x}}^{k+1}-\bm{x}^{k}\|^2_{M_k^{-1}}  \\
&\leq -(1-\rho^2)\|\widetilde{\bm{x}}^{k+1}-\bm{x}^{k}\|^2_{M_k^{-1}},
\end{aligned}
\end{equation*}
where the third equality follows from $\bm{x}^{k+1}=\bm{x}^k - c_kM_k\bm{d}^{k+1}$, the first inequality follows from $\langle -\bm{d}^{k+1},\,\bm{x}^*-\widetilde{\bm{x}}^{k+1}\rangle\geq-\varepsilon_{k+1}$ since  $\bm{d}^{k+1} \in \mathcal{T}^{\varepsilon_{k+1}}(\widetilde{\bm{x}}^{k+1})$ and $0\in\mathcal{T}(\bm{x}^*)$, and the last inequality follows from condition \eqref{VHPE-inexcond}. Since $\frac{1}{1+\eta_k}M_k\preceq M_{k+1}$, we know that $M_{k+1}^{-1}\preceq(1+\eta_k)M_k^{-1}$. This together with the above inequality implies that, for any $\bm{x}^*\in\Omega$,
\begin{equation}\label{recurineq}
\begin{aligned}
\|\bm{x}^{k+1}-\bm{x}^*\|^2_{M_{k+1}^{-1}}
&\leq (1+\eta_k)\|\bm{x}^{k+1}-\bm{x}^*\|^2_{M_{k}^{-1}}  \\
&\leq (1+\eta_k)\|\bm{x}^{k}-\bm{x}^*\|^2_{M_k^{-1}}
- (1+\eta_k)(1-\rho^2)\|\widetilde{\bm{x}}^{k+1}-\bm{x}^{k}\|^2_{M_k^{-1}} \\
&\leq (1+\eta_k)\|\bm{x}^{k}-\bm{x}^*\|^2_{M_k^{-1}}.
\end{aligned}
\end{equation}
Since $\{\eta_k\}$ is a nonnegative summable sequence, it then follows from the \cite[Lemma 2 in Section 2.2]{p1987introduction} that $\big\{\|\bm{x}^{k}-\bm{x}^*\|^2_{M_k^{-1}}\big\}$ is convergent, and hence there exits some $\mu\geq0$ such that
\begin{equation}\label{limzz}
\lim\limits_{k\to\infty} \|\bm{x}^{k}-\bm{x}^*\|_{M_k^{-1}} = \mu.
\end{equation}
Thus, $\{\bm{x}^k\}$ is bounded since $\lambda_{\max}(M_k)\leq\overline{\lambda}$ for all $k\geq0$.

\textbf{Statement (ii)}. Let $\Pi_{\Omega}(\bm{x})$ denote the projection of $\bm{x}$ onto $\Omega$. It is clear that $0\in\mathcal{T}(\Pi_{\Omega}(\bm{x}^k))$. Then, we get from \eqref{recurineq} (by setting $\bm{x}^*=\Pi_{\Omega}(\bm{x}^k)$) that
\[
    \begin{aligned}
    \mathrm{dist}_{M_{k+1}^{-1}}(\bm{x}^{k+1},\Omega)
    &\leq \|\bm{x}^{k+1}-\Pi_{\Omega}(\bm{x}^k)\|_{M_{k+1}^{-1}} \\
    &\leq (1+\eta_k)\|\bm{x}^{k}-\Pi_{\Omega}(\bm{x}^k)\|^2_{M_k^{-1}}  \\
    &= (1+\eta_k)\mathrm{dist}_{M_{k}^{-1}}(\bm{x}^{k},\Omega).
    \end{aligned}
\]

\textbf{Statement (iii)}. From \eqref{recurineq} and $\eta_k\geq0$, we have
\[
    0\leq(1-\rho^2)\|\widetilde{\bm{x}}^{k+1}-\bm{x}^{k}\|^2_{M_k^{-1}}
    \leq (1+\eta_k)\|\bm{x}^{k}-\bm{x}^*\|^2_{M_k^{-1}} - \|\bm{x}^{k+1}-\bm{x}^*\|^2_{M_{k+1}^{-1}}.
\]
This, together with the convergence of $\big\{\|\bm{x}^{k}-\bm{x}^*\|^2_{M_k^{-1}}\big\}$, $\eta_k\to0$, $0\leq\rho<1$, and $\lambda_{\max}(M_k)\leq\overline{\lambda}$, implies that $\lim_{k\to\infty}\|\widetilde{\bm{x}}^{k+1}-\bm{x}^{k}\|=0$. Moreover, since $c_k\geq c>0$ and $\lambda_{\min}(M_k)\geq\underline{\lambda}>0$ for all $k\geq0$, we then get from \eqref{VHPE-inexcond} that $\lim_{k\to\infty}\|c_kM_k\bm{d}^{k+1} + \widetilde{\bm{x}}^{k+1} - \bm{x}^k\|=0$ and $\lim_{k\to\infty}\varepsilon_{k+1}=0$.
Note also that
$
    \begin{aligned}
    c\underline{\lambda}\|\bm{d}^{k+1}\| \leq \|c_kM_k\bm{d}^{k+1}\| \leq \|c_kM_k\bm{d}^{k+1} + \widetilde{\bm{x}}^{k+1} - \bm{x}^k\| + \|\widetilde{\bm{x}}^{k+1} - \bm{x}^k\|.
    \end{aligned}
$
Thus, we have $\lim_{k\to\infty}\|\bm{d}^{k+1}\|=0$.

\textbf{Statement (iv)}. Since $\{\bm{x}^k\}$ is bounded, it then has at least one cluster point. Suppose that $\bm{x}^{\infty}$ is a cluster point and $\{\bm{x}^{k_i}\}$ is a convergent subsequence such that $\lim_{i\to\infty}\bm{x}^{k_i}=\bm{x}^{\infty}$. Since $\lim_{k\to\infty}\|\widetilde{\bm{x}}^{k+1}-\bm{x}^{k}\|=0$, we also have $\lim_{i\to\infty}\widetilde{\bm{x}}^{k_i+1}=\bm{x}^{\infty}$. Recall from condition \eqref{VHPE-inexcond} that $\bm{d}^{k+1} \in \mathcal{T}^{\varepsilon_{k+1}}(\widetilde{\bm{x}}^{k+1})$. Then, for any $\bm{x}\in\mathbb{R}^{\ell}$ and $\bm{u}\in\mathcal{T}(\bm{x})$, we have $\langle \bm{u}-\bm{d}^{k_i+1}, \,\bm{x}-\widetilde{\bm{x}}^{k_i+1}\rangle \geq -\varepsilon_{k_i+1}$. Hence,
\[
    \langle \bm{u}-0, \,\bm{x}-\widetilde{\bm{x}}^{k_i+1}\rangle
    \geq \langle \bm{d}^{k_i+1}, \,\bm{x}-\widetilde{\bm{x}}^{k_i+1}\rangle
    -\varepsilon_{k_i+1}.
\]
Since $\lim_{i\to\infty}\widetilde{\bm{x}}^{k_i+1}=\bm{x}^{\infty}$, $\lim_{k\to\infty}\|\bm{d}^{k+1}\|=0$, and $\lim_{k\to\infty}\varepsilon_{k+1}=0$, by passing to the limit when $i\to\infty$, we obtain that
\[
    \langle \bm{u}-0, \,\bm{x}-\bm{x}^{\infty}\rangle \geq 0, \quad \forall\,\bm{u},\bm{x} ~\mbox{satisfying}~\bm{u}\in\mathcal{T}(\bm{x}).
\]
From the maximal monotonicity of $\mathcal{T}$, we know that $0\in\mathcal{T}(\bm{x}^{\infty})$. Now, by replacing $\bm{x}^*$ in \eqref{limzz} by $\bm{x}^{\infty}$, we can readily obtain that
$
    \lim\limits_{k\to\infty} \|\bm{x}^{k}-\bm{x}^{\infty}\|_{M_k^{-1}} = 0.
$
This thus implies that $\{\bm{x}^k\}$ converges to $\bm{x}^{\infty}$ since $\lambda_{\max}(M_k)\leq\overline{\lambda}$, and completes the proof.
\end{proof}


Henceforth, for all $k\geq0$, we let $\mathcal{P}_k:=(\mathcal{I}+c_kM_k\mathcal{T})^{-1}$ and $\mathcal{Q}_k:=\mathcal{I}-\mathcal{P}_k$. Since $\mathcal{I}+c_kM_k\mathcal{T}$ is a strongly monotone operator, it follows from
\cite[Proposition 12.54]{rw1998variational} that $\mathcal{P}_k$ is single-valued. Thus, $\mathcal{P}_{k}(\bm{x}^k)$ is the unique solution of the subproblem \eqref{VHPE-subpro}. One can also show that
\begin{equation*}
0\in\mathcal{T}(\bm{x}) ~\Longleftrightarrow~
\mathcal{P}_k(\bm{x})=\bm{x} ~\Longleftrightarrow~
\mathcal{Q}_k(\bm{x})=0.
\end{equation*}
Moreover, we summarize some properties of $\mathcal{P}_k$ and $\mathcal{Q}_k$ in the following proposition, whose proofs are similar to those of
\cite[Proposition 1]{r1976monotone}.

\begin{proposition}\label{prop-PQ}
For all $k\geq0$, it holds that
\begin{itemize}
\item[{\rm (a)}] $\bm{x}=\mathcal{P}_k(\bm{x})+\mathcal{Q}_k(\bm{x})$ and $c_k^{-1}M_k^{-1}\mathcal{Q}_k(\bm{x})\in\mathcal{T}(\mathcal{P}_k(\bm{x}))$ for all $\bm{x}\in\mathbb{R}^{\ell}$;
\item[{\rm (b)}] $\langle \mathcal{P}_k(\bm{x})-\mathcal{P}_k(\bm{x}'),
    \,\mathcal{Q}_k(\bm{x})-\mathcal{Q}_k(\bm{x}')\rangle_{M_k^{-1}}\geq0$ for all $\bm{x},\,\bm{x}'\in\mathbb{R}^{\ell}$;
\item[{\rm (c)}] $\|\mathcal{P}_k(\bm{x})-\mathcal{P}_k(\bm{x}')\|_{M_k^{-1}}^2
    +\|\mathcal{Q}_k(\bm{x})-\mathcal{Q}_k(\bm{x}')\|_{M_k^{-1}}^2
    \leq\|\bm{x}-\bm{x}'\|_{M_k^{-1}}^2$ for all $\bm{x},\,\bm{x}'\in\mathbb{R}^{\ell}$.
\end{itemize}
\end{proposition}

We are now ready to give the proof of Theorem \ref{thm-VHPEconrate}.

\begin{proof}[Proof of Theorem \ref{thm-VHPEconrate}.]
By applying \eqref{ineqzomega} consecutively, we have that, for all $k\geq0$,
\[
    \mathrm{dist}_{M_{k}^{-1}}(\bm{x}^{k},\Omega)
    \leq {\textstyle\prod^{k-1}_{i=0}} (1+\eta_i)\,\mathrm{dist}_{M_0^{-1}}(\bm{x}^0,\Omega)
    \leq {\textstyle\prod^{\infty}_{i=0}} (1+\eta_i)\,\mathrm{dist}_{M_0^{-1}}(\bm{x}^0,\Omega).
\]
Moreover, for all $k\geq0$,
\[
    \begin{aligned}
    \mathrm{dist}_{M_{k}^{-1}}(\mathcal{P}_{k}(\bm{x}^k),\Omega)
    &\leq \|\mathcal{P}_{k}(\bm{x}^k)-\Pi_{\Omega}(\bm{x}^k)\|_{M_{k}^{-1}}
    = \|\mathcal{P}_{k}(\bm{x}^k)-\mathcal{P}_k(\Pi_{\Omega}(\bm{x}^k))\|_{M_{k}^{-1}} \\
    &\leq \|\bm{x}^k-\Pi_{\Omega}(\bm{x}^k)\|_{M_{k}^{-1}}
    = \mathrm{dist}_{M_{k}^{-1}}(\bm{x}^{k},\Omega),
    \end{aligned}
\]
where the first equality follows from $\mathcal{P}_k(\Pi_{\Omega}(\bm{x}^k))=\Pi_{\Omega}(\bm{x}^k)$ since $\Pi_{\Omega}(\bm{x}^k)\in\mathcal{T}^{-1}(0)$. Then, from the above two inequalities, $\lambda_{\max}(M_k)\leq\overline{\lambda}$, and $\prod^{\infty}_{i=0}(1+\eta_i)<\infty$ (since $\{\eta_k\}$ is a nonnegative summable sequence), it holds that for all $k\geq0$
\[
    \mathrm{dist}(\mathcal{P}_{k}(\bm{x}^k),\Omega)
    \leq \sqrt{\overline{\lambda}}\,\mathrm{dist}_{M_{k}^{-1}}(\mathcal{P}_{k}(\bm{x}^k),\Omega)
    \leq \sqrt{\overline{\lambda}}\,{\textstyle\prod^{\infty}_{i=0}} (1+\eta_i)\mathrm{dist}_{M_0^{-1}}(\bm{x}^0,\Omega)
    <\infty.
\]
Note from Proposition \ref{prop-PQ}(a) that $c_k^{-1}M_k^{-1}\mathcal{Q}_k(\bm{x}^k)\in\mathcal{T}(\mathcal{P}_k(\bm{x}^k))$. Thus, we apply Assumption \ref{assmp-errbdweak} with $r:=\sqrt{\overline{\lambda}}\prod^{\infty}_{i=0} (1+\eta_i)\mathrm{dist}_{M_0^{-1}}(\bm{x}^0,\Omega)$ and know that, there exists
a $\kappa>0$ such that
\[
    \mathrm{dist}(\mathcal{P}_{k}(\bm{x}^k),\Omega)
    \leq \kappa\,\mathrm{dist}\big(0,\mathcal{T}(\mathcal{P}_k(\bm{x}^k))\big)
    \leq \kappa\|c_k^{-1}M_k^{-1}\mathcal{Q}_k(\bm{x}^k)\|, \quad \forall\,k\geq0.
\]
This together with $\lambda_{\min}(M_k)\geq\underline{\lambda}>0$ further implies that, for all $k\geq0$,
\begin{equation}\label{ineqdQ1}
\begin{aligned}
\mathrm{dist}_{M_{k}^{-1}}(\mathcal{P}_{k}(\bm{x}^k),\Omega)
&\leq \frac{1}{\sqrt{\underline{\lambda}}}\,\mathrm{dist}(\mathcal{P}_{k}(\bm{x}^k),\Omega)
\leq \frac{\kappa}{c_k\sqrt{\underline{\lambda}}}\,\|M_k^{-1}\mathcal{Q}_k(\bm{x}^k)\|  \\
&\leq \frac{\kappa}{c_k\underline{\lambda}}\,\|M_k^{-1}\mathcal{Q}_k(\bm{x}^k)\|_{M_k}
= \frac{\kappa}{c_k\underline{\lambda}}\,\|\mathcal{Q}_k(\bm{x}^k)\|_{M_k^{-1}}.
\end{aligned}
\end{equation}
Moreover, note that $\mathcal{Q}_k(\Pi_{\Omega}(\bm{x}^k))=0$. Then,
\begin{equation}\label{ineqdQ2}
\begin{aligned}
\|\mathcal{Q}_k(\bm{x}^k)\|_{M_k^{-1}}^2
&= \|\mathcal{Q}_k(\bm{x}^k)-\mathcal{Q}_k(\Pi_{\Omega}(\bm{x}^k))\|_{M_k^{-1}}^2 \\
&\leq \|\bm{x}^k-\Pi_{\Omega}(\bm{x}^k)\|_{M_k^{-1}}^2 - \|\mathcal{P}_k(\bm{x}^k) - \mathcal{P}_k(\Pi_{\Omega}(\bm{x}^k))\|_{M_k^{-1}}^2 \\
&\leq \|\bm{x}^k-\Pi_{\Omega}(\bm{x}^k)\|_{M_k^{-1}}^2 - \|\mathcal{P}_k(\bm{x}^k) - \Pi_{\Omega}(\bm{x}^k)\|_{M_k^{-1}}^2  \\
&\leq \mathrm{dist}_{M_{k}^{-1}}^2(\bm{x}^{k},\Omega) - \mathrm{dist}_{M_{k}^{-1}}^2(\mathcal{P}_k(\bm{x}^k),\Omega),
\end{aligned}
\end{equation}
where the first inequality follows from Proposition \ref{prop-PQ}(c). Combining \eqref{ineqdQ1} and \eqref{ineqdQ2} yields
\begin{equation}\label{ineqdist}
\mathrm{dist}_{M_{k}^{-1}}(\mathcal{P}_{k}(\bm{x}^k),\Omega)
\leq \frac{\kappa}{\sqrt{\kappa^2+\underline{\lambda}^2c_k^2}}
\,\mathrm{dist}_{M_{k}^{-1}}(\bm{x}^{k},\Omega).
\end{equation}

We next show that
\begin{equation}\label{ineqzpkz}
\begin{aligned}
\|\bm{x}^{k+1}-\mathcal{P}_k(\bm{x}^k)\|_{M_{k}^{-1}}
\leq \rho(1-\rho)^{-1}\|\bm{x}^{k+1}-\bm{x}^k\|_{M_k^{-1}}.
\end{aligned}
\end{equation}
First, one can see from the definition of $\mathcal{P}_k$ that $\bm{x}^k\in c_kM_k\mathcal{T}(\mathcal{P}_k(\bm{x}^k))+\mathcal{P}_k(\bm{x}^k)$ for all $k\geq0$, that is, for all $k\geq0$, there exits a $\bm{w}^{k+1}\in\mathcal{T}(\mathcal{P}_k(\bm{x}^k))$ such that $c_kM_k\bm{w}^{k+1}+\mathcal{P}_k(\bm{x}^k)-\bm{x}^k=0$. Then, we see that
\[
    \begin{aligned}
    &\quad \|c_kM_k\bm{d}^{k+1} + \widetilde{\bm{x}}^{k+1} - \bm{x}^k\|^2_{M_k^{-1}} \\
    &= \|c_kM_k\bm{d}^{k+1} + \widetilde{\bm{x}}^{k+1} - \bm{x}^k - \big(c_kM_k\bm{w}^{k+1}+\mathcal{P}_k(\bm{x}^k)-\bm{x}^k\big)\|^2_{M_k^{-1}} \\
    &= \|c_kM_k\bm{d}^{k+1}-c_kM_k\bm{w}^{k+1} + \widetilde{\bm{x}}^{k+1}-\mathcal{P}_k(\bm{x}^k)\|^2_{M_k^{-1}} \\
    &= \|c_kM_k\bm{d}^{k+1}-c_kM_k\bm{w}^{k+1}\|^2_{M_k^{-1}}
    + \|\widetilde{\bm{x}}^{k+1}-\mathcal{P}_k(\bm{x}^k)\|^2_{M_k^{-1}}
    + 2c_k\langle\bm{d}^{k+1}-\bm{w}^{k+1},\,\widetilde{\bm{x}}^{k+1}
    -\mathcal{P}_k(\bm{x}^k)\rangle.
    \end{aligned}
\]
Recall that $\bm{d}^{k+1} \in \mathcal{T}^{\varepsilon_{k+1}}(\widetilde{\bm{x}}^{k+1})$ and hence $\langle\bm{d}^{k+1}-\bm{w}^{k+1},\,\widetilde{\bm{x}}^{k+1}
-\mathcal{P}_k(\bm{x}^k)\rangle\geq-\varepsilon_{k+1}$. Substituting it in the above relation yields
\begin{equation}\label{ineqvw1}
\begin{aligned}
&\quad \|c_kM_k\bm{d}^{k+1}-c_kM_k\bm{w}^{k+1}\|^2_{M_k^{-1}}
+ \|\widetilde{\bm{x}}^{k+1}-\mathcal{P}_k(\bm{x}^k)\|^2_{M_k^{-1}} \\
&\leq \|c_kM_k\bm{d}^{k+1} + \widetilde{\bm{x}}^{k+1} - \bm{x}^k\|^2_{M_k^{-1}}
+ 2c_k\varepsilon_{k+1}
\leq \rho^2\|\widetilde{\bm{x}}^{k+1} - \bm{x}^k\|^2_{M_k^{-1}},
\end{aligned}
\end{equation}
where the last inequality follows from \eqref{VHPE-inexcond}. Moreover, using \eqref{VHPE-inexcond} again, we see that
\[
    \rho\|\widetilde{\bm{x}}^{k+1} - \bm{x}^k\|_{M_k^{-1}}
    \geq \|c_kM_k\bm{d}^{k+1} + \widetilde{\bm{x}}^{k+1} - \bm{x}^k\|_{M_k^{-1}}
    \geq \|\widetilde{\bm{x}}^{k+1} - \bm{x}^k\|_{M_k^{-1}} - \|c_kM_k\bm{d}^{k+1}\|_{M_k^{-1}},
\]
which implies that
\begin{equation}\label{ineqvw2}
\|\widetilde{\bm{x}}^{k+1} - \bm{x}^k\|_{M_k^{-1}}
\leq (1-\rho)^{-1}\|c_kM_k\bm{d}^{k+1}\|_{M_k^{-1}}
= (1-\rho)^{-1}\|\bm{x}^{k+1}-\bm{x}^k\|_{M_k^{-1}}.
\end{equation}
Thus, combining \eqref{ineqvw1} and \eqref{ineqvw2}, one can deduce that
\[
    \|c_kM_k\bm{d}^{k+1}-c_kM_k\bm{w}^{k+1}\|_{M_k^{-1}}
    \leq \rho(1-\rho)^{-1}\|\bm{x}^{k+1}-\bm{x}^k\|_{M_k^{-1}}.
\]
Using this inequality, we further obtain that, for all $k\geq0$,
\[
    \begin{aligned}
    &\|\bm{x}^{k+1}-\mathcal{P}_k(\bm{x}^k)\|_{M_{k}^{-1}}
    = \|\bm{x}^{k}-c_kM_k\bm{d}^{k+1}-\mathcal{P}_k(\bm{x}^k)\|_{M_{k}^{-1}} \\
    &= \|c_kM_k\bm{d}^{k+1}-c_kM_k\bm{w}^{k+1}\|_{M_{k}^{-1}}
    \leq \rho(1-\rho)^{-1}\|\bm{x}^{k+1}-\bm{x}^k\|_{M_k^{-1}},
    \end{aligned}
\]
which proves \eqref{ineqzpkz}.

Now, we see that
\[
    \begin{aligned}
    & \quad  \|\bm{x}^{k+1}-\Pi_{\Omega}(\mathcal{P}_k(\bm{x}^k))\|_{M_k^{-1}}
    ~~\leq ~~\|\bm{x}^{k+1}-\mathcal{P}_k(\bm{x}^k)\|_{M_k^{-1}} + \|\mathcal{P}_k(\bm{x}^k)-\Pi_{\Omega}(\mathcal{P}_k(\bm{x}^k))\|_{M_k^{-1}} \\
    &\stackrel{\eqref{ineqzpkz}}{\leq} \rho(1-\rho)^{-1}\|\bm{x}^{k+1}-\bm{x}^k\|_{M_k^{-1}} + \|\mathcal{P}_k(\bm{x}^k)-\Pi_{\Omega}(\mathcal{P}_k(\bm{x}^k))\|_{M_k^{-1}} \\
    &~~\leq ~~ \rho(1-\rho)^{-1}\|\bm{x}^{k+1}-\Pi_{\Omega}(\mathcal{P}_k(\bm{x}^k))\|_{M_k^{-1}}
    + \rho(1-\rho)^{-1}\|\bm{x}^k-\Pi_{\Omega}(\mathcal{P}_k(\bm{x}^k))\|_{M_k^{-1}} \\
    &\qquad\quad  + \|\mathcal{P}_k(\bm{x}^k)-\Pi_{\Omega}(\mathcal{P}_k(\bm{x}^k))\|_{M_k^{-1}}.
    \end{aligned}
\]
Thus, by rearranging terms in the above relation, we have that
\[
    \begin{aligned}
    &\quad (1-\rho(1-\rho)^{-1})\|\bm{x}^{k+1}-\Pi_{\Omega}(\mathcal{P}_k(\bm{x}^k))\|_{M_k^{-1}} \\
    &\leq \rho(1-\rho)^{-1}\|\bm{x}^k-\Pi_{\Omega}(\mathcal{P}_k(\bm{x}^k))\|_{M_k^{-1}}
    + \|\mathcal{P}_k(\bm{x}^k)-\Pi_{\Omega}(\mathcal{P}_k(\bm{x}^k))\|_{M_k^{-1}} \\
    &\leq \rho(1-\rho)^{-1}\|\bm{x}^k-\mathcal{P}_k(\bm{x}^k)\|_{M_k^{-1}}
    + \rho(1-\rho)^{-1}\|\mathcal{P}_k(\bm{x}^k)-\Pi_{\Omega}(\mathcal{P}_k(\bm{x}^k))\|_{M_k^{-1}} \\
    &\qquad + \|\mathcal{P}_k(\bm{x}^k)-\Pi_{\Omega}(\mathcal{P}_k(\bm{x}^k))\|_{M_k^{-1}} \\
    &= \rho(1-\rho)^{-1}\|\mathcal{Q}_k(\bm{x}^k)\|_{M_k^{-1}}
    + (1+\rho(1-\rho)^{-1})\|\mathcal{P}_k(\bm{x}^k)
    -\Pi_{\Omega}(\mathcal{P}_k(\bm{x}^k))\|_{M_k^{-1}}  \\
    &\leq \rho(1-\rho)^{-1}\mathrm{dist}_{M_{k}^{-1}}(\bm{x}^{k},\Omega)
    + (1+\rho(1-\rho)^{-1})\mathrm{dist}_{M_{k}^{-1}}(\mathcal{P}_{k}(\bm{x}^k),\Omega) \\
    &\leq \left(\rho(1-\rho)^{-1}
    +\frac{(1+\rho(1-\rho)^{-1})\kappa}{\sqrt{\kappa^2+\underline{\lambda}^2c_k^2}}\right)
    \mathrm{dist}_{M_{k}^{-1}}(\bm{x}^{k},\Omega),
    \end{aligned}
\]
where the second last inequality follows from \eqref{ineqdQ2} and the last inequality follows from \eqref{ineqdist}. Now, using this inequality, it holds that, for all $k\geq0$,
\[
    \begin{aligned}
    &\mathrm{dist}_{M_{k+1}^{-1}}(\bm{x}^{k+1},\Omega)
    \leq (1+\eta_k)\mathrm{dist}_{M_{k}^{-1}}(\bm{x}^{k+1},\Omega)
    \leq (1+\eta_k)\|\bm{x}^{k+1}-\Pi_{\Omega}(\mathcal{P}_k(\bm{x}^k))\|_{M_k^{-1}} \\
    &\leq \frac{1+\eta_k}{1-\rho(1-\rho)^{-1}}\left(\rho(1-\rho)^{-1}
    +\frac{(1+\rho(1-\rho)^{-1})\kappa}{\sqrt{\kappa^2+\underline{\lambda}^2c_k^2}}\right)
    \mathrm{dist}_{M_{k}^{-1}}(\bm{x}^{k},\Omega).
    \end{aligned}
\]
It is easy to see that, by taking $\rho$ sufficiently small and $c_k$ sufficiently large, we can make the scalar on the right-hand side of the above relation arbitrarily small and hence less than one. Then, we obtain the desired results and complete the proof.
\end{proof}

\section{Dual-based ADMM-type methods}\label{appendix-admm}

In this section, we present how to apply the popular alternating direction method of multipliers (ADMM, see, e.g. \cite{boyd2011distributed,gabay1976dual}) to the following dual problem of problem \eqref{eq:pmain}:
\begin{equation}\label{eq:dmain-ADMM}
\begin{aligned}
&\min_{W\in \mathbb{R}^{\widetilde m\times\widetilde n}, \Xi\in\mathbb{R}^{m\times n}, \bm{u},\,\bm{\zeta}\in\mathbb{R}^m, \bm{v},\,\bm{\xi}\in\mathbb{R}^n} -\langle S, \,W\rangle - \inner{\bm{\alpha},\bm{u}} - \inner{\bm{\beta},\bm{v}}
+ p^*(-\Xi) + p_r^*(-\bm{\zeta}) + p_c^*(-\bm{\xi})  \\
&\hspace{2.3cm}\mathrm{s.t.} \hspace{2.5cm} \bm{u}\bm{1}_n^\top + \bm{1}_m\bm{v}^\top + A^\top W B^\top + \Xi = C, \quad
 \bm{u} + \bm{\zeta} = 0, \quad
 \bm{v} + \bm{\xi} = 0.
\end{aligned}
\end{equation}
Specifically, given a positive scalar $\sigma>0$, the augmented Lagrangian function associated with \eqref{eq:dmain-ADMM} is given by
\begin{equation*}
\begin{aligned}
&\mathcal{L}_{\sigma}\big(W,\bm{u},\bm{v},\Xi,\bm{\zeta},\bm{\xi},X,\bm{y},\bm{z}\big)
:= -\langle S, \,W\rangle - \inner{\bm{\alpha},\bm{u}}
- \inner{\bm{\beta},\bm{v}} + p^*(-\Xi) + p_r^*(-\bm{\zeta}) + p_c^*(-\bm{\xi}) \\
&\qquad + \big\langle X, \,\bm{u}\bm{1}_n^\top + \bm{1}_m\bm{v}^\top + A^\top W B^\top + \Xi - C\big\rangle
+ \inner{\bm{y}, \bm{u} + \bm{\zeta}}
+ \inner{\bm{z}, \bm{v} + \bm{\xi}}  \\
&\qquad + \frac{\sigma}{2}\big\|\bm{u}\bm{1}_n^\top + \bm{1}_m\bm{v}^\top + A^\top W B^\top + \Xi - C\big\|_F^2
+ \frac{\sigma}{2}\norm{\bm{u} + \bm{\zeta}}^2
+ \frac{\sigma}{2}\norm{\bm{v} + \bm{\xi}}^2.
\end{aligned}
\end{equation*}
Then, the ADMM for solving the dual problem \eqref{eq:dmain-ADMM} can be described in Algorithm~\ref{algo:dADMM}.

\begin{algorithm}[htb!]
\caption{ADMM for solving the dual problem \eqref{eq:dmain-ADMM} (dADMM)}\label{algo:dADMM}
\textbf{Input:} the penalty parameter $\sigma>0$, and the initializations $W^0\in \mathbb{R}^{\widetilde m\times\widetilde n}$, $\Xi^0\in\mathbb{R}^{m\times n}$, $\bm{u}^0,\,\bm{\zeta}^0,\,\bm{y}^0\in\mathbb{R}^m$, $\bm{v}^0,\,\bm{\xi}^0,\,\bm{z}^0\in\mathbb{R}^n$. Set $k = 0$.
	
\While{a termination criterion is not met,}{ 		
\vspace{0.5mm}
\begin{itemize}[leftmargin=.2cm]
\item[] \textbf{Step 1.} Compute
\begin{equation*}
\big(W^{k+1}, \bm{u}^{k+1}, \bm{v}^{k+1}\big) = \arg\min\limits_{W, \bm{u}, \bm{v}}~\mathcal{L}_{\sigma}\big(W,\bm{u},\bm{v},\Xi^k,\bm{\zeta}^k,\bm{\xi}^k,X^k,\bm{y}^k,\bm{z}^k\big).
\end{equation*}

\item[] \textbf{Step 2.} Compute
\begin{equation*}
\big(\Xi^{k+1},\bm{\zeta}^{k+1},\bm{\xi}^{k+1}\big) = \arg\min\limits_{\Xi,\bm{\zeta},\bm{\xi}}~\mathcal{L}_{\sigma}
\big(W^{k+1},\bm{u}^{k+1},\bm{v}^{k+1},\Xi,\bm{\zeta},\bm{\xi},X^k,\bm{y}^k,\bm{z}^k\big).
\end{equation*}

\item[] \textbf{Step 3.} Set
\begin{equation*}
\begin{aligned}
X^{k+1} &= X^k + \gamma\sigma\big(\bm{u}^{k+1}\bm{1}_n^\top + \bm{1}_m(\bm{v}^{k+1})^\top + A^\top W^{k+1} B^\top + \Xi^{k+1} - C\big), \\
\bm{y}^{k+1} &= \bm{y}^{k}+\gamma\sigma\big(\bm{u}^{k+1}+\bm{\zeta}^{k+1}\big), \quad
\bm{z}^{k+1} = \bm{z}^{k} + \gamma\sigma\big(\bm{v}^{k+1} + \bm{\xi}^{k+1}\big),
\end{aligned}
\end{equation*}	
where $\gamma\in\left(0, \frac{1+\sqrt{5}}{2}\right)$ is the dual step-size that is typically set to $1.618$. \vspace{-3mm}
\end{itemize}
}
\textbf{Output:} $\big(W^k,\bm{u}^k,\bm{v}^k,\Xi^k,\bm{\zeta}^k,\bm{\xi}^k,X^k,\bm{y}^k,\bm{z}^k\big)$.
\end{algorithm}

Note that, in \textbf{Step 1} of Algorithm \ref{algo:dADMM}, a linear system of size $(\tilde{m}\tilde{m} + m + n)\times (\tilde{m}\tilde{m} + m + n)$ has to be solved in order to update the dual variables $(W, \bm{u}, \bm{v})$. Thus, when the problem size is large, the computation of this step would be very expensive. To bypass such an issue, we also consider applying a symmetric Gauss-Seidel based ADMM (SGSADMM, see, e.g. \cite{chen2021equivalence,chen2017efficient}),
which is described in Algorithm \ref{algo:dSGSADMM}. Moreover, as discussed in \cite{chen2021equivalence}, a larger step size $\gamma$ is also allowed in SGSADMM, which often leads to better numerical performance.

\begin{algorithm}[htb!]
\caption{SGSADMM for solving the dual problem \eqref{eq:dmain-ADMM} (dSGSADMM)}\label{algo:dSGSADMM}
\textbf{Input:} the penalty parameter $\sigma>0$, and the initializations $W^0\in \mathbb{R}^{\widetilde m\times\widetilde n}$, $\Xi^0\in\mathbb{R}^{m\times n}$, $\bm{u}^0,\,\bm{\zeta}^0,\,\bm{y}^0\in\mathbb{R}^m$, $\bm{v}^0,\,\bm{\xi}^0,\,\bm{z}^0\in\mathbb{R}^n$. Set $k = 0$.
	
\While{a termination criterion is not met,}{ 		
\vspace{0.5mm}
\begin{itemize}[leftmargin=.2cm]
\item[] \textbf{Step 1.} Compute
\[
\begin{aligned}
    \widetilde{W}^{k+1} = &~\arg\min\limits_{W}~\mathcal{L}_{\sigma}\big(W,\bm{u}^k,\bm{v}^k,\Xi^k,\bm{\zeta}^k,\bm{\xi}^k,X^k,\bm{y}^k,\bm{z}^k\big), \\
    \tilde{\bm{u}}^{k+1} = &~\arg\min\limits_{\bm{u}}~\mathcal{L}_{\sigma}\big(\widetilde{W}^{k+1},\bm{u},\bm{v}^k,\Xi^k,\bm{\zeta}^k,\bm{\xi}^k,X^k,\bm{y}^k,\bm{z}^k\big), \\
    \tilde{\bm{v}}^{k+1} = &~\arg\min\limits_{\bm{v}}~\mathcal{L}_{\sigma}\big(\widetilde{W}^{k+1}, \tilde{\bm{u}}^{k+1},\bm{v},\Xi^k,\bm{\zeta}^k,\bm{\xi}^k,X^k,\bm{y}^k,\bm{z}^k\big).
\end{aligned}
\]

\item[] \textbf{Step 2.} Compute
\begin{equation*}
\big(\Xi^{k+1},\bm{\zeta}^{k+1},\bm{\xi}^{k+1}\big) = \arg\min\limits_{\Xi,\bm{\zeta},\bm{\xi}}~\mathcal{L}_{\sigma}
\big(\widetilde{W}^{k+1},\tilde{\bm{u}}^{k+1},\tilde{\bm{v}}^{k+1},\Xi,\bm{\zeta},\bm{\xi},X^k,\bm{y}^k,\bm{z}^k\big).
\end{equation*}

\item[] \textbf{Step 3.}
\[
\begin{aligned}
    \bm{v}^{k+1} = &~\arg\min\limits_{\bm{v}}~\mathcal{L}_{\sigma}\big(\widetilde{W}^{k+1}, \tilde{\bm{u}}^{k+1},\bm{v},\Xi^{k+1},\bm{\zeta}^{k+1},\bm{\xi}^{k+1},X^k,\bm{y}^k,\bm{z}^k\big), \\
    \bm{u}^{k+1} = &~\arg\min\limits_{\bm{u}}~\mathcal{L}_{\sigma}\big(\widetilde{W}^{k+1},\bm{u},\bm{v}^{k+1},\Xi^{k+1},\bm{\zeta}^{k+1},\bm{\xi}^{k+1},X^k,\bm{y}^k,\bm{z}^k\big), \\
    W^{k+1} = &~~\arg\min\limits_{W}~\mathcal{L}_{\sigma}\big(W,\bm{u}^{k+1},\bm{v}^{k+1},\Xi^{k+1},\bm{\zeta}^{k+1},\bm{\xi}^{k+1},X^k,\bm{y}^k,\bm{z}^k\big).
\end{aligned}
\]

\item[] \textbf{Step 4.} Set
\begin{equation*}
\begin{aligned}
X^{k+1} &= X^k + \gamma\sigma\big(\bm{u}^{k+1}\bm{1}_n^\top + \bm{1}_m(\bm{v}^{k+1})^\top + A^\top W^{k+1} B^\top + \Xi^{k+1} - C\big), \\
\bm{y}^{k+1} &= \bm{y}^{k}+\gamma\sigma\big(\bm{u}^{k+1}+\bm{\zeta}^{k+1}\big), \quad
\bm{z}^{k+1} = \bm{z}^{k} + \gamma\sigma\big(\bm{v}^{k+1} + \bm{\xi}^{k+1}\big),
\end{aligned}
\end{equation*}	
where $\gamma\in(0, 2)$ is the dual step-size that is typically set to 1.95. \vspace{-3mm}
\end{itemize}
}
\textbf{Output:} $\big(W^k,\bm{u}^k,\bm{v}^k,\Xi^k,\bm{\zeta}^k,\bm{\xi}^k,X^k,\bm{y}^k,\bm{z}^k\big)$.
\end{algorithm}

\section{Second-order cone programming reformulation}\label{appendix-socp}

In this section, we present an explicit second-order cone programming (SOCP) reformulation of problem \eqref{eq-regOTpro}. To this end, we first characterize the constraint set $\mathcal{T}$ as
\begin{equation*}
\begin{aligned}
\mathcal{T} = &~ \left\{ X\in \mathbb{R}^{m\times n}~:~ AXB = S, \; \bm{\alpha} - X\bm{1}_n \in \mathcal{K}_r,\; \bm{\beta} - X^\top \bm{1}_m \in \mathcal{K}_c,\; X\geq 0\right\} \\
= &~ \left\{X\in \mathbb{R}^{m\times n}~:~\bm{b}_l\leq \mathcal{A}(X) \leq \bm{b}_u,\; X\geq 0\right\},
\end{aligned}
\end{equation*}
where $\mathcal{A}: \mathbb{R}^{m\times n}\to \mathbb{R}^{\tilde{m}\tilde{n} + m + n}$ is a linear mapping and $\bm{b}_l,\,\bm{b}_u\in \mathbb{R}^{\tilde{m}\tilde{n} + m + n}$ are two vectors that can be constructed easily from the problem data. Then, we introduce some slack variables $r,s\in \mathbb{R}$ and $\bm{t}\in \mathbb{R}^{|\mathcal{G}|}$ which are used to majorize the objective function. Specifically, we shall replace the term $\frac{1}{2}\lambda_2\norm{X}_F^2$ with $\lambda_2s$ together with the constraints $\norm{X}_F^2 \leq 2rs,\; r = 1$, and the term $\lambda_1\sum_{G\in \mathcal{G}}\omega_G\norm{\bm{x}_G}$ with $\lambda_1\inner{\bm{\omega},\bm{t}}$ together with the constraints $\norm{\bm{x}_G} \leq t_G$ for all $G\in \mathcal{G}$, where $\bm{\omega}\in \mathbb{R}^{|\mathcal{G}|}$ is the vector storing all weights of the partition $\mathcal{G}$. Let $d>0$ be any positive integer, we denote the second-order cone in $\mathbb{R}^{d+1}$ as $\mathcal{Q}^{d+1}:= \left\{(x_0, \bm{x}_t)\in \mathbb{R}^{d+1}\,:\, x_0\geq\norm{\bm{x}_t}\right\}$ and the rotated second-order cone in $\mathbb{R}^{d+2}$ as
\begin{equation*}
\mathcal{Q}^{d+2}_r
:= \left\{(x_1, x_2, \bm{z}) \in \mathbb{R}^{d+2}\,:\, 2x_1x_2 \geq \norm{\bm{z}}^2,\; x_1\geq 0,\;x_2\geq 0\right\}.
\end{equation*}
Using the above notation, we see that  \eqref{eq-regOTpro} can be reformulated as the following SOCP problem:
\begin{equation*}
\begin{aligned}
\min_{X\in\mathbb{R}^{m\times n}, \,r \in
\mathbb{R}, \,s\in \mathbb{R}, \,\bm{t}\in \mathbb{R}^{|\mathcal{G}|}}
& ~~ \inner{C, X} + \lambda_1\inner{\bm{\omega},\bm{t}} + \lambda_2s \\
\mathrm{s.t.} \hspace{1.55cm}
&~~~ \bm{b}_l\leq
\mathcal{A}(X) \leq \bm{b}_u,\; r=1,\;  X\geq 0, \;r \geq0,\; s\geq 0,\; \bm{t} \geq 0,\\
&~~~ (r,s,\mathrm{vec}(X))\in
\mathcal{Q}_r^{mn+2},\; (t_G,\bm{x}_G) \in \mathcal{Q}^{|G|+1},\; \forall G\in \mathcal{G}.
\end{aligned}
\end{equation*}

\bibliographystyle{plain}
\bibliography{Ref_regOT}

\end{document}